\documentclass[12pt]{article}
\usepackage[margin=2.5cm]{geometry}
\usepackage[T1]{fontenc}
\usepackage{lmodern}
\usepackage{enumitem}
\usepackage{microtype}
\usepackage{graphicx}
\usepackage{xcolor}
\usepackage{tikz}
\usetikzlibrary{calc,fadings,positioning,decorations.pathmorphing,intersections,
  patterns,decorations.markings,svg.path,graphs,graphs.standard,matrix,
  arrows.meta,arrows,mindmap,shapes.geometric,calligraphy,decorations.pathreplacing}
\usepackage{amsfonts,amsmath,amssymb,amsthm,mathtools,mathrsfs,dsfont,bm}
\usepackage{float}
\usepackage{array}
\usepackage{multirow}
\usepackage{booktabs}
\usepackage{caption}
\usepackage{subcaption}
\usepackage{authblk}
\usepackage{chngcntr}

\usepackage{hyperref}
\hypersetup{
    pdftitle={Fractional domatic number and minimum degree},
    pdfauthor={Quentin Chuet, Hugo Demaret, Hoang La, François Pirot},
    colorlinks,
    linkcolor={blue},
    citecolor={red!60!black},
    urlcolor={blue!80!black}
}
\usepackage{aliascnt}   
\usepackage{cleveref}   

\colorlet{lgray}{gray!60}
\colorlet{lteal}{teal!70}
\colorlet{llteal}{teal!50}
\colorlet{lllteal}{teal!30}
\definecolor{borderteal}{HTML}{284B4B}

\definecolor{color1}{HTML}{28D025}
\definecolor{color2}{HTML}{FFCA0E}
\definecolor{color3}{HTML}{EE8100}
\definecolor{color4}{HTML}{FF3939}
\definecolor{color5}{HTML}{FF5CB8}
\definecolor{color6}{HTML}{8D47FF}
\definecolor{color7}{HTML}{00B7BF}
\definecolor{color8}{HTML}{AABCC0}

\makeatletter
\tikzset{
        hatch distance/.store in=\hatchdistance,
        hatch distance=5pt,
        hatch thickness/.store in=\hatchthickness,
        hatch thickness=5pt
        }
\pgfdeclarepatternformonly[\hatchdistance,\hatchthickness]{north east hatch teal}
    {\pgfqpoint{-1pt}{-1pt}}
    {\pgfqpoint{\hatchdistance}{\hatchdistance}}
    {\pgfpoint{\hatchdistance-1pt}{\hatchdistance-1pt}}%
    {
        \pgfsetcolor{\tikz@pattern@color}
        \pgfsetlinewidth{\hatchthickness}
        \pgfpathmoveto{\pgfqpoint{0pt}{0pt}}
        \pgfpathlineto{\pgfqpoint{\hatchdistance}{\hatchdistance}}
        \pgfusepath{stroke}
    }
\makeatletter
\tikzset{
        hatch distance/.store in=\hatchdistance,
        hatch distance=5pt,
        hatch thickness/.store in=\hatchthickness,
        hatch thickness=5pt
        }
\pgfdeclarepatternformonly[\hatchdistance,\hatchthickness]{north east hatch gray}
    {\pgfqpoint{-1pt}{-1pt}}
    {\pgfqpoint{\hatchdistance}{\hatchdistance}}
    {\pgfpoint{\hatchdistance-1pt}{\hatchdistance-1pt}}%
    {
        \pgfsetcolor{\tikz@pattern@color}
        \pgfsetlinewidth{\hatchthickness}
        \pgfpathmoveto{\pgfqpoint{0pt}{0pt}}
        \pgfpathlineto{\pgfqpoint{\hatchdistance}{\hatchdistance}}
        \pgfusepath{stroke}
    }

\tikzset{
  node split radius/.initial=1,
  node split color 1/.initial=teal,
  node split color 2/.initial=red,
  node split color 3/.initial=blue,
  node split half/.style={node split={#1,#1+180}},
  node split/.style args={#1,#2}{
    path picture={
      \tikzset{
        x=($(path picture bounding box.east)-(path picture bounding box.center)$),
        y=($(path picture bounding box.north)-(path picture bounding box.center)$),
        radius=\pgfkeysvalueof{/tikz/node split radius}}
      \foreach \ang[count=\iAng, remember=\ang as \prevAng (initially #1)] in {#2,360+#1}
        \fill[line join=round, draw, fill=\pgfkeysvalueof{/tikz/node split color \iAng}]
          (path picture bounding box.center)
          --++(\prevAng:\pgfkeysvalueof{/tikz/node split radius})
          arc[start angle=\prevAng, end angle=\ang] --cycle;
} } }

\newcommand{\deflabel}[2]{\expandafter\def\csname label{#1}\endcsname{#2}}
\newcommand{\lab}[1]{\csname label{#1}\endcsname}

\DeclareMathOperator{\fdom}{FDOM}
\DeclareMathOperator{\dom}{DOM}
\DeclareMathOperator{\mad}{mad}

\DeclareMathOperator{\et}{and}

\newcommand{\eps}{\varepsilon}
\newcommand{\bA}{\mathbf{A}}
\newcommand{\B}{\mathcal{B}}
\newcommand{\bB}{\mathbf{B}}
\newcommand{\cF}{\mathcal{F}}

\newcommand{\bi}{{\bm{\iota}}}
\newcommand{\cS}{\mathcal{S}}
\newcommand{\cD}{\mathcal{D}}
\newcommand{\cU}{\mathcal{U}}
\newcommand{\bD}{\mathbf{D}}

\newcommand{\R}{\mathbb{R}}

\newcommand{\sG}{\mathscr{G}}

\newcommand{\bX}{\mathbf{X}}
\newcommand{\bY}{\mathbf{Y}}

\newcommand{\card}[1]{\lvert #1\rvert}
\newcommand{\pr}[1]{\mathbb{P}\left[#1\right]}
\newcommand{\sst}[2]{\left\{#1 : #2\right\}}
\newcommand{\pth}[1]{\left(#1\right)}
\newcommand{\ceil}[1]{\left \lceil #1\right \rceil}
\newcommand{\restrict}[2]{#1_{\left| #2 \right.}}
\newcommand{\plangirth}[1]{\textsc{Planar}[\delta=2, {\rm girth}\ge #1]}
\newcommand{\link}{\text{--}}

\renewcommand{\phi}{\varphi}

\newcommand{\randomgraph}[2][n]{\mathbb{G}\pth{#1,#2}}
\newcommand{\esp}[1]			{\mathbb{E}\left[#1\right]}

\theoremstyle{plain}
\newtheorem{thm}{Theorem}

\newaliascnt{prop}{thm}
\newtheorem{prop}[prop]{Proposition}
\aliascntresetthe{prop}

\newaliascnt{cor}{thm}
\newtheorem{cor}[cor]{Corollary}
\aliascntresetthe{cor}

\newaliascnt{fact}{thm}
\newtheorem{fact}[fact]{Fact}
\aliascntresetthe{fact}

\newaliascnt{conj}{thm}
\newtheorem{conj}[conj]{Conjecture}
\aliascntresetthe{conj}

\newaliascnt{lemma}{thm}
\newtheorem{lemma}[lemma]{Lemma}
\aliascntresetthe{lemma}

\newtheorem{claim}{Claim}[thm]      
\newtheorem{problem}{Problem}       

\theoremstyle{definition}
\newaliascnt{defi}{thm}
\newtheorem{defi}[defi]{Definition}
\aliascntresetthe{defi}

\newtheoremstyle{case}{}{}{}{}{}{:}{ }{}
\theoremstyle{case}

\counterwithin*{case}{claim}

\theoremstyle{remark}
\newtheorem{rk}{Remark}

\crefname{thm}{Theorem}{Theorems}
\crefname{prop}{Proposition}{Propositions}
\crefname{cor}{Corollary}{Corollaries}
\crefname{fact}{Fact}{Facts}
\crefname{conj}{Conjecture}{Conjectures}
\crefname{lemma}{Lemma}{Lemmas}
\crefname{claim}{Claim}{Claims}
\crefname{problem}{Problem}{Problems}
\crefname{defi}{Definition}{Definitions}
\crefname{case}{Case}{Cases}
\crefname{rk}{Remark}{Remarks}

\newenvironment{proofclaim}[1][\textit{Proof of claim.}\hspace{0.066cm}]%
  {\noindent #1}{\strut\hfill$\lozenge$\vspace{2ex}}

\title{Fractional domatic number and minimum degree}
\author[1]{Quentin Chuet}
\author[2]{Hugo Demaret}
\author[3]{Hoang La}
\author[3]{François Pirot}
\affil[1]{LIRMM, Univ Montpellier, CNRS, Univ Montpellier Paul-Valéry, Montpellier, France}
\affil[2]{Université Caen Normandie, ENSICAEN, CNRS, Normandie Univ, GREYC UMR 6072, F-14000 Caen, France}
\affil[3]{Université Paris-Saclay, CNRS, Laboratoire Interdisciplinaire des Sciences du Numérique, 91405, Orsay, France}

\begin{document}
\maketitle

\begin{abstract}
    The domatic number of a graph $G$ is the maximum number of pairwise disjoint dominating sets of $G$. We are interested in the LP-relaxation of this parameter, which is called the fractional domatic number of $G$. We study its extremal value in the class of graphs of minimum degree $d$. The fractional domatic number of a graph of minimum degree $d$ is always at most $d+1$, and at least $(1-o(1))\, d/\ln d$ as $d\to \infty$. This is asymptotically tight even within the class of split graphs.
    Our main result concerns the case $d=2$; we show that, excluding $8$ exceptional graphs, the fractional domatic number of every connected graph of minimum degree (at least) $2$ is at least $5/2$. We also show that this bound cannot be improved if only finitely many graphs are excluded, even for bipartite graphs of arbitrarily large girth. This proves in a stronger sense a conjecture by Gadouleau, Harms, Mertzios, and Zamaraev (2024). This also extends and generalises results from McCuaig and Shepherd (1989), from Fujita, Kameda, and Yamashita (2000), and from Abbas, Egerstedt, Liu, Thomas, and Whalen (2016).
    Finally, we show that graphs of sub-exponential expansion and sufficiently large girth $g$ have fractional domatic number $3 - o_g(1)$.
\end{abstract}

\section{Introduction}
Given a graph $G$, a \emph{dominating set} of $G$ is a set $X\subseteq V(G)$ such that $N[X]=V(G)$. Dominating sets are often used to model monitoring problems in networks. A classical problem in Graph Theory is to find the minimum size $\gamma(G)$ of a dominating set of $G$, called the \emph{domination number} of $G$~\cite{McSh89,MaTa96}. A possible approach to this problem is to study a stronger parameter, the \emph{domatic number} of $G$, denoted $\dom(G)$, which is the 
maximum number of pairwise disjoint dominating sets of $G$. Since dominating sets are stable through vertex addition, $\dom(G)$ can equivalently be defined as the maximum size of a partition of $V(G)$ into dominating sets. 
Such a partition is usually called a \emph{domatic partition of $G$}, and can also be captured by the notion of \emph{dominating $k$-colouring} of $G$, that is a mapping $\phi\colon V(G) \to [k]$ such that $\phi(N[v])=[k]$ for every vertex $v\in V(G)$ --- in other words, every closed neighbourhood spans all the colours in $\phi$. Then $\dom(G)$ is the maximum $k$ such that $G$ has a dominating $k$-colouring. 
A straightforward application of the Pigeonhole Principle yields that the minimum colour class in a dominating $k$-colouring of $G$ has size at most $|V(G)|/k$, hence $\gamma(G) \le \card{V(G)}/\dom(G)$. As a consequence, any lower bound on $\dom(G)$ yields an upper bound on $\gamma(G)$.
For instance, Matheson and Tarjan \cite{MaTa96} proved that every $n$-vertex triangulated disc (i.e. a $2$-connected planar graph, all internal faces of which are triangles) has domination number at most $n/3$, by showing that its domatic number is at least $3$.
A parallel can be made with independent sets and the chromatic number of graphs. Remarkably, there is no known proof for the fact that every $n$-vertex planar graph contains an independent set of size at least $n/4$ that does not rely on the \textsc{$4$-Colour Theorem}.

We are interested in the LP-relaxation of the domatic number, called the \emph{fractional domatic number}, which provides a tighter bound on the domination number. This parameter was first formally introduced in \cite{suomela2006locality}, although some related notions already appeared in \cite{FKY00}. It has many equivalent definitions; we mention some of them hereafter.

\paragraph{Linear Programming}
We denote $\cD(G)$ the collection of dominating sets of a given graph $G$. Then the fractional domatic number of $G$, denoted $\fdom(G)$, is the solution to the following linear program:

\begin{equation}
    \label{eq:LP-fdom}
    \begin{aligned}
    \mbox{Maximise}& \sum_{D \in \mathcal{D}(G)} x_D,\\
    \mbox{Subject to}& \left\{\begin{array}{rl}\displaystyle
        \sum_{\substack{D \in \cD(G)\\ v\in D}} x_D \le 1& \mbox{for all $v \in V(G)$;}\\
        x_D \ge 0 & \mbox{for all $D\in \cD(G)$.}
    \end{array} \right.
    \end{aligned}
\end{equation}

\paragraph{Probability distribution}
If $(x_D)_{D\in \cD(G)}$ yields a solution to \eqref{eq:LP-fdom}, then observe that a renormalisation of the weights $x_D$ yields a probability distribution over $\cD(G)$, in such a way that a random dominating set $\bD$ drawn according to that distribution satisfies $\pr{v\in \bD}\le \frac{1}{\fdom(G)}$ for every vertex $v\in V(G)$. 
So the fractional domatic number of $G$ can alternatively be defined as 
\begin{equation}
    \label{eq:fdom-proba}
    \fdom(G) = \max \sst{\frac{1}{p}}{\begin{array}{c}
        \mbox{there is a random dominating set $\bD$ of $G$ such that}\\
        \pr{v\in \bD}\le p \mbox{ for every vertex $v\in V(G)$.}
    \end{array}}.
\end{equation}

\paragraph{Dominating sets with bounded overlaps}
The theory of Linear Programming tells us that the solution to \eqref{eq:LP-fdom} is a rational number. So there is an integer $q$ such that $q\cdot x_D$ is an integer for every $D\in \cD(G)$. We construct a multiset $\cF$ of dominating sets of $G$ by adding $q\cdot x_D$ copies of $D$ to $\cF$ for every $D\in \cD(G)$. By doing so, every vertex $v\in V(G)$ appears in at most $q$ dominating sets of $\cF$; we call that number of occurrences the \emph{multiplicity of $v$ in $\cF$}.
Letting $m(\cF)$ be the maximum multiplicity of a vertex in $\cF$, we obtain that 
\begin{equation}
\label{eq:fdom-family}
    \fdom(G) = \max\, \frac{\cF}{m(\cF)},
\end{equation}
where the maximum is taken over every multiset $\cF$ of dominating sets of $G$. 

\paragraph{Dominating $(p:q)$-colourings}
If a given vertex $v\in V(G)$ appears in fewer than $m(\cF)$ dominating sets of $\cF$, we may add it to some extra ones so that its multiplicity in $\cF$ is exactly $m(\cF)$. This ensures that the maximum in \eqref{eq:fdom-family} is attained by a multiset $\cF$ in which every vertex has equal multiplicity. We represent this with a \emph{dominating $(p:q)$-colouring} of $G$, that is a mapping $\phi\colon V(G) \to \binom{[p]}{q}$ such that $\bigcup_{u\in N[v]} \phi(u) = [p]$ for every vertex $v\in V(G)$ --- in other words, every closed neighbourhood spans all the colours in $\phi$. This is illustrated in  \Cref{fig:introduction:fractionaldomatic:graphs}. When $p$ and $q$ are not explicit, we say that $\phi$ is a \emph{fractional dominating colouring} of $G$. Then we have
\begin{equation}
    \label{eq:fdom-colouring}
    \fdom(G) = \max \sst{\frac{p}{q}}{\mbox{there is a dominating $(p:q)$-colouring of $G$.}}.
\end{equation}

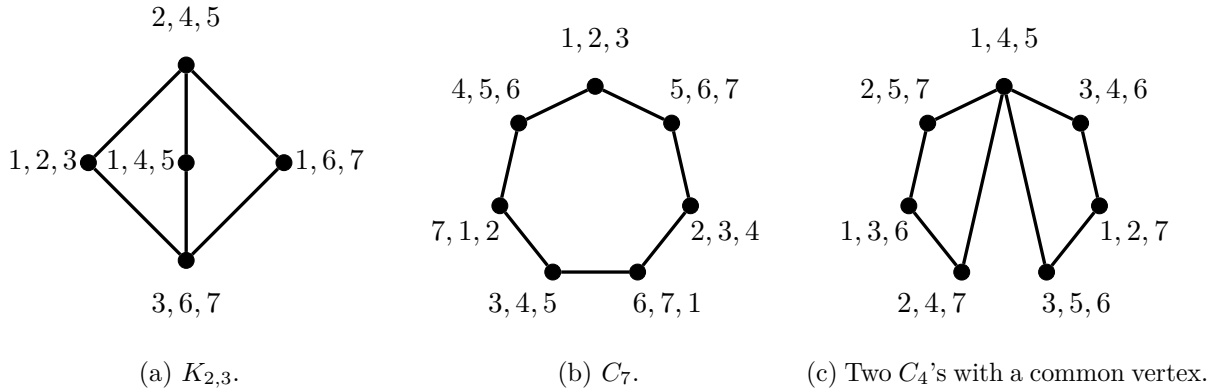
\begin{figure}[!htbp]
\centering
\begin{subfigure}[b]{0.3\textwidth}
    \centering
    \begin{tikzpicture}[baseline=0pt, every node/.style={shape=circle,draw=black,fill,inner sep=0pt,minimum size=6pt,font=\small}, every edge/.style={line width=1.25pt,draw,black}]
            \def\R{1.3}
            \coordinate[label=right:{$1,6,7$}, circle] (A) at (0 : \R);
            \coordinate[ label={$2,4,5$}, circle] (B) at (90 : \R);
            \coordinate[label=left:{$1,2,3$}, circle] (C) at (180 : \R);
            \coordinate[label=below:{$3,6,7$}, circle] (D) at (270 : \R);
            \coordinate[label=left:{$1,4,5$}, circle] (O) at (0, 0);
            \draw (A) edge (B);
            \draw (B) edge (C);
            \draw (C) edge (D);
            \draw (D) edge (A);
            \draw (D) edge (O);
            \draw (O) edge (B);
\end{tikzpicture}
    \caption{$K_{2,3}$.}
    \label{fig:introduction:K23}
\end{subfigure}
\begin{subfigure}[b]{0.3\textwidth}
    \centering
\begin{tikzpicture}[baseline=0pt, every node/.style={shape=circle,draw=black,fill,inner sep=0pt,minimum size=6pt}, every edge/.style={line width=1.25pt,draw,black}]

    \def\n{7}
    \def\R{1.3} 
    \deflabel{1}{$1,2,3$}
    \deflabel{2}{$4,5,6$}
    \deflabel{3}{$7,1,2$}
    \deflabel{4}{$3,4,5$}
    \deflabel{5}{$6,7,1$}
    \deflabel{6}{$2,3,4$}
    \deflabel{7}{$5,6,7$}

    \foreach \i in {1,...,\n} {
        \pgfmathsetmacro{\angle}{90+360/\n*(\i-1)}
        
        \coordinate[label = \angle : \lab{\i}, circle] (x\i) at ({\angle}:\R);
    }
     \foreach \i in {1,...,\n} {
        \pgfmathtruncatemacro{\next}{mod(\i,\n)+1}
        \draw (x\i) edge (x\next);
    }
    \end{tikzpicture}
    \caption{$C_7$.}
    \label{fig:introduction:C7}
\end{subfigure}
\begin{subfigure}[b]{0.38\textwidth}
    \centering
    \begin{tikzpicture}[every node/.style={shape=circle,draw=black,fill,inner sep=0pt,minimum size=6pt}, every edge/.style={line width=1.25pt,draw,black}]

    \def\n{7}
    \def\R{1.3}
    \deflabel{1}{$3,5,6$}
    \deflabel{2}{$1,2,7$}
    \deflabel{3}{$3,4,6$}
    \deflabel{4}{$1,4,5$}
    \deflabel{5}{$2,5,7$}
    \deflabel{6}{$1,3,6$}
    \deflabel{7}{$2,4,7$}
    
    \foreach \i in {1,...,\n} {
        \pgfmathsetmacro{\angle}{90+360/\n*(\i-4)}    
        \coordinate[label = \angle : \lab{\i}, circle] (x\i) at ({\angle}:\R);
    }
    \foreach \i in {1,...,6} {
        \pgfmathtruncatemacro{\next}{mod(\i,\n)+1}
        \draw (x\i) edge (x\next);
    }

    \draw (x1) edge (x4);
    \draw (x4) edge (x7);
                
\end{tikzpicture}
    \caption{Two $C_4$'s with a common vertex.}
    \label{fig:introduction:2C4}
\end{subfigure}
\caption{Examples of dominating $(7:3)$-colourings. Each number represents a dominating set.}
\label{fig:introduction:fractionaldomatic:graphs}
\end{figure}

\subsection*{Domatic number and chromatic number}
The study of the domatic number is analogous to that of the chromatic number.
The chromatic number of a graph $G$, denoted $\chi(G)$, is the minimum size of a partition of $V(G)$ into independent sets. The fractional chromatic number $\chi_f(G)$ corresponds to the LP-relaxation of the chromatic number. As for $\fdom(G)$, there are many equivalent definitions of $\chi_f(G)$; let us mention that related to colourings. Given integers $p\ge q$, a \emph{proper $(p:q)$-colouring} of $G$ is a mapping $\phi\colon V(G) \to \binom{[p]}{q}$ such that $\phi(u)\cap \phi(v) = \emptyset$ for every edge $uv\in E(G)$. When $p$ and $q$ are not explicit, we say that $\phi$ is a \emph{fractional proper colouring} of $G$.
Then $\chi_f(G)$ is the minimum $p/q$ such that there is a proper $(p:q)$-colouring of $G$. 
Said otherwise, each vertex is assigned a set of $q$ colours among $p$ different ones in total, and two adjacent vertices must have no colour in common. For example, \Cref{fig:introduction:C7} is also a fractional proper colouring of $C_7$, since each colour class forms an independent set --- this is not the case in \Cref{fig:introduction:K23}.

There are many similarities between the chromatic number and the domatic number, as well as between their fractional counterparts. Typically, we have the following well-known (and relatively easy to derive) sequence of inequalities for the chromatic parameters of a graph $G$ on $n$ vertices:

\begin{equation}
    \label{eq:chromatic-inequalities}
   \max \left\{\omega(G) , \frac{n}{\alpha(G)}\right\} \le \chi_f(G) \le \chi(G) \le \Delta(G)+1
\end{equation}
where $\omega(G)$, $\alpha(G)$, and $\Delta(G)$ denote the clique number, the independence number, and the maximum degree of $G$ respectively.

Interestingly, there is a sequence of inequalities for the domatic parameters of a graph $G$ on $n$ vertices mirroring \eqref{eq:chromatic-inequalities}.
\begin{equation}
    \label{eq:domatic-inequalities}
    (1-o(1))\frac{\delta(G)}{\ln \Delta(G)} \le \dom(G) \le \fdom(G) \le \min \left\{\delta(G)+1, \frac{n}{\gamma(G)}\right\}.
\end{equation}
The left-most lower bound is from \cite{FHKS02}, and the right-most upper bounds can be found in \cite{GoHe20}.
Let us provide some insight on the similarities between \eqref{eq:chromatic-inequalities} and \eqref{eq:domatic-inequalities}.

The bound $\chi(G)\le \Delta(G)+1$ comes from a simple algorithm that extends a partial proper colouring of a graph greedily, introducing a new colour each time it is needed. This is possible because proper colourings are hereditary: a proper colouring of a graph induces a proper colouring of any of its subgraphs. This is not the case for dominating colourings, so there is no similar bound in the sequence of inequalities associated with $\dom(G)$. Instead, lower bounds on $\dom(G)$ that depend on $\delta(G)$ and $\Delta(G)$ can be derived with the probabilistic method \cite{FHKS02}.

A proper colouring of $G$ is a partition of $V(G)$ into independent sets, all of which have size at most $\alpha(G)$. From that observation, we infer the bound $\chi(G) \ge n/\alpha(G)$. Similarly, a dominating colouring of $G$ is a partition of $V(G)$ into dominating sets, all of which have size at least $\gamma(G)$. Hence, $\gamma(G)$ is to $\dom(G)$ what $\alpha(G)$ is to $\chi(G)$, and we derive the bound $\dom(G) \le n/\gamma(G)$. 

We now delve into the relationship between $\delta(G)$ and $\dom(G)$ that has similarities and differences with that between $\chi(G)$ and $\omega(G)$.
In a proper colouring, every colour must appear at most once in any clique. Similarly, in a dominating colouring, every colour must appear at least once in any closed neighbourhood. This explains why the relation between $\delta(G)+1$ and $\dom(G)$ shares similarities with that between $\omega(G)$ and $\chi(G)$. 
The latter has been extensively studied after it has been observed that there exist triangle-free graphs $G$ (with $\omega(G)=2$) with arbitrarily large chromatic number (and even arbitrarily small independence ratio, see~\cite{Bol81}).
This motivated the study of classes of graphs where, on the contrary, the chromatic number is bounded as a function of the clique number. For a given class of graph $\mathcal{G}$, we say that $\mathcal{G}$ is $\chi$-bounded --- respectively $\chi_f$-bounded --- if there exists a function $f$ such that $\chi(G)\le f(\omega(G))$ --- respectively $\chi_f(G)\le f(\omega(G))$ --- for every graph $G\in\mathcal{G}$. 
As previously discussed, the class of all graphs is neither $\chi$-bounded, nor $\chi_f$-bounded. 
The studies of simple sufficient conditions for and obstructions to $\chi$-boundedness have been an extensive research area of Chromatic Graph Theory; we refer the reader to a survey by Scott and Seymour~\cite{ScottSeymour20} to learn more about this. 

The same notion can be defined for $\dom$ and $\fdom$ thanks to their relation to minimum degree. More formally, we say that a class $\mathcal{G}$ of graphs is $\dom$-bounded --- respectively $\fdom$-bounded --- if and only if there exists a function $f$ such that $\delta(G)\le f(\dom(G))$  --- respectively $\delta(G)\le f(\fdom(G))$ --- for every graph $G\in \mathcal{G}$.

While the class of all graphs is not $\dom$-bounded due to \cite{Zel83}, \Cref{prop:asymptotic-bound} implies that it is $\fdom$-bounded.

\subsection*{Extremal value of the fractional domatic number}


We first provide the following bounds on $\fdom(G)$ with respect to the minimum degree of $G$. The upper bound already appears in \cite{GoHe20}; we provide an alternative proof in \Cref{prop:d+1} as an illustration of our method to derive upper bounds on $\fdom(G)$.
The lower bound can be deduced from the probabilistic definition of the fractional domatic number together with the proof of \cite[Theorem 1.2.2]{AlSp90}; we include the proof for completeness.

\begin{prop}
    
    \label{prop:asymptotic-bound}
    For every graph $G$ of minimum degree $d$, 
\[
\frac{d+1}{1+\ln(d+1)} \le \fdom(G) \le d+1. 
\]

\end{prop}

\begin{proof}[Proof of the lower bound]
    The following proof is similar to that of \cite[Theorem 1.2.2]{AlSp90}.
    Let $G$ be a graph of minimum degree $d \ge 0$. We construct a random set $\mathbf{X}$ by taking each vertex $v \in V(G)$ independently with probability $p$, for some $p \in [0,1]$. Let $\bD$ be the union of $\mathbf{X}$ and the set of vertices $v \in V(G)$ not dominated by $\mathbf{X}$. By construction, $\bD$ is a random dominating set of $G$. Now, given a vertex $v \in V(G)$, we have
    \[
    \pr{N[v] \cap \mathbf{X} = \emptyset} = \left( 1-p\right)^{\deg(v)+1} \le \left( 1-p\right)^{d+1}.
    \]
    Thus,
    \[
    \pr{v \in \bD} \le p + \left(1-p\right)^{d+1}.
    \]
    By letting $p \coloneqq \frac{\ln(d+1)}{d+1}$, we have
    \begin{align*}
        \pr{v \in \bD} &\le \frac{\ln(d+1)}{d+1} + \left( 1 - \frac{\ln(d+1)}{d+1}\right)^{d+1}\\ 
        &\le \frac{\ln(d+1)}{d+1} + \exp{\left(\frac{-\ln(d+1)}{d+1}\right)^{d+1}} &\mbox{since for every $x\in \R$, } 1+x \le \exp(x)\\
        &\le \frac{1 + \ln(d+1)}{d+1}.
    \end{align*}
    From \Cref{eq:fdom-proba}, the result follows.
\end{proof}

Since $\dom(G) \le \fdom(G)$ for every graph $G$, the class of graphs that reach the upper bound in \Cref{prop:asymptotic-bound} contains the \emph{domatically full graphs}, i.e. the class of graphs $G$ such that $\dom(G)=\delta(G)+1$. A handful of graph classes have been shown to be domatically full, such as maximal outerplanar graphs \cite{CoHe77}, interval graphs \cite{LHC90}, and more generally strongly chordal graphs \cite{Far84} --- a graph $G$ is \emph{strongly chordal} if $G$ is chordal and moreover every even cycle of $G$ of length at least $6$ has a chord that cuts it into two even cycles.

The lower bound in \Cref{prop:asymptotic-bound} is asymptotically tight: 
in~\cite{Alon1990}, Alon studies the transversal number of random uniform hypergraphs, from which we infer the existence of graphs $G$ of arbitrarily large minimum degree $d$ such that $\fdom(G) = (1+o(1))d/\ln d$. We provide in \Cref{prop:tight-fdom} an explicit construction of a bipartite graph that reaches this asymptotic lower bound --- this construction can be modified so that it becomes a split graph.

\paragraph{Graphs of minimum degree \texorpdfstring{$2$}{2}}
\Cref{prop:asymptotic-bound} is essentially best possible when $d$ is large, yet many questions remain open when $d$ is small. If $G$ has an isolated vertex, then clearly $\fdom(G)=1$. Otherwise, $G$ has minimum degree at least $1$, and $\fdom(G)\ge\dom(G)\ge 2$, as can be observed by taking a maximal independent set and its complement. Recently, Gadouleau, Harms, Mertzios, and Zamaraev ~\cite{GHMZ24} gave a simple characterization of every connected graph $G$ such that $\fdom(G)=2$: a connected graph $G$ has $\fdom(G)=2$ if and only if $\delta(G)<2$ or $G=C_4$; otherwise, they conjectured that $\fdom(G)\ge \frac{7}{3}$.
This is true when restricting to (induced)$K_{1,6}$-free graphs, which is a consequence of a result of Abbas, Egerstedt, Liu, Thomas, and Whalen~\cite{AELTW16}. They proved that, except for the family $\B$ of eight graphs depicted in \Cref{fig:badgraphs}, every connected graph of minimum degree at least $2$ has a dominating $(5:2)$-colouring. Indeed, every graph in $\B\setminus \{C_4\}$ has a dominating $(7:3)$-colouring; this is depicted in \Cref{fig:introduction:fractionaldomatic:graphs} --- we note that every graph $G\in \B$ that is not depicted in \Cref{fig:introduction:fractionaldomatic:graphs} contains $C_7$ as a spanning subgraph, so any dominating $(7:3)$-colouring of $C_7$ is also a dominating $(7:3)$-colouring of $G$.
It is possible to drop the $K_{1,6}$-free hypothesis if one desires only a bound on the domination number rather than the fractional domatic number: McCuaig and Shepherd~\cite{McSh89} showed that every $n$-vertex connected graph of minimum degree at least $2$ not in $\B\setminus \{K_{2,3}\}$ has a dominating set of size at most $2n/5$. 
Focusing on regular graphs, Fujita, Yamashita, and Kameda~\cite{FKY00} showed that every cubic graph $G$ (every vertex of $G$ has degree 3) has a dominating $(5:2)$-colouring.

Our main contribution is to show a common strengthening of all the results stated above.
\begin{thm}\label{thm:5/2}
    Let $\B$ be the family of graphs depicted in \Cref{fig:badgraphs}. For every connected graph $G\notin \B$ of minimum degree at least $2$, it holds that
    \[ \fdom(G) \ge \frac{5}{2}.\]
\end{thm}

    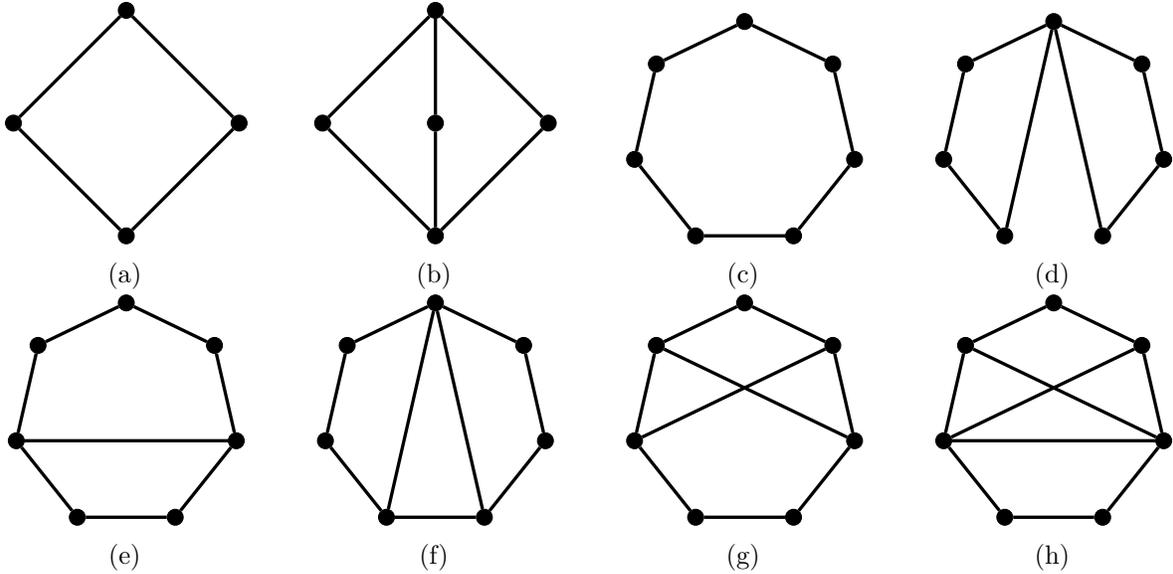
\begin{figure}[!htbp]
            \centering
            \begin{subfigure}[b]{0.24\textwidth}
            \centering
            \begin{tikzpicture}[every node/.style={shape=circle,draw=black,fill,inner sep=0pt,minimum size=6pt}, every edge/.style={line width=1.25pt,draw,black}]
                \graph[clockwise, radius=1.5cm] {subgraph C_n [V={a, b, c, d},empty nodes,name=A]  };
            \end{tikzpicture}
            \caption{}
            \label{fig:badgraphs:C4}
        \end{subfigure}
        \begin{subfigure}[b]{0.24\textwidth}
            \centering
        \begin{tikzpicture}[every node/.style={shape=circle,draw=black,fill,inner sep=0pt,minimum size=6pt}, every edge/.style={line width=1.25pt,draw,black}]
                \graph[clockwise, radius=1.5cm] {subgraph C_n [V={a, b, c, d},empty nodes,name=A]  };
                       
                \node[] (center) {};
                \draw[] (A c) edge [] (center);
                \draw[] (center) edge[] (A a);
        \end{tikzpicture}
        \caption{}
        \label{fig:badgraphs:K23}
        \end{subfigure}
        \begin{subfigure}[b]{0.24\textwidth}
            \centering
            \begin{tikzpicture}[every node/.style={shape=circle,draw=black,fill,inner sep=0pt,minimum size=6pt}, every edge/.style={line width=1.25pt,draw,black}]
                \graph[clockwise, radius=1.5cm] {subgraph C_n [V={a, b, c, d,e,f,g},empty nodes,name=A]  };
                
            \end{tikzpicture}
            \caption{}
            \label{fig:badgraphs:C7}
        \end{subfigure} 
        \begin{subfigure}[b]{0.24\textwidth}
            \centering
            \begin{tikzpicture}[every node/.style={shape=circle,draw=black,fill,inner sep=0pt,minimum size=6pt}, every edge/.style={line width=1.25pt,draw,black}]
                \graph[clockwise=7, radius=1.5cm,empty nodes,name=A] {a,b,c,d,e,f,g};
                \draw[] (A e) edge[] (A f);
                \draw[] (A f) edge[] (A g);
                \draw[] (A g) edge[] (A a);
                \draw[] (A a) edge[] (A b);
                \draw[] (A b) edge[] (A c);
                \draw[] (A c) edge[] (A d);
                \draw[] (A d) edge[] (A a);
                \draw[] (A e) edge[] (A a);
            \end{tikzpicture}
            \caption{}
            \label{fig:badgraphs:vertexgluedC4}
        \end{subfigure}
        \begin{subfigure}[b]{0.24\textwidth}
            \centering
            \begin{tikzpicture}[every node/.style={shape=circle,draw=black,fill,inner sep=0pt,minimum size=6pt}, every edge/.style={line width=1.25pt,draw,black}]
                \graph[clockwise, radius=1.5cm] {subgraph C_n [V={a, b, c, d,e,f,g},empty nodes,name=A]  };
                \draw[] (A f) edge[] (A c);
            \end{tikzpicture}
            \caption{}
            \label{fig:badgraphs:C7C4}
        \end{subfigure}
        \begin{subfigure}[b]{0.24\textwidth}
            \centering
            \begin{tikzpicture}[every node/.style={shape=circle,draw=black,fill,inner sep=0pt,minimum size=6pt}, every edge/.style={line width=1.25pt,draw,black}]
                \graph[clockwise, radius=1.5cm] {subgraph C_n [V={a, b, c, d,e,f,g},empty nodes,name=A]  };
                \draw[] (A d) edge[] (A a);
                \draw[] (A e) edge[] (A a);
            \end{tikzpicture}
            \caption{}
            \label{fig:badgraphs:C4C4edge}
        \end{subfigure}
         \begin{subfigure}[b]{0.24\textwidth}
            \centering
            \begin{tikzpicture}[every node/.style={shape=circle,draw=black,fill,inner sep=0pt,minimum size=6pt}, every edge/.style={line width=1.25pt,draw,black}]
                \graph[clockwise, radius=1.5cm] {subgraph C_n [V={a, b, c, d,e,f,g},empty nodes,name=A]  };
                \draw[] (A b) edge[] (A f);
                \draw[] (A c) edge[] (A g);
            \end{tikzpicture}
            \caption{}
            \label{fig:badgraphs:C7X}
        \end{subfigure}
        \begin{subfigure}[b]{0.24\textwidth}
            \centering
            \begin{tikzpicture}[every node/.style={shape=circle,draw=black,fill,inner sep=0pt,minimum size=6pt}, every edge/.style={line width=1.25pt,draw,black}]
                \graph[clockwise, radius=1.5cm] {subgraph C_n [V={a, b, c, d,e,f,g},empty nodes,name=A]  };
                \draw[] (A b) edge[] (A f);
                \draw[] (A c) edge[] (A g);
                \draw[] (A f) edge[] (A c);
            \end{tikzpicture}
            \caption{}
            \label{fig:badgraphs:C7XC4}
        \end{subfigure}
           
            \caption{The family $\B$ --- every graph $G\in \mathcal{B}$ has $\fdom(G)<5/2$.}
            \label{fig:badgraphs}
        \end{figure}

As explained above, this implies that every connected graph $G\neq C_4$ of minimum degree at least $2$ has $\fdom(G) \ge 7/3$.
The value $5/2$ is best possible: one obtains a graph $G'$ with $\fdom(G')\le 5/2$ by \emph{gluing} a $C_5$ to any graph $G$ (see \Cref{prop:5/$2$-infinite-family}). 
And even if one restricts to bipartite graphs of large girth, it is possible to construct an infinite family of graphs whose fractional domatic number approaches $5/2$ (see \Cref{prop:5/2-girth}).
It is not possible to replace the bound on $\fdom$ with an explicit dominating $(p:q)$-colouring, where $p$ and $q$ are absolute constants that satisfy $p/q=5/2$, or even $p/q>2$ (see \Cref{prop:pqnotenough}).

\paragraph{Planar graphs of large girth}
One of the key ingredients in our proof of \Cref{thm:5/2} is to show how to extend a fractional dominating colouring to a long enough \emph{suspended path} --- a maximal path with internal vertices of degree $2$. 
This strategy lets us derive a lower bound on the fractional domatic number of planar graphs of minimum degree $2$ that approaches the theoretical upper bound $3$ as their girth increases. 

\begin{thm}\label{thm:planar}
    If $G$ is planar graph of minimum degree at least $2$ and girth at least $g$, then
    \[
    \fdom(G) = 3 - O(1/g) \quad \mbox{as $g \to \infty$.}
    \]
\end{thm}

We moreover generalise the above theorem to several classes of graphs of bounded expansion in \Cref{sec:planar}. Note that by \Cref{prop:5/2-girth}, this cannot be further generalised to the class of all graphs.

\subsection*{Outline and terminology}
\label{sec:terminology}

\noindent 

Given a graph $G$ and an integer $k$, a \emph{suspended $k$-path} in $G$ is a path $P$ of length $k$ with (distinct) endpoints of degree at least $3$ in $G$, and inner vertices of degree $2$ in $G$. 
We denote $G\setminus P$ the graph obtained by removing the edges and inner vertices of $P$ from $G$. Observe that if $G$ has minimum degree at least $2$, then so does $G\setminus P$.
If two suspended paths have the same length and the same endpoints, we say that they are \emph{twins}.

A \emph{hammock} in $G$ is the union of a suspended $2$-path and a suspended $3$-path between two non-adjacent vertices; an illustration is given in \Cref{fig:hammock}.

In $G$, given an integer $d$, a \emph{$d$-vertex} (resp. \emph{$d^-$-vertex}, \emph{$d^+$-vertex}) is a vertex of degree $d$ (resp. at most $d$, at least $d$).

A \emph{separation of order $k$} of $G$ is a pair $(G_0, G_1)$ of subgraphs of $G$ such that\footnote{For all graphs $A,B$, let $A \cup B = (V(A) \cup V(B), E(A) \cup E(B))$ and $A \cap B = (V(A) \cap V(B), E(A) \cap E(B))$.} $G_0\cup G_1 = G$ and $|V (G_0 \cap G_1)| = k$. When $V(G_0\cap G_1)=\{v\}$, we call $v$ a \emph{cut-vertex} of $G$. A graph with no cut-vertices is \emph{$2$-connected}.
A separation $(G_0,G_1)$ \emph{induced by a separating set of vertices} $X$ is such that $V(G_0)\cap V(G_1)=X$.
Observe that if $G$ is $2$-connected, then every $2$-vertex of $G$ is part of a suspended path.

Given a graph $G$ and a positive integer $k$, the $k$-subdivision of $G$, denoted $G^{1/k}$, is obtained by replacing every edge of $G$ by a path of length $k$.

Let $k$ be a positive integer. We denote $\{1,\dots,k\}$ by $[k]$.
In a few places, we rely on an extension of the binomial coefficient, that is defined as $\binom{x}{k} \coloneqq \frac{1}{k!}\prod_{i=0}^{k-1}(x-i)$, for every real number $x$ and nonnegative integer $k$.
Finally, given a finite set of elements $S$, we denote $\cU(S)$ the uniform distribution over $S$.
\medskip

The paper is organised as follows.
In \Cref{sec:upper-bounds}, we derive optimal asymptotic upper bounds on the extremal value of $\fdom$ for graphs of given minimum degree. This is achieved by considering the dual of a fractional dominating colouring, which we call a \emph{fractional bottleneck}.
In \Cref{sec:lower-bounds}, we provide a proof of \Cref{thm:5/2}. To do so, we introduce the notion of \emph{$f$-dominating fractional colourings} for a given demand function $f$, and derive two technical lemmas that let us combine two $f$-dominating fractional colourings obtained inductively on a graph $G$ with a vertex-cut of size $1$ or $2$. 
In \Cref{sec:planar}, we consider several families of graphs of bounded expansion of minimum degree $2$, and show that their fractional domatic number approaches $3$ when their girth increases.
We conclude our work in \Cref{sec:conclusion} where we discuss the algorithmic aspects of our results, explaining how to make them constructive. We finally propose a list of open problems related to our work.

In Appendix \ref{appendix:np-hardness} we discuss the NP-hardness of computing the fractional domatic number.
In Appendix \ref{appendix:counterexample}, we disprove a proposition from \cite{FKY00} concerning domatically full regular graphs.

\section{Upper bounds on \texorpdfstring{$\fdom$}{FDOM} using duality}
\label{sec:upper-bounds}

Given a graph $G$, a set $X\subseteq V(G)$ is called a \emph{bottleneck} if every dominating set of $G$ intersects $X$. We denote $\tau(G)$ the size of a minimum bottleneck in $G$. Clearly, we have $\dom(G) \le \tau(G)$, and $\tau(G) \le \delta(G)+1$ since every closed neighbourhood is a bottleneck in $G$. In fact, every bottleneck must contain a closed neighbourhood, therefore we have exactly $\tau(G) = \delta(G)+1$. A \emph{fractional bottleneck} is a weight assignment $w\colon V(G) \to [0,1]$ such that for every dominating set $D$ of $G$, one has $w(D) \coloneqq \sum_{v\in D} w(d) \ge 1$. We let $\tau^*(G)$ denote the minimum total weight of a fractional bottleneck of $G$, that is the solution to the following linear program, which is the dual of \eqref{eq:LP-fdom}.

\begin{equation}
    \label{eq:LP-dual}
    \begin{aligned}
    \mbox{Minimise}& \sum_{v\in V(G)} w(v),\\
    \mbox{Subject to}& \left\{\begin{array}{rl}\displaystyle
        w(D) \ge 1& \mbox{for all $D \in \cD(G)$;}\\
        w(v) \ge 0 & \mbox{for all $v\in V(G)$.}
    \end{array} \right.
    \end{aligned}
\end{equation}

\noindent
By the Strong Duality Theorem, one has $\fdom(G)=\tau^*(G)$.

Fractional bottlenecks are a convenient tool for certifying upper bounds on the fractional domatic number. As an introductory demonstration, we use duality to prove that the trivial upper bound $\dom(G) \le \delta(G) + 1$ also holds for the fractional domatic number.

\begin{prop}\label{prop:d+1}
    Let $G$ be a graph of minimum degree $d$. Then $\fdom(G) \le d + 1$.
\end{prop}
\begin{proof}
    Let $v$ be a vertex of degree $d$, and assign the weight $1$ to every vertex of $N[v]$. Every dominating set $D$ must intersect $N[v]$, and thus has weight $w(D) \ge 1$. Therefore, we have a fractional bottleneck of weight $d+1$, which certifies that $\fdom(G) = \tau^*(G) \le d+1$.
\end{proof}

It follows directly from the Pigeonhole Principle that, if $G$ has $n$ vertices and domination number $\gamma$, then $\dom(G) \le n/\gamma$. Using fractional bottlenecks, it is straightforward to show that the same upper bound holds for $\fdom(G)$; this was already shown in \cite{GoHe20} using a simple counting argument applied to \Cref{eq:fdom-family}, but we include our alternative proof for the sake of completeness and as another simple illustration of fractional bottlenecks.

\begin{prop}\label{prop:n/gamma}
    Let $G$ be a graph with $n$ vertices and domination number $\gamma$. Then $\fdom(G) \le \frac{n}{\gamma}$, with equality if $G$ is vertex-transitive.
\end{prop}
\begin{proof}
    We assign the weight $1/\gamma$ to every vertex of $G$. Every dominating set $D$ contains at least $\gamma$ vertices, and thus has weight $w(D) \ge 1$. Therefore, we have a fractional bottleneck of weight $n/\gamma$, which certifies that $\fdom(G) = \tau^*(G) \le \frac{n}{\gamma}$.

    If $G$ is vertex-transitive, let $\mathcal{D}_\gamma$ be the set of dominating sets of $G$ of size $\gamma$, and observe that any two vertices $u$ and $v$ are contained in the same number of dominating sets of $\mathcal{D}_\gamma$. Indeed, by considering an automorphism of $G$ which maps $u$ onto $v$, we obtain a bijection between $\{D \in \mathcal{D}_\gamma : u \in D\}$ and $\{D \in \mathcal{D}_\gamma : v \in D\}$. Consequently, the uniform distribution $\bD \sim \mathcal{U(D_\gamma)}$ verifies $\pr{u \in \bD} = \frac{\gamma}{n}$ for all $u\in V(G)$, since all vertices are equally likely to belong to $\bD$ and $\sum_{v\in V(G)} \pr{v\in \bD} = \mathbb{E}\big[|\bD|\big]=\gamma$. This shows that $\fdom(G) \ge \frac{n}{\gamma}$ if $G$ is vertex-transitive, and equality follows from the previous point.
\end{proof}

While it is possible to prove \Cref{prop:d+1} and \Cref{prop:n/gamma} from \Cref{eq:fdom-family} by reasoning on all possible families of dominating sets of the graph, as was done in \cite{GoHe20}, this approach becomes tedious when computing the fractional domatic number of specific graphs. Instead, by exhibiting a well-chosen fractional bottleneck, we simply need to check that all dominating sets have weight at least $1$, which turns out to be a manageable task when the graph presents many symmetries.

To illustrate how duality helps us compute the fractional domatic number of certain graph families, we apply it to the class of complete bipartite graphs.

\begin{prop}
    For every integers $m\ge n\ge 2$, one has
    \[\fdom(K_{m,n}) = 1+n\pth{1-\frac{1}{m}}.\]
\end{prop}

\begin{proof}
    Let $G=K_{m,n}$, and let $A,B$ be the two parts of $G$, with $|A|=m$ and $|B|=n$.

    We first exhibit a fractional dominating colouring of $G$ of weight $1+n(1-1/m)$, under the form of a solution to \eqref{eq:LP-fdom}. For every $a\in A$ and $b\in B$, $D\coloneqq \{a,b\}$ is a dominating set of $G$; we assign it the weight $x_D \coloneqq 1/m$. We assign to the dominating set $A$ the weight $x_A \coloneqq 1-n/m$. All other dominating sets have weight $0$. It is straightforward to check that this satisfies the constraints of \eqref{eq:LP-fdom}.

    We now exhibit a fractional bottleneck of $G$ of weight $1+n(1-1/m)$. Let $w(u)\coloneqq 1/m$ for every $u\in A$, and $w(v)\coloneqq 1-1/m$ for every $v\in B$. Let $D$ be a dominating set of $G$. If $D = A$, then $w(D) = 1$. If $D=B$, then $w(D) = n(1-1/m) \ge 2 - 2/m \ge 1$. Otherwise, $D$ contains at least one vertex from $A$ and $B$, thus $w(D) \ge 1$. This weight assignment satisfies the constraints of \eqref{eq:LP-dual}.
\end{proof}

\newcommand{\Gnd}{H_{n,d}}

For integers $1 \le d < n$, let $\Gnd$ be the bipartite graph consisting of parts $A = [n]$ and $B = \binom{A}{d}$ and such that $ab \in A \times B$ is an edge if and only if $a \in b$. In particular, $H_{n,1}$ is isomorphic to a matching on $2n$ vertices, $H_{n,2}$ is isomorphic to the subdivision of the complete graph $K_n$, and more generally $\Gnd$ is the incidence graph of a complete $d$-uniform hypergraph on $n$ vertices.

\begin{fact}\label{fact:Gnd}
    For every $1\le d < n$, $\Gnd$ has minimum degree equal to $d$.
\end{fact}

By \Cref{prop:asymptotic-bound} and \Cref{fact:Gnd}, we obtain $\fdom(\Gnd) \ge \frac{d+1}{1 + \ln(d+1)}$. We show that if $n$ is much larger than $d$, then this lower bound is asymptotically tight as $d \rightarrow \infty$. 

\begin{prop}
\label{prop:tight-fdom}
    Let $d$ be an integer, and $n \coloneqq d(q + 1)$ for some integer $q \ge \ln(d)$. Then $\fdom(\Gnd) \le (1+o(1))\, d/\ln d$ as $d\to \infty$.
\end{prop}

\begin{proof}
    We construct a fractional bottleneck of $\Gnd$ of weight $(1+o(1))\, d/\ln d$ as $d\to\infty$. 
    Let $m\coloneqq (\ln d - 2\ln \ln d)q$. We assign to each vertex $a\in A$ weight $w_A \coloneqq 1/m$, and to each vertex $b\in B$ weight $w_B \coloneqq 1/\binom{n-m}{d}$.
    The weight of a dominating set of $\Gnd$ which contains more than $n-d$ vertices of $A$ is at least $\frac{n-d}{m} > \frac{dq}{q \ln d} > 1$.
    If a dominating set of $\Gnd$ contains $i \in [0,n-d]$ vertices in $A$, it must contain at least $\binom{n-i}{d}$ vertices in $B$. 
    So the weight induced by a dominating set which contains $i$ vertices in $A$ is at least $iw_A + \binom{n-i}{d}w_B$. If $i\ge m$, then $iw_A \ge 1$, while if $i\le m$, then $\binom{n-i}{d}w_B \ge 1$, so this is a valid fractional bottleneck of $\Gnd$. Its weight is
    
    \begin{align*}
        n w_A + \binom{n}{d}w_B &= \frac{d(q+1)}{(\ln d - 2\ln\ln d)q} + \prod_{i=1}^d \frac{n-d+i}{n-m-d+i} \\
        &\le \frac{d(q+1)}{(\ln d - 2\ln \ln d)q} + \pth{\frac{n-d}{n-m-d}}^d\\
        &\le \frac{d}{\ln d - 2\ln\ln d-1} +  \pth{\frac{q(d+1) - d}{q(d+1)- d - q(\ln d - 2\ln \ln d)}}^d \\
        &= (1+o(1))\, \frac{d}{\ln d} +  \pth{1 + \frac{\ln d-2\ln\ln d}{d- \ln d + 2\ln\ln d}}^d \\
        &\le (1+o(1))\, \frac{d}{\ln d} + \exp\pth{\frac{\ln d-2\ln\ln d}{d- \ln d + 2\ln\ln d} \cdot d}\\
        &= (1+o(1))\, \frac{d}{\ln d} + \exp\pth{\ln d-2\ln\ln d + \frac{(\ln d - 2 \ln\ln d)^2}{d - \ln d + 2\ln \ln d}}\\
        &\le (1+o(1))\, \frac{d}{\ln d} + \frac{2d}{(\ln d)^2}\\
        &= (1+o(1))\, \frac{d}{\ln d} \qedhere
    \end{align*} 
\end{proof}
\begin{rk}
\label{rk}
    At no point in our analysis have we used the structure of $\Gnd[A]$. By construction, $A$ is an independent set, but we could add any edges within $A$, and the same upper bound still holds. In particular, if we make $A$ a clique, we have a split graph (which is chordal).
\end{rk}

When $G$ is a graph of minimum degree $d$, one could wonder whether it is possible to find a dominating $(p:q)$-colouring of $G$ with $p/q$ increasing as a function of $d$, and $q$ bounded as a function of $d$. The following proposition shows that this is not possible in general.

\begin{prop}
    \label{prop:pqnotenough}
    For every integers $q_0,d\ge 1$, if $n\ge (d-1)\binom{2q_0+1}{q_0} + 1$, then there is no dominating $(p:q)$-colouring of $\Gnd$ for any integers $q\le q_0$ and $p>2q$.
\end{prop}

\begin{proof}
    First, observe that if a graph $G$ has a dominating $(p:q)$-colouring $\sigma$, then it also has a dominating $(p':q)$-colouring for every $q \le p'\le p$ (one can obtain such a colouring from $\sigma$ by replacing in $\sigma(v)$ the subset of colours $x$ such that $x>p'$ with a subset of $[q] \setminus \sigma(v)$ of same cardinality, for every vertex $v$). Therefore, it suffices to prove the statement with $p=2q+1$.
    Let $\phi$ be any $(2q+1:q)$-colouring of $\Gnd$ with $q\le q_0$. By the Pigeonhole Principle (and using that $\binom{2q+1}{q} \le \binom{2q_0+1}{q_0}$), there exists at least $d$ distinct vertices $u_1, \ldots, u_d \in A$ such that $\phi(u_1) = \cdots = \phi(u_d)$. Consider the vertex $w \coloneqq \{u_1, \ldots, u_d\}$ in $B$. 
    Then one has $\card{\phi(N[w])} \le 2q$, thus $\phi$ is not a dominating $(2q+1:q)$-colouring of $\Gnd$.
\end{proof}

\begin{rk}
\label{rk2}
    Again, we never use the structure of $\Gnd[A]$ in our analysis. Therefore, we could add any set of edges within $A$, and the same conclusion holds.
\end{rk}

Taking the particular case $q_0 = 1$, \Cref{prop:pqnotenough} states that if $n \ge 3d-2$, then there is no dominating $(3:1)$-colouring of $\Gnd$. 
So we retrieve (with a more direct proof) the result from \cite{Zel83} that $\dom(\Gnd) \le 2$

We now show that the lower bound $5/2$ on the fractional domatic number of connected graphs of minimum degree $2$ is best possible up to finitely many exceptional graphs.

\begin{figure}[!htbp]
        \centering
            \begin{tikzpicture}[long dash/.style={dash pattern=on 10pt off 2pt, draw=gray},Pattern/.style={pattern=north east hatch teal, pattern color=gray!30, hatch distance=7pt, hatch thickness=2pt}]
            \draw[preaction={fill=gray!15}, Pattern,ultra thick,long dash,name path=left] plot[smooth cycle] coordinates {(3.3,1.5) (2.112,1.5) (2.112,-2) (3.3,-2)};

            \node[shape=circle,fill=gray!20,inner sep=1pt,minimum size=7pt] at (2.772,-0.33) (g) {$G_0$};

            \node[label={[label distance = 0.1cm]above left:{$u_0$}},shape=circle,draw=black,fill,inner sep=0pt,minimum size=6pt, above =0.5cmof g] (u1) {};
            \node[label={[label distance = 0.1cm]below left:{$v_0$}},shape=circle,draw=black,fill,inner sep=0pt,minimum size=6pt, below=0.5cm of g] (u3) {};

                \node[label={\small },shape=circle,inner sep=0pt,minimum size=6pt, above=0.3cm of u3] (dotua1) {};
                \node[label={\small },shape=circle,inner sep=0pt,minimum size=6pt, above right=0.3cm of u3] (dotua2) {};
                \node[label={\small },shape=circle,inner sep=0pt,minimum size=6pt, right=0.3cm of u3] (dotua3) {};
                \node[label={\small },shape=circle,inner sep=0pt,minimum size=6pt, below right=0.3cm of u3] (dotua4) {};
                \node[label={\small },shape=circle,inner sep=0pt,minimum size=6pt, below=0.3cm of u3] (dotua5) {};
                \draw[line width=2pt,dotted] (u3) -- (dotua1);
                \draw[line width=2pt,dotted] (u3) -- (dotua2);
                \draw[line width=2pt,dotted] (u3) -- (dotua3);
                \draw[line width=2pt,dotted] (u3) -- (dotua4);
                \draw[line width=2pt,dotted] (u3) -- (dotua5);

                \node[label={\small },shape=circle,inner sep=0pt,minimum size=6pt, above=0.3cm of u1] (dotub1) {};
                \node[label={\small },shape=circle,inner sep=0pt,minimum size=6pt, above right=0.3cm of u1] (dotub2) {};
                \node[label={\small },shape=circle,inner sep=0pt,minimum size=6pt, right=0.3cm of u1] (dotub3) {};
                \node[label={\small },shape=circle,inner sep=0pt,minimum size=6pt, below right=0.3cm of u1] (dotub4) {};
                \node[label={\small },shape=circle,inner sep=0pt,minimum size=6pt, below=0.3cm of u1] (dotub5) {};
                \draw[line width=2pt,dotted] (u1) -- (dotub1);
                \draw[line width=2pt,dotted] (u1) -- (dotub2);
                \draw[line width=2pt,dotted] (u1) -- (dotub3);
                \draw[line width=2pt,dotted] (u1) -- (dotub4);
                \draw[line width=2pt,dotted] (u1) -- (dotub5);

                \node[label={[label distance = 0.1cm]left:{$z$}},shape=circle,draw=black,fill,inner sep=0pt,minimum size=6pt, left =1.32cm of g] (d) {};
                \node[label={[label distance = 0.1cm]above left:{$x$}},shape=circle,draw=black,fill,inner sep=0pt,minimum size=6pt, above left =1.32cm of d] (a) {};
                \node[label={[label distance = 0.1cm]below left:{$y$}},shape=circle,draw=black,fill,inner sep=0pt,minimum size=6pt, below left =1.32cm of d] (c) {};
                
                \draw[line width=1.25pt,black!60] (u1) -- (a) -- (c)  -- (u3) -- (d) -- (u1);
                
            \end{tikzpicture}
\caption{A hammock between two non-adjacent vertices $u_0$ and $v_0$.}
\label{fig:hammock}
\end{figure}
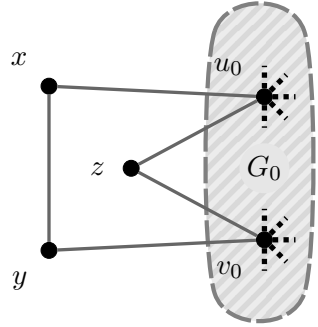

\begin{prop}
    \label{prop:5/$2$-hammock}
    If a graph $G$ contains a hammock, then $\fdom(G)\le 5/2$. 
\end{prop}

\begin{proof}
    Let $u_0,v_0$ be the endpoints of the hammock; let $u_0 \link x \link y \link v_0$ and $u_0 \link z \link v_0$ be the two suspended paths of the hammock; 
    let $G_0 \coloneqq G\setminus \{x,y,z\}$ be the graph obtained from $G$ by removing the hammock; 
    and let $H$ be the subgraph (isomorphic to $C_5$) induced by the hammock in $G$. See \Cref{fig:hammock} for an illustration.
    We prove the result by constructing a fractional bottleneck $w\colon V(G) \to [0,1]$ of $G$ of weight $5/2$ as follows: 
    \[ w(x)\coloneqq \begin{cases}
        1/2 &\mbox{if } x\in V(H);\\
        0 & \mbox{otherwise.}
    \end{cases}\]
    The weight of $w$ is $\frac{1}{2}\,|V(H)|=5/2$, and it remains to check that $w$ is indeed a fractional bottleneck of $G$. Let $D$ be any dominating set of $G$; we show that $w(D)\ge 1$.
    Let $D_0 \coloneqq D\cap V(G_0)$. If $\{u_0,v_0\} \subseteq  D_0$, then $w(D) \ge w(u_0)+w(v_0)=1$ and we are done. 
    If $\{u_0,v_0\} \cap D_0 = \emptyset$, then $x,y,z$ are not dominated by $D_0$, so they must be dominated by $D\setminus D_0$. It follows that $z\in D$, and (at least) one of $x$ and $y$ is in $D$. Hence $w(D) \ge w(z) + \min \{w(x),w(y)\} = 1$ and we are done.
    The last remaining case to consider is when exactly one of $u_0$ and $v_0$ is in $D_0$, say $u_0$ without loss of generality. Then $y$ is not dominated by $D_0$, so it must be dominated by $D\setminus D_0$. We infer that $D\setminus D_0$ is non-empty, hence $w(D) = w(D_0) + w(D\setminus D_0) \ge w(u_0) + \frac{1}{2} = 1$, as desired.
\end{proof}

So, in order to construct a graph $G$ of minimum degree $2$ with $\fdom(G)\ge 5/2$, one can take any connected graph $G_0$ of minimum degree $2$ with $2$ non-adjacent vertices $u_0,v_0$ (one could even have $\deg(u_0) < 2$ or $\deg(v_0) < 2$ in $G_0)$, and add a hammock between $u_0$ and $v_0$.

\begin{cor}
    \label{prop:5/$2$-infinite-family}
    There are infinitely many connected graphs $G$ of minimum degree $2$ such that $\fdom(G)\le 5/2$. 
\end{cor}

Every graph constructed to prove \Cref{prop:5/$2$-infinite-family} contains a copy of $C_5$ as an induced subgraph. So one may wonder how the fractional domatic number is affected if we forbid $C_5$ as a subgraph. We show that it is not possible to derive a general lower bound better than $5/2$ even if we restrict to bipartite graphs of girth $6$. We do so by constructing a family of bipartite graphs of girth $6$ whose fractional domatic number, which we compute exactly, approaches $5/2$. This also shows that $\fdom$ can take infinitely many values in the proximity of $5/2$.

\begin{figure}[!ht]
    \centering
    \begin{tikzpicture}[baseline=0pt, every node/.style={shape=circle,draw=black,fill,inner sep=0pt,minimum size=6pt}, every edge/.style={line width=1.25pt,draw,black}]
    \def\R{1}
    \def\n{4}
    \def\x{0.2}

    \fill[blue!20] (-\x,\R-\x) rectangle (\x,\n*\R+\x);
    \node[fill=none,draw=none] at (0,0.4) {\textcolor{blue}{$W_0$}};
    \fill[blue!20] (3-\x,\R-\x) rectangle (3+\x,\n*\R+\x);
    \node[fill=none,draw=none] at (3,0.4) {\textcolor{blue}{$W_1$}};

    \fill[red!20] (-0.3,\R-\x) rectangle (-1.1,\n*\R+\x);
    \node[fill=none,draw=none] at (-0.7,0.4) {\textcolor{red}{$U_0$}};
    \fill[red!20] (3.3,\R-\x) rectangle (4.1,\n*\R+\x);
    \node[fill=none,draw=none] at (3.7,0.4) {\textcolor{red}{$U_1$}};

    \fill[green!20] (0.4,\R-\x) rectangle (0.8,\n*\R+\x);
    \node[fill=none,draw=none] at (0.6,0.4) {\textcolor{green!50!black}{$V_0$}};
    \fill[green!20] (2.2,\R-\x) rectangle (2.6,\n*\R+\x);
    \node[fill=none,draw=none] at (2.4,0.4) {\textcolor{green!50!black}{$V_1$}};
    
    \foreach \i in {1,...,\n} {
        \coordinate[circle] (w\i) at (0,\i*\R);
        \coordinate[circle] (w'\i) at (3,\i*\R);
    }
    \foreach \i in {1,...,\n} {
        \foreach \j in {1,...,\n} {
            \draw (w\i) -- (w'\j) node[pos=0.17, circle, scale=0.6] {}
            node[pos=0.83, circle, scale=0.6] {};
        }
    }
    \pgfmathtruncatemacro{\prev}{\n-1}
    \foreach \i in {1, ..., \prev} {
       \pgfmathtruncatemacro{\next}{\i+1}
        \foreach \j in {\next,...,\n} {
            \draw (w\i) to[bend left=90] node[midway, circle, scale=0.6] {} (w\j);
            \draw (w'\i) to[bend right=90] node[midway, circle, scale=0.6] {} (w'\j) ;
        }
    }
    
\end{tikzpicture}
    \caption{The graph $G_4$ constructed in \Cref{prop:bipartite-girth6}}
    \label{fig:prop6}
\end{figure}
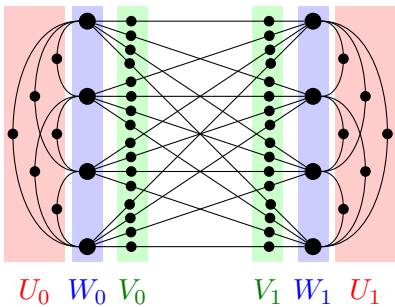

\begin{prop}
    \label{prop:bipartite-girth6}
    For every integer $n \ge 2$, there exists a bipartite graph $G$ of minimum degree $2$ and girth $6$ such that $\fdom(G) \;=\; \frac{5n-2}{2n-1} \;=\; \frac{5}{2} + \frac{1}{4n-2}$.
\end{prop}

\begin{proof}
    We let $G_n$ be the edge-union of ${K_{n,n}}^{1/3}$ and $2{K_n}^{1/2}$ (see \Cref{fig:prop6} for an illustration of $G_4$). We claim that $\fdom(G_n) = \frac{5n-2}{2n-1}$.

    We partition $V(G_n)$ into three sets $U,V,W$, where $U$ is the set of vertices that come from the $1$-subdivision of an edge and $V$ is the set of vertices that come from the $2$-subdivision of an edge. So we have $|U|=2\binom{n}{2}=n(n-1)$, $|V|=2n^2$, and $|W|=2n$. Moreover, every vertex in $U\cup V$ has degree $2$, and every vertex in $W$ has degree $2n-1$. We moreover partition $W$ into $W_0,W_1$ that coincide with the two parts of $K_{n,n}$ before the subdivision, or equivalently to the vertex set of each copy of $K_n$ in $2K_n$. We have $|W_0|=|W_1|=n$.
    
    We let $w\colon V(G_n) \to [0,1]$ be defined by
    \[ w(v) \coloneqq \begin{cases}
        \frac{1}{2n} & \mbox{if } v \in W;\\
        \frac{1}{n(2n-1)} & \mbox{otherwise,}
        \end{cases}.\]
    \begin{claim}
        The weight function $w$ is a fractional bottleneck of $G_n$, of weight $\frac{5n-2}{2n-1}$.
    \end{claim}
    
    \begin{proofclaim}
    Let $D$ be a dominating set of $G$.
    Let $s\coloneqq |D\cap W_0|\le n$ and $t\coloneqq |D\cap W_1|\le n$. Then there are $\binom{n-s}{2}+\binom{n-t}{2}$ vertices in $U$ that are not dominated by $D\cap W$, hence $|D\cap U| \ge \binom{n-s}{2}+\binom{n-t}{2}$.
    The set $V \setminus N(D\cap W)$ induces a graph that consists of $s(n-t)+t(n-s)$ isolated vertices and $(n-s)(n-t)$ isolated copies of $K_2$. We infer that $|D\cap V|\ge n^2-st$.
    We have
    \begin{align*}
        w(D) &= w(D\cap U) + w(D\cap V) + w(D\cap W) \\
        & \ge \frac{\binom{n-s}{2}+\binom{n-t}{2}}{n(2n-1)} + \frac{n^2-st}{n(2n-1)} + \frac{s+t}{2n} \\
        &= \frac{n(n-1) - (s+t)(n-1/2) + (s^2+t^2)/2}{n(2n-1)} + \frac{n^2 - st}{n(2n-1)} + \frac{(s+t)(n-1/2)}{n(2n-1)}\\
        &= \frac{n(2n-1) + (s-t)^2/2}{n(2n-1)} \ge 1,\\
    \end{align*}
    so $w$ is indeed a fractional bottleneck of $G_n$. The weight of $w$ is $\frac{|W|}{2n} + \frac{|U|+|V|}{n(2n-1)} = \frac{5n-2}{2n-1}$.
    \end{proofclaim}

    To construct a fractional dominating colouring of $G_n$ of weight $\frac{5n-2}{2n-1}$, we introduce $3$ possible random outcomes for a random dominating set $\bD$ of $G_n$. 
    \begin{enumerate}[label=(\roman*)]
        \item We let $D_0 \coloneqq W$.
        \item We initially set $\bD_1 \gets U$, then for every copy of $K_2$ induced by $V$ in $G_n$, we add one of its two vertices in $\bD_1$, uniformly at random.
        \item We initially set $\bD_2 \gets \emptyset$. We respectively let $\bX_0$ and $\bX_1$ be $\lfloor n/2\rfloor$-subsets of $W_0$ and $W_1$, chosen uniformly at random. We add $\bX_0 \cup \bX_1$ to $\bD_2$, as well as $U\setminus N(\bX_0\cup \bX_1)$.
        The subgraph of $G_n$ induced by $V\setminus N(\bX_0\cup \bX_1)$ consists of copies of $K_1$ and $K_2$; we add to $\bD_2$ all vertices that belong to a copy of $K_1$, and one vertex chosen uniformly at random from each copy of $K_2$.
    \end{enumerate}
    We now define 
    \[ \bD \coloneqq \begin{cases}
        D_0 & \mbox{with probability} \begin{cases}
            \frac{1}{5n-2} & \mbox{if $n$ is even,}\\
            \frac{3n-1}{(n+1)(5n-2)} & \mbox{if $n$ is odd;}
        \end{cases}\vspace{6pt}\\ 
        \bD_1 & \mbox{with probability} \begin{cases}
            \frac{n+1}{5n-2} & \mbox{if $n$ is even,}\\
            \frac{n-1}{5n-2} & \mbox{if $n$ is odd;}
        \end{cases}\vspace{6pt}\\ 
        \bD_2 & \mbox{with probability} \begin{cases}
            \frac{4n-4}{5n-2} & \mbox{if $n$ is even,}\\
            \frac{4n^2}{(n+1)(5n-2)} & \mbox{if $n$ is odd.}
        \end{cases}
    \end{cases}\]
    \begin{claim}
        The random set $\bD$ is a fractional dominating colouring of $G_n$ of weigth $\frac{5n-2}{2n-1}$.
    \end{claim}

\begin{proofclaim}
    By construction, $D_0$, $\bD_1$, and $\bD_2$ are (random) dominating sets of $G_n$, therefore $\bD$ is a random dominating set of $G_n$. There remains to check that $\pr{v\in \bD}=\frac{2n-1}{5n-2}$ for every vertex $v\in V(G_n)$. 

    First, let $w\in W$. If $n$ is even, then 
    \begin{align*}
        \pr{w\in \bD} &= \frac{1}{5n-2}\, \pr{w\in D_0} + \frac{n+1}{5n-2}\, \pr{w\in \bD_1} + \frac{4n-4}{5n-2}\, \pr{w\in \bD_2}\\
        & = \frac{1}{5n-2} + 0 + \frac{4n-4}{5n-2}\cdot \frac{1}{2} = \frac{2n-1}{5n-2}.
    \end{align*}
    If $n$ is odd, then
    \begin{align*}
        \pr{w\in \bD} &= \frac{3n-1}{(n+1)(5n-2)}\, \pr{w\in D_0} + \frac{n-1}{5n-2}\, \pr{w\in \bD_1} + \frac{4n^2}{(n+1)(5n-2)}\, \pr{w\in \bD_2}\\
        &= \frac{3n-1}{(n+1)(5n-2)} + 0 + \frac{4n^2}{(n+1)(5n-2)}\cdot \frac{n-1}{2n} = \frac{2n-1}{5n-2}.
    \end{align*}
    Next, let $u\in U$. If $n$ is even, then
    \begin{align*}
        \pr{u\in \bD} &= \frac{1}{5n-2}\, \pr{u\in D_0} + \frac{n+1}{5n-2}\, \pr{u\in \bD_1} + \frac{4n-4}{5n-2}\, \pr{u\in \bD_2}\\
        & = 0+\frac{n+1}{5n-2} + \frac{4n-4}{5n-2}\cdot \frac{\binom{n/2}{2}}{\binom{n}{2}} = \frac{2n-1}{5n-2}.
    \end{align*}
    If $n$ is odd, then
    \begin{align*}
        \pr{u\in \bD} &= \frac{3n-1}{(n+1)(5n-2)}\, \pr{u\in D_0} + \frac{n-1}{5n-2}\, \pr{u\in \bD_1} + \frac{4n^2}{(n+1)(5n-2)}\, \pr{u\in \bD_2}\\
        &= 0 + \frac{n-1}{5n-2} + \frac{4n^2}{(n+1)(5n-2)}\cdot \frac{\binom{(n+1)/2}{2}}{\binom{n}{2}} = \frac{2n-1}{5n-2}.
    \end{align*}
    Finally, let $v\in V$. We write $N(v)=\{w,v'\}$ with $w\in W$ and $v'\in V$, and we let $\{w'\} \coloneqq  N(v')\cap W$. We have
    \begin{align*} 
    \pr{v\in \bD_2} &= \pr{w\notin \bD_2 \et w' \in \bD_2} + \frac{1}{2}  \cdot \pr{w\notin \bD_2 \et w' \notin \bD_2} \\
    &= \begin{cases}
        \frac{3}{8} & \mbox{if $n$ is even;}\\
        \frac{(n-1)(n+1)}{4n^2} + \frac{1}{2}\cdot \frac{(n+1)^2}{4n^2} & \mbox{if $n$ is odd.}
    \end{cases}
    \end{align*}
    So, if $n$ is even, one has
    \begin{align*}
        \pr{v\in \bD} &= \frac{1}{5n-2}\, \pr{v\in D_0} + \frac{n+1}{5n-2}\, \pr{v\in \bD_1} + \frac{4n-4}{5n-2}\, \pr{v\in \bD_2}\\
        & = 0 + \frac{n+1}{5n-2}\cdot\frac{1}{2} + \frac{4n-4}{5n-2}\cdot \frac{3}{8} = \frac{2n-1}{5n-2}.
    \end{align*}
    Otherwise, if $n$ is odd, one has
    \begin{align*}
        \pr{v\in \bD} &= \frac{3n-1}{(n+1)(5n-2)}\, \pr{v\in D_0} + \frac{n-1}{5n-2}\, \pr{v\in \bD_1} + \frac{4n^2}{(n+1)(5n-2)}\, \pr{v\in \bD_2}\\
        &= 0 + \frac{n-1}{5n-2} \cdot\frac{1}{2} + \frac{4n^2}{(n+1)(5n-2)} \pth{\frac{(n-1)(n+1)}{4n^2} + \frac{1}{2}\cdot \frac{(n+1)^2}{4n^2}} = \frac{2n-1}{5n-2}.
    \end{align*}
    This covers all possible cases, hence $\bD$ is a fractional dominating colouring of $G_n$ of weight $\frac{5n-2}{2n-1}$.
\end{proofclaim}

We conclude that $\fdom(G_n) = \frac{5n-2}{2n-1}$.
\end{proof}

We will next show that \Cref{thm:5/2} is essentially best possible even when considering graphs of large girth. The construction is based on random graphs, for which the analysis uses the multiplicative Chernoff bounds.

\begin{lemma}[Chernoff bound, multiplicative form]\label{chernoff-multiplicative}
	Let $\bX \coloneqq \bX_1 + \dots + \bX_n$ be the sum of independent Bernoulli variables (i.e. $\bX_i \in \{0,1\}$). Let $\mu \coloneqq \esp{\bX}$ and $\delta \ge 0$. Then the following inequalities hold.
	\begin{itemize}
		\item $\pr{\bX \ge (1 + \delta)\mu} \le e^{-\delta^2\mu / (2 + \delta)}$.
		\item $\pr{\bX \le (1 - \delta)\mu} \le e^{-\delta^2\mu / 2}$, if $\delta \le 1$.
		\item $\pr{|\bX - \mu| \ge \delta\mu} \le 2e^{-\delta^2\mu / 3}$, if $\delta \le 1$.
	\end{itemize}
\end{lemma}

We now prove that there exist graphs of minimum degree $2$ and arbitrarily large girth with fractional domatic number arbitrarily close to $\frac{5}{2}$. In contrast, it is known that every graph $G$ of minimum degree $2$ and girth $g$ has domination number at most $\gamma(G) \le (\frac{1}{3} + \frac{2}{3g})n$ \cite{LoDi08}.

\begin{prop}\label{prop:5/2-girth}
	For every $\eps>0$ and $g\ge 6$, there is a bipartite graph $H$ of minimum degree $2$, of girth at least $g$, and with
	\[ \fdom(H) \le \frac{5}{2}+\eps. \]
\end{prop}
\begin{proof}
	Let $\eps>0$ be fixed, let $g \geq 6$, and let $n$ be a large enough integer in terms of $\eps$.
	Let $G \gets \randomgraph[2n]{p}$ be a random graph, where $p \coloneqq n^{-\frac{g-1}{g}}$. Let $V_0 \coloneqq \{1,\dots,n\}$ and $V_1 \coloneqq \{n+1,\dots,2n\}$. 
	\begin{claim}\label{claim:5/2-girth-randomgraph}
		With high probability,
		\begin{enumerate}[label=(\roman*)]
			\item \label{it:5/2-girth-1} for each $i\in [0,1]$, every vertex $v\in V_i$ has $\deg_{V_{1-i}}(v) \ge (1-\eps)pn$;
			\item \label{it:5/2-girth-2} $G$ has at most $\eps n$ cycles of length less than $g/2$;
			\item \label{it:5/2-girth-3} for every subset $X \subseteq V(G)$ of size at least $\eps n$, the number of edges in $G[X]$ is in $(1\pm \eps)p\binom{|X|}{2}$;
			\item \label{it:5/2-girth-4} for every $A \subseteq V_0$ and $B \subseteq V_1$, each of size at least $\eps n$, the number of edges in the bipartite subgraph $G[A,B]$ is in $(1\pm \eps)p|A|\cdot|B|$.
		\end{enumerate}
	\end{claim}
	\begin{proofclaim}
		\ref{it:5/2-girth-1} follows by a direct application of the Chernoff bound (cf. \Cref{chernoff-multiplicative}), and \ref{it:5/2-girth-2} follows from Markov's inequality after observing that $G$ has less than $n^3p^3 + \dots + n^{g/2}p^{g/2} = \mathcal{O}(\sqrt{n})$ cycles of length at most $g/2$.
		
		In order to prove \ref{it:5/2-girth-3}, we observe that there are at most $2^{2n}$ subsets $X \subseteq V(G)$ of size at least $\eps n$ to consider; we fix any such set $X$. Let $\bY_X$ be the number of edges in $G[X]$; $\bY_X$ follows a binomial law of parameters $\left({|X| \choose 2}, p\right)$. Let $\mu \coloneqq \esp{\bY_X} = p{|X| \choose 2}$. By the Chernoff bound (cf. \Cref{chernoff-multiplicative}), we have
		\[\pr{|\bY_X - \mu| \ge \varepsilon\mu} \le 2e^{-\varepsilon^2\mu / 2} \le 2e^{-\varepsilon^4n^{1+1/g}/10} \ll 4^{-n}.\]
		By a union bound over all subsets $X$ to consider, we conclude that with high probability, every subset $X \subseteq V(G)$ of size at least $\varepsilon n$ induces a subgraph with $(1 \pm \varepsilon)p{|X| \choose 2}$ edges in $G$, which concludes the proof of \ref{it:5/2-girth-3}
		
		The proof of \ref{it:5/2-girth-4} is very similar to that of \ref{it:5/2-girth-3}; there are at most $2^n$ subsets $A$ to consider and at most $2^n$ subsets $B$ to consider, and for each fixed pair $(A,B)$ the probability that the number of edges in $G[A,B]$ deviates from the expectation is much smaller than $4^{-n}$.
	\end{proofclaim}
	
	Let us fix a realisation of $G$ with the properties described in the statement of \Cref{claim:5/2-girth-randomgraph}, and let $G'$ be obtained from $G$ by removing one arbitrary vertex from each cycle of length less than $g$ --- observe that $G'$ has between $(2-\eps)n$ and $2n$ vertices.  Let $W_0 \coloneqq V_0 \cap V(G')$ and $W_1 \coloneqq V_1\cap V(G')$; each of these two sets has size at least $(1-\eps)n$.
	Let $H$ be obtained from $G'$ by subdividing each edge $e\in \binom{W_0}{2}\cup \binom{W_1}{2}$ once --- we denote $U_0$ and $U_1$ the sets of vertices that arise from that subdivision --- and each edge $e\in W_0 \times W_1$ twice --- we denote $Z$ the set of vertices that arise from that subdivision.
	Then, by construction, $H$ is bipartite, has girth at least $g$, and minimum degree $2$. 
	We claim that $\fdom(H) \le 5/2 + 12\eps$ when $\eps \le 1/96$.
	
	Let $w(v) \coloneqq \frac{1}{2n}$ if $v\in W_0\cup W_1$, and $w(v) \coloneqq \frac{1}{2pn^2}$ if $v\in U_0 \cup U_1 \cup Z$.  
	Then 
	\begin{align*}
		w(V(H)) &= \frac{|V(G')|}{2n} + \frac{|E(G'[W_0])|+|E(G'[W_1])|}{2pn^2} + \frac{2|E(G'[W_0,W_1])|}{2pn^2} \\
		&\le 1 + \frac{2(1+\eps)p\binom{n}{2}}{2pn^2} + \frac{2(1+\eps)pn^2}{2pn^2} \qquad\qquad\text{by \ref{it:5/2-girth-3} and \ref{it:5/2-girth-4}}\\
		&< 5/2 + 3\eps/2.
	\end{align*} We now show that every dominating set $D$ of $H$ has $w(D) \ge 1-4\eps$; we rescale the weight function $w$ at the end of the proof by a factor $\frac{1}{1-4\eps}$ in order to obtain a fractional bottleneck of $H$.
	Let $D$ be a dominating set of $H$, and let $t_0 \coloneqq |D \cap W_0|$ and $t_1\coloneqq |D\cap W_1|$.
	Every vertex $u\in U_i$ ($i \in \{0,1\}$) whose neighbourhood is contained in $W_i \setminus D$ must belong to $D$, hence 
	\[|D \cap U_i| \ge |E(G'[W_i\setminus D])|.\]
	
	\begin{claim}\label{claim:5/2-girth}
		For $i \in \{0,1\}$, $w(D \cap (U_i \cup W_i)) \ge \frac{1}{4} + (t_i/2n)^2 - \eps$
	\end{claim}
	\begin{proofclaim}
		We prove the claim for $i = 0$, as the case $i=1$ is symmetric.
		
		If $t_0 > (1-2\eps)n$, then $w(D \cap W_0) \ge \frac{(1 - 2\eps)n}{2n} = \frac{1}{2} - \varepsilon \ge \frac{1}{4} + (t_0/2n)^2 - \varepsilon$.
		
		Otherwise, if $t_0 \le (1-2\eps)n$, then $|W_0 \setminus D| \ge \eps n$ and by \ref{it:5/2-girth-3} we have
		\begin{align*}
			|D \cap U_0| &\ge |E(G'[W_0 \setminus D])| \ge (1-\eps)p{|W_0| - t_0 \choose 2} \ge \frac{(1-2\eps)}{2}p((1-\eps)n - t_0)^2\\
			&\ge \frac{(1-2\eps)}{2}p((1-\eps)^2n^2 + t_0^2 - 2nt_0) \ge \frac{(1 - 4\eps)pn^2}{2} + \frac{(1-2\eps)pt_0^2}{2} - (1-2\eps)pt_0n,
		\end{align*}
		and therefore
		\begin{align*}
			w(D \cap U_0) + w(D \cap W_0) &\ge \frac{(1-4\eps)}{4} + \frac{(1-2\eps)t_0^2}{4n^2} - \frac{(1-2\eps)t_0}{2n} + \frac{t_0}{2n}\\
			&\ge \frac{1}{4} - \eps + \frac{t_0^2}{4n^2} - \frac{2\eps t_0^2}{4n^2} + \frac{\eps t_0}{n}\\
			&\ge \frac{1}{4} + (t_0/2n)^2 - \varepsilon,
		\end{align*}
	where we have used that $0 \le (t_0/n) \le 1$ and that $x - \frac{x^2}{2} \ge 0$ for all $0 \le x \le 1$.
	\end{proofclaim}
	
	For every $e\in E(G'[W_0,W_1])$, if $e\not\subseteq D$, then at least one of the two vertices in $Z$ that arise from $e$ belong to $D$. Hence 
	\[|D\cap Z| \ge |E(G'[W_0,W_1])\setminus E(G'[W_0\cap D, W_1 \cap D])|.\]

	We consider two possible cases and show that $w(D) \ge 1- 4\eps$ in both cases.
	\begin{itemize}
		\item If $t_0 \le \eps n$ (the case $t_1 \le \eps n$ is symmetric), then $|W_0 \setminus D| \ge (1-2\eps)n$, thus $|D \cap Z| \ge |E(G'[W_0\setminus D, W_1])| \ge (1-2\eps)(1-\eps)^2pn^2 \ge (1 - 4\eps)pn^2$ by \ref{it:5/2-girth-4}, and therefore by using \Cref{claim:5/2-girth} we obtain
		\begin{align*}
			w(D) &= w(D \cap (U_0 \cup W_0)) + w(D \cap (U_1 \cup W_1)) + w(D \cap Z)\\
			&\ge 2\pth{\frac{1}{4} - \eps} + \frac{(1-4\eps)pn^2}{2pn^2} \ge 1 - 4\eps.
		\end{align*}
		\item If $t_0 \ge \eps n$ and $t_1\ge \eps n$, then $|D \cap Z| \ge (1-\eps)^3pn^2 - (1+\eps)pt_0t_1$ by \ref{it:5/2-girth-4}, therefore
		\begin{align*}
			w(D) &= w(D \cap (U_0 \cup W_0)) + w(D \cap (U_1 \cup W_1)) + w(D \cap Z)\\
			&\ge \frac{1}{2} + (t_0/2n)^2 + (t_1/2n)^2 -2\eps + \frac{1 - 3\eps}{2} - \frac{(1+\eps)(t_0/n)(t_1/n)}{2}\\
			&\ge 1 + \pth{\frac{t_0}{2n} - \frac{t_1}{2n}}^2 - \frac{7\eps}{2} - \frac{\eps(t_0/n)(t_1/n)}{2} \ge 1 - 4\eps.
		\end{align*}
	\end{itemize}
	Therefore $\frac{w}{1-4\eps}$ is a fractional bottleneck of $H$, of weight at most $\frac{5/2 + \frac{3}{2}\eps}{1-4\eps} \le \frac{5}{2} + 12\eps$ when $\eps \le \frac{1}{96}$.
\end{proof}

\section{Lower bound on \texorpdfstring{$\fdom$}{FDOM} of graphs with minimum degree 2}
\label{sec:lower-bounds}

\subsection{Preliminary definitions and notations}

Let $f\colon V(G) \to [0,1]$ be a \emph{demand function} on $G$.
For every real number $r$ such that $0<r\le 1$, an \emph{$f$-dominating $r$-colouring} of $G$ is a random subset $\bD \subseteq V(G)$ such that for every vertex $v\in V(G)$,
    \[
        \pr{N[v] \cap \bD \neq \emptyset} \ge f(v)  \mbox{ and } \pr{v\in \bD} = r.
    \] 
When $f(v)=1$ for every vertex $v\in V(G)$, we simply refer to it as a \emph{dominating $r$-colouring}.
Let $p,q$ be integers such that $0<q \le p$. An \emph{$f$-dominating $(p:q)$-colouring} of $G$ is a mapping $\phi\colon V(G) \to \binom{[p]}{q}$ such that $|\phi(N[v])| \ge f(v)p$ for every vertex $v\in V(G)$. 
We will also refer to them as \emph{partially dominating colourings} as they can be understood as generalisations of dominating colourings where each vertex $v$ is ``partially dominated'', i.e. $v$ only has to see a proportion $f(v)$ of the total number of colours in its closed neighbourhood.

\subsection{Technical tools}

The two notions of partially dominating colourings defined above are closely related.
\begin{lemma}
    \label{lem:set-to-fraction}
    Let $G$ be a graph, $f$ a demand function on $G$, and $q,p$ integers such that $0< q\le p$. 
    If there exists an $f$-dominating $(p:q)$-colouring $\phi$ of $G$, then there exists an $f$-dominating $q/p$-colouring of $G$.
\end{lemma}

\begin{proof}
    Let $i$ be drawn uniformly at random from $[p]$, and let $\bD \coloneqq \phi^{-1}(\{i\})$. For every vertex $v$, we have $\pr{v\in\bD} = |\phi(v)|/p = q/p$, and $\pr{N[v]\cap \bD\neq \emptyset} = \frac{|\phi(N[v])|}{p} \ge f(v)$.
    We conclude that $\bD$ is the desired $f$-fractional $q/p$-colouring of $G$.
\end{proof}

In the definition of a partially dominating $r$-colouring, we require equality for the probability $r$ that a vertex is in the random set. This equality helps us inductively build partially dominating $r$-colourings but can be annoying to ensure. 
The following observation shows that, when building a partially dominating $r$-colouring, we may replace that equality constraint with an inequality.
Intuitively, domination constraints cannot be violated if we add additional colours to vertices with fewer colours so that every vertex is assigned a fraction $r=q/p$ of the total number of colours.

\begin{lemma}
    \label{lem:weakdomination}
    Let $G$ be a graph, $f$ a demand function on $G$, and $r$ a real number such that $0<r\le 1$.
    If there exists $\bD$ such that $\pr{N[v] \cap \bD \neq \emptyset} \ge f(v)$ and $\pr{v\in \bD} \le r$ for every vertex $v\in V(G)$,
    then there exists an $f$-dominating $r$-colouring $\bD'$ of $G$.
\end{lemma}

\begin{proof}
    We construct a random set $\bX$ by taking each vertex $v\notin \bD$ independently with probability $\frac{r - \pr{v\in \bD}}{1-\pr{v\in \bD}}$.
    Let $\bD'$ be the union of $\bD$ and $\bX$.
    By construction, $\bD \subseteq \bD'$, so we have $\pr{N[v] \cap \bD' \neq \emptyset} \ge f(v)$ for every vertex $v\in V(G)$.
    Moreover,
    \[
        \pr{v\in \bD'} = \pr{v\in \bD} + \pr{v\in \bX} 
        = \pr{v\in \bD} + \pth{1-\pr{v\in \bD}} \frac{r - \pr{v\in \bD}}{1-\pr{v\in \bD}}
        = r.
        \qedhere
    \]
\end{proof}

Consider a graph with two blocks, i.e. maximal $2$-connected subgraphs, glued on a common cut-vertex. If each block has a colouring with the same number of colours on each vertex, then we can combine the two colourings by permuting the colours to match them on the cut-vertex and get a colouring of the whole graph. 
Moreover, we can choose a permutation that also maximises the proportion of colours appearing in the neighbourhood of the cut-vertex.
As a consequence, two partially dominating colourings, one for each block, where the only vertex that is partially dominated is the cut-vertex, can combine to form a dominating colouring of the whole graph.
Formally, we prove the following more general statement.

\begin{lemma}
\label{lem:gluing}
    Let $G$ be a graph, $(G_0,G_1)$ a separation of order $1$ of $G$, and $v\in V(G)$ such that $V(G_0\cap G_1) = \{v\}$.
    Let $f_0,f_1$ be demand functions on $G_0$ and $G_1$ respectively. 
    Let $r$ be a real number such that $0<r\le 1$.
    If every $G_i$ has an $f_i$-dominating $r$-colouring $\bD_i$, then $G$ has an $f$-dominating $r$-colouring $\bD$, where
    \[ \begin{cases}
        f(v) \coloneqq \min \left\{1,  f_0(v)+f_1(v)-r\right\} &\mbox{and}\\
        f(w) \coloneqq f_i(w) &\mbox{if $w\in V(G_i) \setminus \{v\}$.}
    \end{cases}\]
\end{lemma}

\begin{proof}
    For every $i\in \{0,1\}$, let $A_i$ be the event $v\in \bD_i$, $B_i$ the event $v\notin \bD_i$ and $N_{G_i}(v) \cap \bD_i \neq \emptyset$, and $C_i$ the event $N_{G_i}[v]\cap \bD_i = \emptyset$ --- observe that the events $A_i, B_i, C_i$ form a partition of the probability space.
    Moreover, $(\pr{A_0}+\pr{B_0})+\pr{B_1} = \pr{N_{G_0}[v] \cap \bD_0 \ne \emptyset} + (\pr{N_{G_1}[v] \cap \bD_1 \ne \emptyset}-\pr{v\in \bD_1}) \ge f_0(v)+f_1(v)-r$.
    
    For every $X\in \{A,B,C\}$ and every $i\in\{0,1\}$, we let $\bD_i^X$ be $\bD_i$ conditioned on the event $X_i$.
    Observe that
    \begin{align*}
        \pr{N_{G_i}[v]\cap \bD_i^A\neq \emptyset} & = \pr{N_{G_i}[v]\cap \bD_i\neq \emptyset \mid v\in \bD_i} = 1,\\
        \pr{N_{G_i}[v]\cap \bD_i^B\neq \emptyset} & = \pr{N_{G_i}[v]\cap \bD_i\neq \emptyset \mid v\notin \bD_i \et N_{G_i}(v)\cap \bD_i \neq \emptyset} = 1, \mbox{ and}\\
        \pr{N_{G_i}[v]\cap \bD_i^C\neq \emptyset} & = \pr{N_{G_i}[v]\cap \bD_i\neq \emptyset \mid N_{G_i}[v]\cap \bD_i = \emptyset} = 0.
    \end{align*}
    
    \paragraph{Case 1.} Suppose that $\pr{B_0}+\pr{B_1} \ge 1-r$.
    Observe that $r+\pr{C_0}+\pr{C_1} = r + (1-\pr{A_0}-\pr{B_0}) + (1-\pr{A_1}-\pr{B_1}) = 2 - r-\pr{B_0}-\pr{B_1} \le 1$.
    Let 
    \[\bD \coloneqq \begin{cases}
        \bD_0^{A} \cup \bD_1^{A} & \mbox{with probability $r$,}\\
        \bD_0^{C} \cup \bD_1^{B} & \mbox{with probability $\pr{C_0}$,}\\
        \bD_0^{B} \cup \bD_1^{C} & \mbox{with probability $\pr{C_1}$, and}\\
        \bD_0^{B} \cup \bD_1^{B} & \mbox{otherwise.}\\
        \end{cases}\]
    By construction, $\pr{v\in \bD}=r$.
    Moreover,
    \begin{align*}
    \pr{N_G[v]\cap \bD \neq \emptyset} 
    &= r\,\pr{N_G[v]\cap (\bD_0^A\cup \bD_1^A)\neq \emptyset}
    +\pr{C_0}\pr{N_G[v]\cap (\bD_0^C\cup \bD_1^B)\neq \emptyset}\\
    &~~~+ \pr{C_1}\pr{N_G[v]\cap (\bD_0^B\cup \bD_1^C)\neq \emptyset}\\
    &~~~+(1-r-\pr{C_0}-\pr{C_1})\,\pr{N_G[v]\cap (\bD_0^B\cup \bD_1^B)\neq \emptyset}\\
    &\ge r\,\pr{N_{G_0}[v]\cap \bD_0^A\neq \emptyset}
    + \pr{C_0}\pr{N_{G_1}[v]\cap \bD_1^B\neq \emptyset}\\
    &~~~+ \pr{C_1}\pr{N_{G_0}[v]\cap \bD_0^B\neq \emptyset}\\
    &~~~+ (1-r-\pr{C_0}-\pr{C_1})\,\pr{N_{G_0}[v]\cap \bD_0^B\neq \emptyset}\\
    &\ge r + \pr{C_0} + \pr{C_1} + 1-r-\pr{C_0}-\pr{C_1}\\
    &\ge 1 \ge f(v).
    \end{align*}
    Observe that, for each $i\in \{0,1\}$, by construction, $\bD \cap V(G_i)$ follows the same probability distribution as $\bD_i$. Hence $\bD$ induces in particular an $f_i$-dominating $r$-colouring of $G_i$. This ends the proof of Case~1.

    \paragraph{Case 2.} Suppose on the contrary that $\pr{B_0}+\pr{B_1} < 1-r$.
    Let 
    \[\bD \coloneqq \begin{cases}
        \bD_0^{A} \cup \bD_1^{A} & \mbox{with probability $r$,}\\
        \bD_0^{B} \cup \bD_1^{C} & \mbox{with probability $\pr{B_0}$,}\\
        \bD_0^{C} \cup \bD_1^{B} & \mbox{with probability $\pr{B_1}$, and}\\
        \bD_0^{C} \cup \bD_1^{C} & \mbox{otherwise.}\\
        \end{cases}\]
    By construction, $\pr{v\in \bD}=r$. Moreover, a computation similar to that performed in Case~1 yields that
    \begin{align*}
    \pr{N_G[v]\cap \bD \neq \emptyset} 
    &\ge r\,\pr{N_{G_0}[v]\cap \bD_0^A\neq \emptyset}
    + \pr{B_0}\pr{N_{G_0}[v]\cap \bD_0^B\neq \emptyset}\\
    &~~~+ \pr{B_1}\pr{N_{G_1}[v]\cap \bD_1^B\neq \emptyset}\\
    &~~~+ (1-r-\pr{B_0}-\pr{B_1})\,\pr{N_G[v]\cap (\bD_0^C\cup \bD_1^C)\neq \emptyset}\\
    &\ge r + \pr{B_0} + \pr{B_1}\\
    &\ge f_0(v) + f_1(v) - r = f(v).
    \end{align*}
    As in Case~1, for each $i\in \{0,1\}$, $\bD \cap V(G_i)$ follows the same probability distribution as $\bD_i$. Hence $\bD$ induces in particular an $f_i$-dominating $r$-colouring of $G_i$. This ends the proof of Case~2.
    
\end{proof}

Consider a separation $(G',H)$ of $G$ of order 2 with $V(G'\cap H) = \{u,v\}$. Let $\phi'$ be a fractional dominating colouring of $G'$. Like the previous case, if $H$ is colourable, then we look for a way to combine the colourings of $G'$ and $H$ to get one fractional dominating colouring of the whole graph. However, this combination requires that the number of colours in common between $u$ and $v$ is the same for both colourings. Let us assume that $|\phi'(u)\cap\phi'(v)| = \alpha |\phi'(u)|$, for some $\alpha \in [0,1]$. 
If $H$ has two colourings $\phi_0$ and $\phi_1$ such that $\phi_0(u) \cap \phi_0(v) = \emptyset$ and $\phi_1(u)=\phi_1(v)$, then we can produce a colouring $\phi$ of $H$ such that $|\phi(u)\cap\phi(v)|=\alpha|\phi(u)|$; this is obtained by letting a proportion $\alpha$ of the colours of $\phi$ be copies of colours of $\phi_1$, and a proportion $(1-\alpha)$ be copies of colours of $\phi_0$.
Then, one can obtain $\widetilde{\phi}'$ from $\phi'$ by taking $|\phi(u)|$ copies of each colour class, and similarly one obtains $\widetilde{\phi}$ from $\phi$ by taking $|\phi'(u)|$ copies of each colour class.
Up to colour permutation, we now have $\widetilde{\phi}(u)=\widetilde{\phi}'(u)$ and $\widetilde{\phi}(v)=\widetilde{\phi}'(v)$, so $\widetilde{\phi}$ and $\widetilde{\phi}'$ can be combined together. We formalise this argument in the following lemma. 

\begin{lemma}
    \label{lem:extension}
    Let $G$ be a graph, $(G',H)$ a separation of order $2$ of $G$, and $u,v\in V(G)$ such that $V(G'\cap H) = \{u,v\}$.
    Let $f',h$ be demand functions on $G'$ and $H$ respectively.
    Let $r$ be a real number such that $0<r<1/2$.
    If $G'$ has an $f'$-dominating $r$-colouring $\bD'$ and $H$ has two $h$-dominating $r$-colourings $\bD_0,\bD_1$ such that 
    $u$ and $v$ belong to disjoint subsets of realisations of $\bD_0$ and to the exact same realisations of $\bD_1$,
    then $G$ has an $f$-dominating $r$-colouring $\bD$ where $f(w) = f'(w)$ if $w\in V(G')$ and $f(w)=h(w)$ otherwise.
\end{lemma}

\begin{proof}
    Let $\alpha \coloneqq \pr{u\in \bD' \mid v\in \bD'}$ and $\beta =\pr{u\notin \bD' \mid v\notin \bD'}$.
    Observe that
    \begin{align*}
    \pr{u\in \bD'\et v\in \bD'} &= \pr{u\in\bD'\mid v\in \bD'}\pr{v\in\bD'} &&= \alpha r,\\
    \pr{u\in \bD'\et v\notin \bD'} &=\pr{u\in\bD'}-\pr{u\in \bD' \et v\in \bD'} &&= (1-\alpha)r,\\
    \pr{u\notin \bD'\et v\in \bD'} &= \pth{1-\pr{u\in\bD'\mid v\in \bD'}}\pr{v\in\bD'} &&= (1-\alpha)r, \\
    \pr{u\notin \bD'\et v\notin \bD'} &= \pr{u\notin\bD'}-\pr{u\in \bD' \et v\notin \bD'} &&= 1-(2-\alpha)r,\\
    \beta &= \frac{\pr{u\notin \bD'\et v\notin \bD'}}{\pr{v\notin \bD'}} &&= \frac{1-(2-\alpha)r}{1-r}, \mbox{ and}\\
    \beta - \alpha &= \frac{1-(2-\alpha)r - \alpha(1-r)}{1-r} &&= (1-\alpha)\frac {1-2r}{1-r}.
    \end{align*}
    Since $r<1/2$, $\beta-\alpha\ge 0$ and $0\le \alpha/\beta \le 1$.
    Let $\bi \gets Ber(\alpha/\beta)$ be a random $(0,1)$-valued variable such that $\pr{\bi=1}=\alpha/\beta$.
    
    For every $i\in \{0,1\}$, let $\bD_i^u,\bD_i^v$ be $\bD_i$ conditioned on the events $u\in \bD_i$ and $v\in \bD_i$ respectively, $\bD_i^{\neg u},\bD_i^{\neg v}$ be $\bD_i$ conditioned on the events $u\notin \bD_i$ and $v\notin \bD_i$ respectively, and $\bD_i^{\neg u,v}$ be $\bD_i$ conditioned on the event ($u\notin \bD_i$ and $v\notin \bD_i$).
    By hypothesis, 
    $\bD_1^u = \bD_1^v$ and $\bD_1^{\neg u} = \bD_1^{\neg v}$.
    Now, let
    \[\bD \coloneqq \begin{cases}
    \bD' \cup \bD_1^{u} & \mbox{if $u\in \bD'$ and $v\in \bD'$,}\\
    \bD' \cup \bD_0^{u} & \mbox{if $u\in \bD'$ and $v\notin \bD'$,}\\
    \bD' \cup \bD_0^{v} & \mbox{if $u\notin \bD'$ and $v\in \bD'$,}\\
    \bD' \cup \bD_1^{\neg u} & \mbox{if $u\notin \bD'$ and $v\notin \bD'$ and $\bi = 1$, and}\\
    \bD' \cup \bD_0^{\neg u,v} & \mbox{if $u\notin \bD'$ and $v\notin \bD'$ and $\bi = 0$.}\\
    \end{cases}\]
    By definition of $\bD$, we have
    \begin{align*}
    \pr{u\in \bD} &= \pr{u\in \bD' \et v\in\bD'} + \pr{u\in \bD' \et v\notin \bD'} = \pr{u\in \bD'} = r;\\
    \pr{v\in \bD} &= \pr{u\in \bD' \et v\in\bD'} + \pr{u\notin \bD' \et v\in \bD'} = \pr{v\in \bD'} = r;
    \end{align*}
    and
    \begin{align*}
    \pr{N_G[u] \cap \bD\neq \emptyset} &\ge \pr{N_{G'}[u]\cap \bD'\neq \emptyset} \ge f'(u);\\
    \pr{N_G[v] \cap \bD\neq \emptyset} &\ge \pr{N_{G'}[v]\cap \bD'\neq \emptyset} \ge f'(v).
    \end{align*}
    For every vertex $w\in V(G') \setminus\{u,v\}$, 
    \begin{align*}
    \pr{w\in \bD} &= \pr{w\in\bD'} =r, \mbox{ and}\\
    \pr{N_G[w]\cap \bD\neq \emptyset} &\ge \pr{N_{G'}[w]\cap \bD'\neq \emptyset} \ge f'(w).
    \end{align*}
    
    By hypothesis,
    the events $u\in \bD_0$, $v\in\bD_0$, and ($u\notin \bD_0$ and $v\notin \bD_0$) form a partition of the probability space.
    Therefore,
    \[
    \pr{u\notin \bD_0 \et v \notin \bD_0} = 1-\pr{u\in \bD_0}-\pr{v\in \bD_0} = 1-2r.
    \]
    For every subset of vertices $W\in V(H)\setminus\{u,v\}$, we have
    \begin{align*}
    \pr{W\cap \bD \neq \emptyset} 
    &= \pr{W\cap \bD_1^{u} \neq \emptyset}\pr{u\in \bD'\et v\in \bD'}
    + \pr{W\cap \bD_0^{u} \neq \emptyset}\pr{u\in \bD'\et v\notin \bD'}\\ 
    &~~~+ \pr{W\cap \bD_0^{v} \neq \emptyset}\pr{u\notin \bD'\et v\in \bD'}\\
    &~~~+ \pr{W\cap \bD_1^{\neg u} \neq \emptyset}\pr{u\notin \bD'\et v\notin \bD' \et \bi = 1}\\
    &~~~+ \pr{W\cap \bD_0^{\neg u,v} \neq \emptyset}\pr{u\notin \bD'\et v\notin \bD' \et \bi = 0}\\
    &= \pr{W\cap \bD_1^{u} \neq \emptyset}\alpha r
    + \pr{W\cap \bD_0^{u} \neq \emptyset}(1-\alpha)r 
    + \pr{W\cap \bD_0^{v} \neq \emptyset}(1-\alpha)r\\
    &~~~+ \pr{W\cap \bD_1^{\neg u} \neq \emptyset}\beta(1-r)\frac{\alpha}{\beta}
    + \pr{W\cap \bD_0^{\neg u,v} \neq \emptyset}\beta(1-r) \pth{1-\frac{\alpha}{\beta}}\\
    &= \pr{W\cap \bD_1 \neq \emptyset \et u\in \bD_1}\alpha
    + \pr{W\cap \bD_0 \neq \emptyset \et u\in \bD_0}(1-\alpha)\\ 
    &~~~+ \pr{W\cap \bD_0 \neq \emptyset \et v\in \bD_0}(1-\alpha)
    + \pr{W\cap \bD_1 \neq \emptyset \et u\notin \bD_1}\alpha\\
    &~~~+ \pr{W\cap \bD_0 \neq \emptyset\et u\notin \bD_0 \et v\notin \bD_1}\frac{1-r}{1-2r}(\beta-\alpha)\\
    &= (1-\alpha)\pr{W\cap \bD_0 \neq \emptyset} + \alpha\pr{W\cap \bD_1 \neq \emptyset}.
    \end{align*}
    For every $w\in V(H)\setminus\{u,v\}$, if we take $W =\{w\}$, we obtain that
    \[
    \pr{w\in\bD} = (1-\alpha)\pr{w\in\bD_0} + \alpha\pr{w\in\bD_1} = (1-\alpha)r +\alpha r = r,
    \]
    and if we take $W = N_G[w]$, we obtain that
    \begin{align*}
    \pr{N_G[w]\cap\bD\neq\emptyset} &= (1-\alpha)\pr{N_{H}[w]\cap \bD_0 \neq \emptyset} + \alpha\pr{N_{H}[w]\cap \bD_1 \neq \emptyset}\\
    &\ge (1-\alpha)h(w) + \alpha h(w) = h(w).    
    \end{align*}
    This ends the proof.
\end{proof}

\subsection{Proof of \texorpdfstring{\Cref{thm:5/2}}{Theorem 2}}

The proof of \Cref{thm:5/2} proceeds by induction (by considering a minimal counterexample). Because of the excluded family of graphs $\B$, we need to make sure at each inductive step that we do not reach a graph in $\B$, which is ensured with a few reducible configurations that progressively exclude graphs from $\B$ as possible induced subgraphs.
The general case considers a suspended path $P$ of length at least $4$ in $G$ (if it exists) and applies induction on the graph $G' \coloneqq G\setminus P$ obtained by removing $P$ from $G$. 
The following lemma shows how to extend a dominating fractional colouring of $G'$ obtained inductively to $G$. 
We note that the requirement that the length of $P$ is at least $4$ is needed to ensure that we obtain a lower bound of (at least) $5/2$ for $\fdom(G)$, provided that $\fdom(G')\ge 5/2$.

\begin{lemma}
    \label{lem:strong-pathlemma}
    Let $G$ be a graph, and let $P$ be a suspended path of $G$ of length $\ell \ge 1$. Let $G'$ be obtained by removing the internal vertices of $P$ from $G$. Then, writing $k\coloneqq \ceil{\ell/3}$, one has
    \[\fdom(G)\ge \min\;\left\{ \frac{3k-1}{k}, \fdom(G')\right\}.\]
\end{lemma}

\begin{proof}
    Let $\phi$ be a dominating $g$-colouring of $G'$, with $g\coloneqq \min\;\left\{ \frac{3k-1}{k}, \fdom(G')\right\}$.
    Let $V(P)\coloneqq\{u_0, u_1, \ldots, u_\ell\}$ with $u_0$ and $u_\ell$ the two endpoints of $P$.
    We will show that $\phi$ extends to a dominating $g$-colouring of $G$ with an application of \Cref{lem:extension}. To do so, we exhibit two 
    $h$-dominating $(3k-1:k)$-colourings $\phi_0$ and $\phi_1$ of $P$ such that $\phi_1(u_0)=\phi_1(u_\ell)$ and $\phi_0(u_0)\cap \phi_1(u_\ell)=\emptyset$, where $h(w)=1$ if $w$ is an internal vertex of $P$, and $h(u_0)=h(u_\ell)=0$.

    Let $\lambda, \mu, \nu$ be three different $h$-dominating $(3k-1:k)$-colourings of $P$ defined as follows.
    For every $i\in \{0,1,\dots,\ell\}$
    \begin{align*}
    \lambda(u_i) &\coloneqq \begin{cases}
        \{1, \ldots, k\} &\mbox{ if } i \bmod 3 = 0 ,\\
        \{k+1, \ldots, 2k\} &\mbox{ if } i \bmod 3 = 1 ,\\
        \{2k, \ldots, 3k-1\} &\mbox{ if } i \bmod 3 = 2 ;
    \end{cases} \\
    \mu(u_i) &\coloneqq \left\{\big((ik) \bmod (3k-1)\big)+1, \ldots, \left(\big((i+1)k-1\big) \bmod (3k-1)\right)+1 \right\}; \\
    \nu(u_i) &\coloneqq \begin{cases}
        \{j+1, \ldots, k+j\} &\mbox{ if } i = 3j,\\
        \{1, \ldots, j+1\} \cup \{k+j+1, \ldots, 2k-1\} &\mbox{ if }i = 3j+1 ,\\
        \{2k, \ldots, 3k-1\} &\mbox{ if } i = 3j +2.\\
    \end{cases}
    \end{align*}
    For each $\phi \in \{\lambda, \mu, \nu\}$, it is straightforward to check that $\phi$ is indeed an $h$-dominating $(3k-1:k)$-colouring of $P$, by observing that $|\phi(u)|=k$ for every $u\in V(P)$ and $\phi(u_{i-1}) \cup \phi(u_i) \cup \phi(u_{i+1}) = [3k-1]$ for every $0<i<\ell$.

    Since $\ceil{\ell/3} = k$, we have the three following cases.
    \begin{itemize}
        \item If $\ell = 3k-2$, then let $\phi_0 \coloneqq \lambda$ and $\phi_1 \coloneqq \nu$.
        Observe that $\lambda(u_0) \cap \lambda(u_{3k-2}) = \emptyset$ and $\nu(u_0) = [k] = \nu(u_{3k-2})$.
        \item If $\ell = 3k-1$, then let $\phi_0 \coloneqq \lambda$ and $\phi_1 \coloneqq \mu$.
        Observe that $\lambda(u_0) \cap \lambda(u_{3k-1}) = \emptyset$ and $\mu(u_0) = [k] = \mu(u_{3k-1})$.
        \item If $\ell = 3k$, then let $\phi_0 \coloneqq \mu$ and $\phi_1 \coloneqq \lambda$.
        Observe that $\mu(u_0) = [k]$, $\mu(u_{3k}) = \{k+1,\dots,2k\}$, and $\lambda(u_0) = [k] = \lambda(u_{3k-1})$.
    \end{itemize}
    Let $\bD_i$ be a random colour class of $\phi_i$ for each $i\in \{0,1\}$. By construction, those are $h$-dominating $\frac{k}{3k-1}$-colourings of $P$ such that 
    $u$ and $v$ belong to disjoint subsets of realisations of $\bD_0$ and to the exact same realisations of $\bD_1$.
    An application of \Cref{lem:extension} concludes the proof.
\end{proof}

This leads us to a base case where all suspended paths have length at most $3$, in which case we describe explicitly how to obtain directly a dominating $2/5$-colouring. By consideration of all the reducible configurations, the base case is as follows.

\begin{figure}[htbp]
    \centering
    \begin{subfigure}[t]{0.39\textwidth}
        \centering
            \begin{tikzpicture}[every node/.style={shape=circle,draw=black,fill,inner sep=0pt,minimum size=6pt}, every edge/.style={line width=1.25pt,draw,black}]
                \node[label=below:{$u$}] (left) {};
                \node[label=below:{$v$}, right =1cm of left] (mid) {};
                \node[label=below:{$w$}, right =1cm of mid] (right) {};
                \coordinate[left=0.5cm of left] (leftmid);
                \coordinate[above left=0.5cm of left] (leftup);
                \coordinate[below left=0.5cm of left] (leftdown);

                \draw[dashed, line width=1.5pt] (left) -- (leftup);
                \draw[dashed, line width=1.5pt] (left) -- (leftmid);
                \draw[dashed, line width=1.5pt] (left) -- (leftdown);

                \coordinate[right=0.5cm of right] (rightmid);
                \coordinate[above right=0.5cm of right] (rightup);
                \coordinate[below right=0.5cm of right] (rightdown);

                \draw[dashed, line width=1.5pt] (right) -- (rightup);
                \draw[dashed, line width=1.5pt] (right) -- (rightmid);
                \draw[dashed, line width=1.5pt] (right) -- (rightdown);
                
                \draw (mid) edge[draw=red] (right);
                \draw (left) edge[bend right,draw=red] (mid);
                \draw (left) edge[bend left, draw=red] (mid);
            \end{tikzpicture}
        \caption{A neighbourhood in $H$.}
        \label{fig:C5P3subdivision:base}
    \end{subfigure}\hfill
    \begin{subfigure}[t]{0.6\textwidth}
            \centering
            \begin{tikzpicture}[every node/.style={shape=circle,draw=black,fill,inner sep=0pt,minimum size=6pt}, every edge/.style={line width=1.25pt,draw,black}]
                \node[label=below:{$u$}] (left) {};
                \node[label=below:{$v$}, right =2cm of left] (right) {};
                \node[label=below:{$w$}, right =2cm of right] (end) {};
                \coordinate[] (mid) at($(left)!0.5!(right)$);
                \coordinate[] (midleft) at($(mid)!0.5!(left)$);
                \coordinate[] (midright) at($(mid)!0.5!(right)$);
                
                \node[draw=red, fill=red, below=0.5cm of midleft] (v1) {};
                \node[draw=red, fill=red, below =0.5cmof midright] (v2) {};
                \node[draw=red, fill=red, above =0.5cmof mid] (v3) {};
                \node[draw=red, fill=red] (v4) at($(right)!0.5!(end)$) {};

                \coordinate[left=0.5cm of left] (leftmid);
                \coordinate[above left=0.5cm of left] (leftup);
                \coordinate[below left=0.5cm of left] (leftdown);

                \draw[dashed, line width=1.5pt] (left) -- (leftup);
                \draw[dashed, line width=1.5pt] (left) -- (leftmid);
                \draw[dashed, line width=1.5pt] (left) -- (leftdown);

                \coordinate[right=0.5cm of end] (rightmid);
                \coordinate[above right=0.5cm of end] (rightup);
                \coordinate[below right=0.5cm of end] (rightdown);

                \draw[dashed, line width=1.5pt] (end) -- (rightup);
                \draw[dashed, line width=1.5pt] (end) -- (rightmid);
                \draw[dashed, line width=1.5pt] (end) -- (rightdown);

                \path[] (left) edge[draw=red] (v3);
                \path[] (v3) edge[draw=red] (right);
                \path[] (v2) edge[draw=red] (right);
                \path[] (v2) edge[draw=red] (v1);
                \path[] (left) edge[draw=red] (v1);
                \path[] (right) edge[draw=red] (v4);
                \path[] (v4) edge[draw=red] (end);
            \end{tikzpicture}
        \caption{The corresponding neighbourhood in $G$ after subdivision.}
        \label{fig:C5P3subdivision:C5}
    \end{subfigure}\hfill
%






                
    \caption{The subdivision rules from $H$ to $G$ in the statement of \Cref{lem:5/2}.}
    \label{fig:C5P3subdivision}
\end{figure}
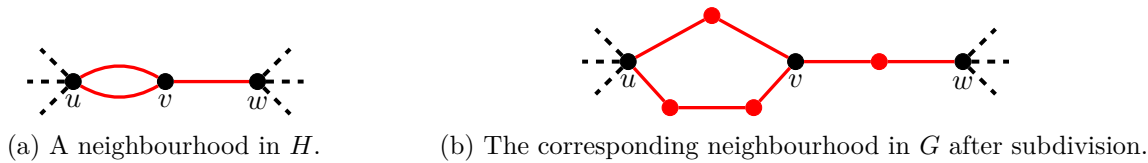

\begin{lemma}\label{lem:5/2}
    Let $H$ be a multigraph of minimum degree at least $3$ and multiplicity at most $2$. Let $G$ be obtained from $H$ by subdividing each simple edge into a suspended $2$-path and each double edge into a hammock.
    Then $\fdom(G) \ge 5/2.$
\end{lemma}

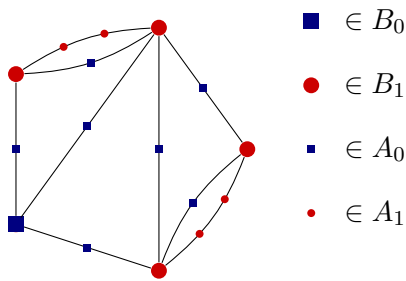
\begin{figure}[!ht]
    \centering
    \begin{tikzpicture}[every node/.style={shape=circle,fill=black!60!white,inner sep=0pt,minimum size=3pt}]

\def\R{1.7}
\def\n{5}

\foreach \i in {1,...,\n}{
	\node[fill=red!80!black, minimum size=6pt] (x\i) at (360*\i/\n : \R) {};
}

\node[fill=blue!50!black, minimum size=6pt, rectangle] at (x3) {};

\draw (x1) to[bend left=15] node[midway, fill=blue!50!black, rectangle] {} (x2);
\draw (x1) to[bend right=15] node[pos=0.67, fill=red!80!black] {}
            node[pos=0.33, fill=red!80!black] {} (x2) ;
            
\draw (x1) to node[midway, fill=blue!50!black, rectangle] {} (x3);
            
\draw (x4) to[bend left=15] node[midway, fill=blue!50!black, rectangle] {} (x5);
\draw (x4) to[bend right=15] node[pos=0.67, fill=red!80!black] {}
            node[pos=0.33, fill=red!80!black] {} (x5) ;
            
\draw (x2) to node[midway, fill=blue!50!black, rectangle] {} (x3);
            
\draw (x3) to node[midway, fill=blue!50!black, rectangle] {}(x4) ;

\draw (x1) to node[midway, fill=blue!50!black, rectangle] {} (x4);

\draw (x1) to node[midway, fill=blue!50!black, rectangle] {} (x5) ;

\node[fill=blue!50!black, minimum size=6pt, rectangle] (a) at (1.5*\R, \R) {};
\node[draw=none, fill=none] (a') at (2*\R, \R) {$\in B_0$};\node[fill=red!80!black, minimum size=6pt] (a) at (1.5*\R, \R/2) {};
\node[draw=none, fill=none] (a') at (2*\R, \R/2) {$\in B_1$};
\node[fill=blue!50!black, rectangle] (b) at (1.5*\R, 0) {};
\node[draw=none, fill=none] (b') at (2*\R, 0) {$\in A_0$};
\node[fill=red!80!black] (b) at (1.5*\R, -\R/2) {};
\node[draw=none, fill=none] (b') at (2*\R, -\R/2) {$\in A_1$};

\end{tikzpicture}
    \caption{A graph that falls in the base case of the induction.}
    \label{fig:hammock-graph}
\end{figure}

\begin{proof}
    Let $A$ be the set of $2$-vertices in $G$, i.e. those that arise from the subdivision, and $B$ the set of $3^+$-vertices in $G$, i.e. those from $V(H)$. 
    Let $A_0\subseteq A$ be the set of $2$-vertices that are part of a suspended $2$-path and $A_1 \coloneqq A\setminus A_0$ the $2$-vertices that are part of a suspended $3$-path. 
    Let $B_1 \subseteq B$ be the set of $3^+$-vertices that are endpoints of a suspended $3$-path and $B_0 \coloneqq B \setminus B_0$ those that are not. An illustration is given in \Cref{fig:hammock-graph}.
    
    Let us build a dominating $2/5$-colouring $\bD$ of $G$. 
    Let $\bB_0, \bB_1, \bA_0,\bA'_0,\bA_1, \bA'_1$ be random sets built as follow.
    \begin{enumerate}[label=(\roman*)]
    \item For each vertex in $B_0$, include it in $\bB_0$ with probability $2/5$, independently from the others.
    \item With probability $1/5$, set $\bB_1 \coloneqq \emptyset$; otherwise, for each vertex in $B_1$, include it in $\bB_1$ with probability $1/2$, independently from the others.
    \item For each vertex $v\in A_0$ such that $N(v)\cap (\bB_0\cup\bB_1) = \emptyset$, include it in $\bA_0$.
    \item For each vertex $v\in B_0\setminus \bB_0$ such that $N(v)\cap \bA_0 = \emptyset$, let $\mathbf{z} \gets \cU(N(v))$ be a uniformly random neighbour of $v$ (by definition of $B_0$, one has $\mathbf{z}\in A_0$), and include it in $\bA'_0$.
    \item For every suspended $3$-path $u\link v\link w\link x$ with $u\in B_1\setminus \bB_1$ and $x\in \bB_1$, include $v$ in $\bA_1$.
    \item For every suspended $3$-path $u\link v\link w\link x$ with $\{u,x\} \subseteq B_1\setminus \bB_1$, let $\mathbf{z}\gets \cU(\{v,w\})$ be chosen uniformly at random from $\{v,w\}$, and include it in $\bA'_1$.
    \end{enumerate}
    Finally, let $\bD = \bB_0\cup\bB_1\cup\bA_0\cup\bA'_0\cup\bA_1\cup\bA'_1$.

    First, we claim that $\bD$ is a dominating set of $G$.
    Every vertex of $A_0$ either is in $\bA_0$ or has a neighbour in $\bB_0\cup\bB_1$.
    Every vertex of $B_0$ either is in $\bB_0$ or has a neighbour in $\bA_0\cup\bA'_0$. 
    Let $v$ be a vertex of $A_1$ and $u\link v\link w\link x$ the suspended $3$-path containing $v$.
    If $u\in \bB_1$, then $v$ is dominated by $u$.
    If $u\notin \bB_1$ and $x\in \bB_1$, then $v\in \bA_1$.
    If $u\notin \bB_1$ and $x\notin\bB_1$, then either $v$ or $w$ is in $\bA'_1$.
    Therefore, $A_1$ is dominated by $\bD$.
    Let $u$ be a vertex of $B_1$ and $u\link v\link w\link x$ a suspended $3$-path containing $u$.
    Since there must exists a hammock between $u$ and $x$, they also share a suspended $2$-path $u\link y\link x$.
    If $u\in \bB_1$, then $u$ is dominated by $\bB_1$.
    If $u\notin \bB_1$ and $x\in \bB_1$, then $v\in \bA_1$ so $u$ is dominated by $v$.
    If $u,x\notin \bB_1$, then $y\in \bA_0$ so $u$ is dominated by $y$.
    Therefore, $B_1$ is dominated by $\bD$.

    Now, we show that every vertex is in $\bD$ with probability at most $2/5$. This concludes the proof with an application of \Cref{lem:weakdomination}.
    For every vertex $b_0\in B_0$,
    \[
    \pr{b_0\in \bD} = \pr{b_0\in \bB_0} = \frac25.
    \]
    For every vertex $b_1\in B_1$,
    \[
    \pr{b_1\in\bD} = \pr{b_1\in\bB_1} = \pr{\bB_1\neq\emptyset}\pr{b_1\in\bB_1\mid \bB_1\neq\emptyset} = \frac45 \cdot \frac12 = \frac25.
    \]
    Let $a_1\in A_1$ and $u\link a_1\link w\link x$ be the suspended $3$-path containing $a_1$. We have
    \begin{align*}
    \pr{a_1\in \bA_1} &= \pr{u\notin \bB_1 \et x\in \bB_1} = \frac45\cdot\frac12\cdot\frac12 = \frac15;\\
    \pr{a_1\in \bA'_1} &= \frac12\pr{u\notin \bB_1 \et x\notin \bB_1}\\
    &= \frac12\left(\pr{\bB_1=\emptyset}+\pr{\bB_1\neq\emptyset}\pr{\{u,x\} \cap  \bB_1 = \emptyset \mid \bB_1\neq\emptyset}\right) \\
    &= \frac12\left(\frac15 + \frac45\cdot\frac12\cdot\frac12\right) = \frac15 \mbox{; and so}\\
    \pr{a_1\in \bD} &= \pr{a_1\in \bA_1} + \pr{a_1\in \bA'_1} = \frac15 + \frac15 = \frac25. 
    \end{align*}
    Let $a_0\in A_0$; write $N(a_0) = \{b,b'\}$, and suppose without loss of generality that $\deg(b)\le \deg(b')$.
    If $\{b,b'\}\subseteq B_1$, then
    \[
    \pr{a_0\in\bD} = \pr{a_0\in \bA_0} = \pr{\{b,b'\}\cap \bB_1 = \emptyset} = \frac25 \mbox{, similar to the computation above.}
    \]
    Otherwise, assume without loss of generality that $b \notin B_1$; 
    \begin{align*}
    \pr{a_0\in \bA_0} &= \pr{\{b,b'\} \cap \big(\bB_0\cup\bB_1\big) = \emptyset}\\
    &= \pr{b\notin \bB_0}\pr{b'\notin \bB_0\cup\bB_1}\\
    &= \left(1-\frac25\right)\left(1-\frac25\right) = \frac{9}{25}; \quad \mbox{and}\\
    \pr{a_0\in \bA'_0} &\le 2\cdot\frac{\pr{b\notin \bB_0 \et N(b)\cap \bA_0 = \emptyset}}{\deg(b)} &
    \begin{array}{l}
    \mbox{since at worst  $\{b,b'\} \subseteq B_0$}\\
    \mbox{and }\deg(b)\le\deg(b');
    \end{array}\!\!\\
    & \le \frac2{\deg(b)}\left(1-\frac25\right)\frac45\left(\frac{1}{2}\right)^{\deg(b)} & \begin{array}{l}
    \mbox{since at worst all $3^+$-vertices}\\
    \mbox{at distance 2 from $b$ are in $\bB_1$;}
    \end{array}\!\!\\
    & \le \frac2{3}\left(1-\frac25\right)\frac45\left(\frac{1}{2}\right)^3 &\mbox{ since }\deg(b)\ge 3;\\
    & \le \frac{1}{25};
    \end{align*}
    and so 
    \begin{align*}
    \pr{a_0\in\bD} &= \pr{a_0\in \bA_0} + \pr{a_0\in \bA'_0} \le \frac{9}{25} + \frac{1}{25} = \frac25.
    \end{align*}
    This ends the proof.
\end{proof}

In order to prove \Cref{thm:5/2}, we prove a stronger statement that deals with vertices of degree 1 and reduce the remaining cases.

\begin{thm}\label{thm:5/$2$-technical}
    Let $\B$ be the family of bad graphs depicted in \Cref{fig:badgraphs}. 
    Let $G\notin \B$ be a connected graph on at least $2$ vertices, and let $f_G\colon V(G) \to [0,1]$ be the demand function such that 
    \[f_G(v) = \begin{cases}
        4/5 & \mbox{if $\deg(v)=1$}; \\
        1 & \mbox{otherwise}.
    \end{cases}\]
    Then $G$ has an $f_G$-dominating $2/5$-colouring.
\end{thm}

\begin{proof}

    Let us assume for the sake of contradiction that there exists a counterexample $\Gamma$ to the statement of \Cref{thm:5/$2$-technical} minimizing the number of edges.

    We first show that $\Gamma$ has no cut-vertex.
    To that end, we introduce the notion of quasi-dominating colouring. Given a graph $G$ and a vertex $v_0\in V(G)$, a $(p:q)$-colouring $\phi$ of $G$ is \emph{$v_0$-quasi-dominating} if $|\phi(N[v_0])|=p-1$ and $|\phi(N[v])|=p$ for every $v\neq v_0$.
    It turns out that for every graph $\Gamma \in \B \setminus \{C_4\}$ and every vertex $v_0\in V(\Gamma)$, there is a $v_0$-quasi-dominating $(5:2)$-colouring of $\Gamma$; such a colouring is depicted in \Cref{fig:quasi-dominating-colouring}.

    \begin{claim}
        \label{claim:5/$2$-cutvertex}
        $\Gamma$ is $2$-connected.
    \end{claim}

    \begin{proofclaim}
        As depicted in \Cref{fig:quasi-dominating-colouring}, every graph $G\in \B\setminus \{C_4\}$ has a $v_0$-quasi-dominating colouring for every vertex $v_0\in V(G)$, and so, by \Cref{lem:set-to-fraction}, $G$ has an $f_{v_0}$-dominating $2/5$-colouring, where 
        \[f_{v_0}(v) \coloneqq \begin{cases}
            4/5 &\mbox{if $v=v_0$};\\
            1 &\mbox{otherwise}.
        \end{cases}\]
        For $G=C_4$, a similar statement holds by replacing $4/5$ with $3/5$. This is certified by the $(5:2)$-colouring depicted in \Cref{fig:C4-} together with \Cref{lem:set-to-fraction}.
        Assume for the sake of contradiction that $\Gamma$ has a cut-vertex $v_0$; let $(G_0,G_1)$ be a separation of $G$ of order $1$ induced by $v_0$.
        By minimality of $\Gamma$, each $G_i$ that is not in $\B$ has an $f_i$-dominating $2/5$-colouring where $f_i=f_{G_i}$. Moreover, each $G_i$ in $\B \setminus \{C_4\}$ has an $f_i$-dominating $2/5$-colouring where $f_i(v_0)=4/5$ and $f_i(u)=1$ for every $u\neq v_0$. If $G_i=C_4$, the same holds with $f_i(v_0)=3/5$ instead.
        Since $\Gamma \notin \B$, we cannot have $G_0=G_1=C_4$; otherwise $\Gamma$ would be the bad graph depicted in \Cref{fig:badgraphs:vertexgluedC4}. Therefore $\max\{f_0(v_0),f_1(v_0)\}\ge 4/5$, while $\min \{f_0(v_0),f_1(v_0)\}\ge 3/5$. 
        We apply \Cref{lem:gluing} and conclude that $\Gamma$ has an $f_{\Gamma}$-dominating $2/5$-colouring, the desired contradiction.
    \end{proofclaim}

    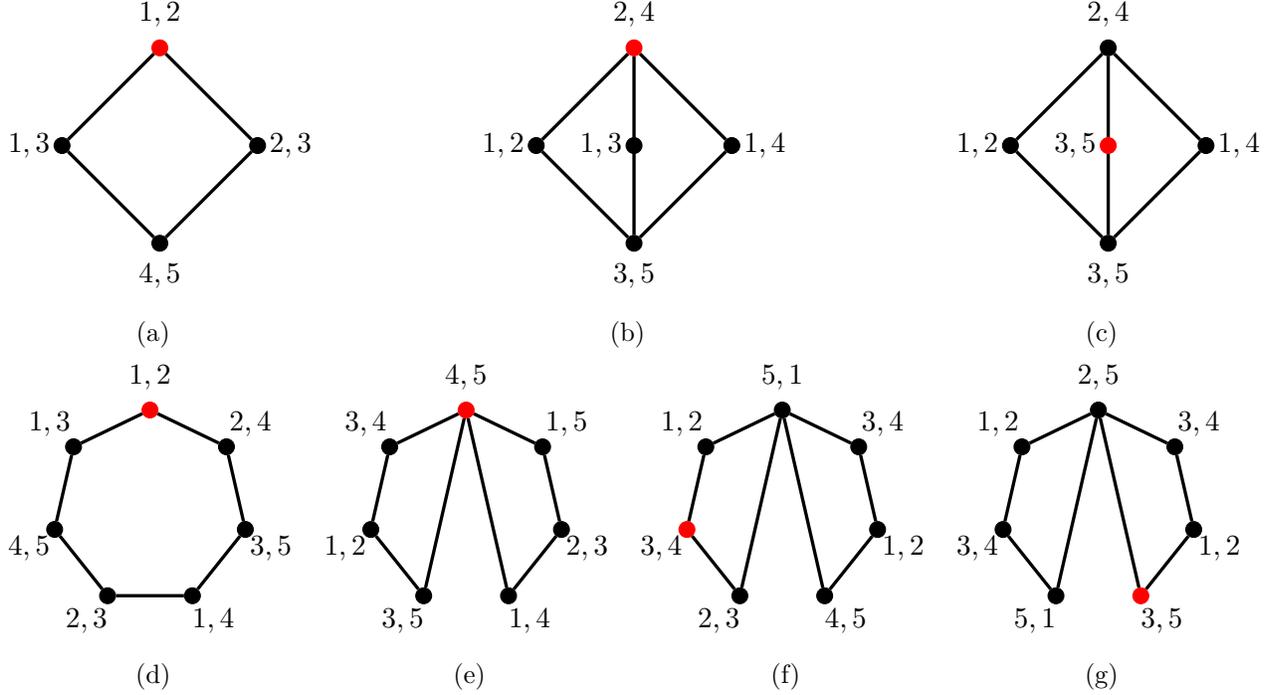
\begin{figure}[!htbp]
        \centering
        \begin{subfigure}[b]{0.24\textwidth}
            \centering
            \begin{tikzpicture}[baseline=0pt, every node/.style={shape=circle,draw=black,fill,inner sep=0pt,minimum size=6pt}, every edge/.style={line width=1.25pt,draw,black}]

    \def\n{4}
    \def\R{1.3} 
    \deflabel{1}{$1,2$}
    \deflabel{2}{$1,3$}
    \deflabel{3}{$4,5$}
    \deflabel{4}{$2,3$}

    \coordinate[color=red, label=\lab{1}, circle] (x1) at (90:\R);
    \foreach \i in {2,...,\n} {
        \pgfmathsetmacro{\angle}{90+360/\n*(\i-1)}
        
        \coordinate[label = \angle : \lab{\i}, circle] (x\i) at ({\angle}:\R);
    }
     \foreach \i in {1,...,\n} {
        \pgfmathtruncatemacro{\next}{mod(\i,\n)+1}
        \draw (x\i) edge (x\next);
    }
                
\end{tikzpicture}
            \caption{}
            \label{fig:C4-}
        \end{subfigure} \hfill
        \begin{subfigure}[b]{0.24\textwidth}
            \centering
                       

\begin{tikzpicture}[baseline=0pt, every node/.style={shape=circle,draw=black,fill,inner sep=0pt,minimum size=6pt}, every edge/.style={line width=1.25pt,draw,black}]
            \def\R{1.3}
                \coordinate[label=right:{$1,4$}, circle] (A) at (0 : \R);
                \coordinate[color=red, label={$2,4$}, circle] (B) at (90 : \R);
                \coordinate[label=left:{$1,2$}, circle] (C) at (180 : \R);
                \coordinate[label=below:{$3,5$}, circle] (D) at (270 : \R);
                \coordinate[label=left:{$1,3$}, circle] (O) at (0, 0);
                \draw (A) edge (B);
                \draw (B) edge (C);
                \draw (C) edge (D);
                \draw (D) edge (A);
                \draw (D) edge (O);
                \draw (O) edge (B);
\end{tikzpicture}
            \caption{}
            \label{fig:K23-a}
        \end{subfigure} \hfill
        \begin{subfigure}[b]{0.24\textwidth}
            \centering
                       

\begin{tikzpicture}[baseline=0pt, every node/.style={shape=circle,draw=black,fill,inner sep=0pt,minimum size=6pt}, every edge/.style={line width=1.25pt,draw,black}]
            \def\R{1.3}
                \coordinate[label=right:{$1,4$}, circle] (A) at (0 : \R);
                \coordinate[label={$2,4$}, circle] (B) at (90 : \R);
                \coordinate[label=left:{$1,2$}, circle] (C) at (180 : \R);
                \coordinate[label=below:{$3,5$}, circle] (D) at (270 : \R);
                \coordinate[color=red, label=left:{$3,5$}, circle] (O) at (0, 0);
                \draw (A) edge (B);
                \draw (B) edge (C);
                \draw (C) edge (D);
                \draw (D) edge (A);
                \draw (D) edge (O);
                \draw (O) edge (B);
\end{tikzpicture}
            \caption{}
            \label{fig:K23-b}
        \end{subfigure}

        \begin{subfigure}[b]{0.24\textwidth}
            \centering
            \begin{tikzpicture}[every node/.style={shape=circle,draw=black,fill,inner sep=0pt,minimum size=6pt}, every edge/.style={line width=1.25pt,draw,black}]

    \def\n{7}
    \def\R{1.3} 
    \deflabel{1}{$1,2$}
    \deflabel{2}{$1,3$}
    \deflabel{3}{$4,5$}
    \deflabel{4}{$2,3$}
    \deflabel{5}{$1,4$}
    \deflabel{6}{$3,5$}
    \deflabel{7}{$2,4$}

    \coordinate[color=red, label = above : \lab{1}, circle] (x1) at (90:\R);
    \foreach \i in {2,...,\n} {
        \pgfmathsetmacro{\angle}{90+360/\n*(\i-1)}
        
        \coordinate[label = \angle : \lab{\i}, circle] (x\i) at ({\angle}:\R);
    }
     \foreach \i in {1,...,\n} {
        \pgfmathtruncatemacro{\next}{mod(\i,\n)+1}
        \draw (x\i) edge (x\next);
    }
    \end{tikzpicture}
            \caption{}
            \label{fig:C7-}
        \end{subfigure} \hfill
        \begin{subfigure}[b]{0.24\textwidth}
            \centering
            \begin{tikzpicture}[every node/.style={shape=circle,draw=black,fill,inner sep=0pt,minimum size=6pt}, every edge/.style={line width=1.25pt,draw,black}]
    \def\n{7}
    \def\R{1.3}
    \deflabel{1}{$1,4$}
    \deflabel{2}{$2,3$}
    \deflabel{3}{$1,5$}
    \deflabel{4}{$4,5$}
    \deflabel{5}{$3,4$}
    \deflabel{6}{$1,2$}
    \deflabel{7}{$3,5$}
    
    \foreach \i in {1,...,\n} {
        \pgfmathsetmacro{\angle}{90+360/\n*(\i-4)}    
        \ifnum\i=4
            \def\couleur{red}
        \else
            \def\couleur{black}
        \fi
            \coordinate[color=\couleur, label = \angle : \lab{\i}, circle] (x\i) at ({\angle}:\R);
    }
    \foreach \i in {1,...,6} {
        \pgfmathtruncatemacro{\next}{mod(\i,\n)+1}
        \draw (x\i) edge (x\next);
    }

    \draw (x1) edge (x4);
    \draw (x4) edge (x7);
                
\end{tikzpicture}
            \caption{}
            \label{fig:2C4-a}
        \end{subfigure} \hfill
        \begin{subfigure}[b]{0.24\textwidth}
            \centering
            \begin{tikzpicture}[every node/.style={shape=circle,draw=black,fill,inner sep=0pt,minimum size=6pt}, every edge/.style={line width=1.25pt,draw,black}]

    \def\n{7}
    \def\R{1.3}
    \deflabel{1}{$1,3$}
    \deflabel{2}{$4,5$}
    \deflabel{3}{$2,3$}
    \deflabel{4}{$1,2$}
    \deflabel{5}{$4,5$}
    \deflabel{6}{$1,3$}
    \deflabel{7}{$4,5$}
    
    \foreach \i in {1,...,\n} {
        \pgfmathsetmacro{\angle}{90+360/\n*(\i-4)}    
        \ifnum\i=6
            \def\couleur{red}
        \else
            \def\couleur{black}
        \fi
            \coordinate[color=\couleur, label = \angle : \lab{\i}, circle] (x\i) at ({\angle}:\R);
    }
    \foreach \i in {1,...,6} {
        \pgfmathtruncatemacro{\next}{mod(\i,\n)+1}
        \draw (x\i) edge (x\next);
    }

    \draw (x1) edge (x4);
    \draw (x4) edge (x7);
                
\end{tikzpicture}
            \caption{}
            \label{fig:2C4-b}
        \end{subfigure} \hfill
        \begin{subfigure}[b]{0.24\textwidth}
            \centering
            \begin{tikzpicture}[every node/.style={shape=circle,draw=black,fill,inner sep=0pt,minimum size=6pt}, every edge/.style={line width=1.25pt,draw,black}]

    \def\n{7}
    \def\R{1.3}
    \deflabel{1}{$3,5$}
    \deflabel{2}{$1,2$}
    \deflabel{3}{$3,4$}
    \deflabel{4}{$2,5$}
    \deflabel{5}{$1,2$}
    \deflabel{6}{$3,4$}
    \deflabel{7}{$5,1$}
    
    \foreach \i in {1,...,\n} {
        \pgfmathsetmacro{\angle}{90+360/\n*(\i-4)}    
        \ifnum\i=1
            \def\couleur{red}
        \else
            \def\couleur{black}
        \fi
            \coordinate[color=\couleur, label = \angle : \lab{\i}, circle] (x\i) at ({\angle}:\R);
    }
    \foreach \i in {1,...,6} {
        \pgfmathtruncatemacro{\next}{mod(\i,\n)+1}
        \draw (x\i) edge (x\next);
    }

    \draw (x1) edge (x4);
    \draw (x4) edge (x7);
                
\end{tikzpicture}
            \caption{}
            \label{fig:2C4-c}
        \end{subfigure}
        
        \caption{For every graph $G \in \B \setminus \{C_4\}$ and every vertex $v_0\in V(G)$ (coloured in red), there is a $v_0$-quasi-dominating $(5:2)$-colouring of $\Gamma$. We ignore the graphs that have a spanning subgraph already in the list.}
        \label{fig:quasi-dominating-colouring}
    \end{figure}

    Observe now that $\Gamma \neq K_2$, since there is a $4/5$-dominating $(5:2)$-colouring of $K_2$ --- colour one vertex with $\{1,2\}$ and the other with $\{3,4\}$. So if there is a degree-$1$ vertex in $\Gamma$, it is adjacent to a vertex of degree at least $2$, which itself is a cut-vertex, a contradiction to \Cref{claim:5/$2$-cutvertex}.
    So $\Gamma$ has minimum degree at least $2$;  in particular, $\fdom(\Gamma)<5/2$.

    \begin{claim}
        \label{claim:5/$2$-adjacent}
        $\Gamma$ does not contain two adjacent $3^+$-vertices.
    \end{claim}

    \begin{proofclaim}
        Suppose for the sake of contradiction that $\Gamma$ has an edge $e = uv$ where $u$ and $v$ are $3^+$ vertices. Let $\Gamma' \coloneqq \Gamma - e$, then $\Gamma'$ has minimum degree at least $2$, and any dominating $2/5$-colouring of $\Gamma'$ is also a dominating $2/5$-colouring of $\Gamma$. Hence we are done if $\Gamma' \notin \B$, by minimality of $\Gamma$.
        
        We now consider the case $\Gamma' \in \B$. We exhaustively list all possible graphs $\Gamma$ that yield that case (where the edge $e$ is highlighted in red) and provide a certificate that they all satisfy $\fdom(\Gamma) \ge 5/2$ in \Cref{fig:badgraph+edge}.
    \end{proofclaim}

    \begin{figure}[!htbp]
        \centering
        \begin{subfigure}[b]{0.24\textwidth}
            \centering
            \begin{tikzpicture}[every node/.style={shape=circle,draw=black,fill,inner sep=0pt,minimum size=6pt}, every edge/.style={line width=1.25pt,draw,black}]

    \def\n{4}
    \def\R{1.1} 
    \deflabel{1}{$1,2$}
    \deflabel{2}{$1,5$}
    \deflabel{3}{$3,4$}
    \deflabel{4}{$1,5$}

    \foreach \i in {1,...,\n} {
        \pgfmathsetmacro{\angle}{90+360/\n*(\i-1)}
        
        \coordinate[label = \angle : \lab{\i}, circle] (x\i) at ({\angle}:\R);
    }
     \foreach \i in {1,...,\n} {
        \pgfmathtruncatemacro{\next}{mod(\i,\n)+1}
        \draw (x\i) edge (x\next);
    }
    \draw (x1) edge[color=red] (x3);
                
\end{tikzpicture}
            \caption{}
            \label{fig:C4+}
        \end{subfigure} \hfill
        \begin{subfigure}[b]{0.24\textwidth}
            \centering
                       

\begin{tikzpicture}[every node/.style={shape=circle,draw=black,fill,inner sep=0pt,minimum size=6pt}, every edge/.style={line width=1.25pt,draw,black}]
            \def\R{1.1}
                \coordinate[label=right:{$1,2$}, circle] (A) at (0 : \R);
                \coordinate[label={$3,4$}, circle] (B) at (90 : \R);
                \coordinate[label=left:{$1,5$}, circle] (C) at (180 : \R);
                \coordinate[label=below:{$2,3$}, circle] (D) at (270 : \R);
                \coordinate[label=left:{$4,5$}, circle] (O) at (0, 0);
                \draw (A) edge (B);
                \draw (B) edge (C);
                \draw (C) edge (D);
                \draw (D) edge (A);
                \draw (D) edge (O);
                \draw (O) edge[color=red] (A);
                \draw (O) edge (B);
\end{tikzpicture}
            \caption{}
            \label{fig:K23+a}
        \end{subfigure} \hfill
        \begin{subfigure}[b]{0.24\textwidth}
            \centering
                       

\begin{tikzpicture}[every node/.style={shape=circle,draw=black,fill,inner sep=0pt,minimum size=6pt}, every edge/.style={line width=1pt,draw,black}]
            \def\R{1.1}
                \coordinate[label=right:{$1,5$}, circle] (A) at (0 : \R);
                \coordinate[label={$1,2$}, circle] (B) at (90 : \R);
                \coordinate[label=left:{$1,5$}, circle] (C) at (180 : \R);
                \coordinate[label=below:{$3,4$}, circle] (D) at (270 : \R);
                \coordinate[label=left:{$1,5$}, circle] (O) at (0, 0);
                \draw (A) edge (B);
                \draw (B) edge (C);
                \draw (C) edge (D);
                \draw (D) edge (A);
                \draw (D) edge (O);
                \draw (O) edge (B);
                \draw (B) edge[bend left, color=red] (D);
\end{tikzpicture}
            \caption{}
            \label{fig:K23+b}
        \end{subfigure} \hfill

        \begin{subfigure}[b]{0.24\textwidth}
            \centering
\begin{tikzpicture}[every node/.style={shape=circle,draw=black,fill,inner sep=0pt,minimum size=6pt}, every edge/.style={line width=1.25pt,draw,black}]

    \def\n{7}
    \def\R{1.1} 
    \deflabel{1}{$3,4$}
    \deflabel{2}{$1,5$}
    \deflabel{3}{$3,4$}
    \deflabel{4}{$2,5$}
    \deflabel{5}{$1,5$}
    \deflabel{6}{$3,4$}
    \deflabel{7}{$2,5$}

    \foreach \i in {1,...,\n} {
        \pgfmathsetmacro{\angle}{90+360/\n*(\i-1)}
        
        \coordinate[label = \angle : \lab{\i}, circle] (x\i) at ({\angle}:\R);
    }
     \foreach \i in {1,...,\n} {
        \pgfmathtruncatemacro{\next}{mod(\i,\n)+1}
        \draw (x\i) edge (x\next);
    }
    \draw (x\n) edge[color=red] (x2);
    \end{tikzpicture}
            \caption{}
            \label{fig:C7+}
        \end{subfigure} \hfill
        \begin{subfigure}[b]{0.24\textwidth}
            \centering
            \begin{tikzpicture}[every node/.style={shape=circle,draw=black,fill,inner sep=0pt,minimum size=6pt}, every edge/.style={line width=1.25pt,draw,black}]
    \def\n{7}
    \def\R{1.1} 
    \deflabel{1}{$2,3$}
    \deflabel{2}{$4,5$}
    \deflabel{3}{$1,2$}
    \deflabel{4}{$4,5$}
    \deflabel{5}{$1,2$}
    \deflabel{6}{$3,4$}
    \deflabel{7}{$5,1$}

    \foreach \i in {1,...,\n} {
        \pgfmathsetmacro{\angle}{90+360/\n*(\i-1)}
        
        \coordinate[label = \angle : \lab{\i}, circle] (x\i) at ({\angle}:\R);
    }
     \foreach \i in {1,...,\n} {
        \pgfmathtruncatemacro{\next}{mod(\i,\n)+1}
        \draw (x\i) edge (x\next);
    }

    \draw (x3) edge (x6);
    \draw (x4) edge[color=red] (x1);

\end{tikzpicture}
            \caption{}
            \label{fig:C7++}
        \end{subfigure} \hfill
        \begin{subfigure}[b]{0.24\textwidth}
            \centering
            \begin{tikzpicture}[every node/.style={shape=circle,draw=black,fill,inner sep=0pt,minimum size=6pt}, every edge/.style={line width=1.25pt,draw,black}]

    \def\n{7}
    \def\R{1.1} 
    \deflabel{1}{$3,4$}
    \deflabel{2}{$1,2$}
    \deflabel{3}{$3,4$}
    \deflabel{4}{$5,1$}
    \deflabel{5}{$2,3$}
    \deflabel{6}{$4,5$}
    \deflabel{7}{$2,3$}

    \foreach \i in {1,...,\n} {
        \pgfmathsetmacro{\angle}{90+360/\n*(\i-4)}
        
        \coordinate[label = \angle : \lab{\i}, circle] (x\i) at ({\angle}:\R);
    }
     \foreach \i in {1,...,6} {
        \pgfmathtruncatemacro{\next}{mod(\i,\n)+1}
        \draw (x\i) edge (x\next);
    }

    \draw (x1) edge (x4);
    \draw (x4) edge (x\n);
    \draw (x6) edge[color=red] (x2);
                
\end{tikzpicture}
            \caption{}
            \label{fig:2C4+a}
        \end{subfigure} \hfill
        \begin{subfigure}[b]{0.24\textwidth}
            \centering
            \begin{tikzpicture}[every node/.style={shape=circle,draw=black,fill,inner sep=0pt,minimum size=6pt}, every edge/.style={line width=1.25pt,draw,black}]

    \def\n{7}
    \def\R{1.1} 
    \deflabel{1}{$3,4$}
    \deflabel{2}{$1,2$}
    \deflabel{3}{$5,1$}
    \deflabel{4}{$5,1$}
    \deflabel{5}{$1,2$}
    \deflabel{6}{$3,4$}
    \deflabel{7}{$1,2$}

    \foreach \i in {1,...,\n} {
        \pgfmathsetmacro{\angle}{90+360/\n*(\i-4)}
        
        \coordinate[label = \angle : \lab{\i}, circle] (x\i) at ({\angle}:\R);
    }
    \foreach \i in {1,...,6} {
        \pgfmathtruncatemacro{\next}{mod(\i,\n)+1}
        \draw (x\i) edge (x\next);
    }

    \draw (x1) edge (x4);
    \draw (x4) edge (x\n);
    \draw (x6) edge[color=red] (x3);
                
\end{tikzpicture}
            \caption{}
            \label{fig:2C4+b}
        \end{subfigure} \hfill

        \caption{All possible graphs $\Gamma\notin \B$ obtained from a graph $\Gamma'\in \B$ with an extra edge (coloured in red) have a dominating $(5:2)$-colouring. We ignore the graphs that have a cut-vertex and those that have a spanning subgraph already in the list.}
        \label{fig:badgraph+edge}
    \end{figure}
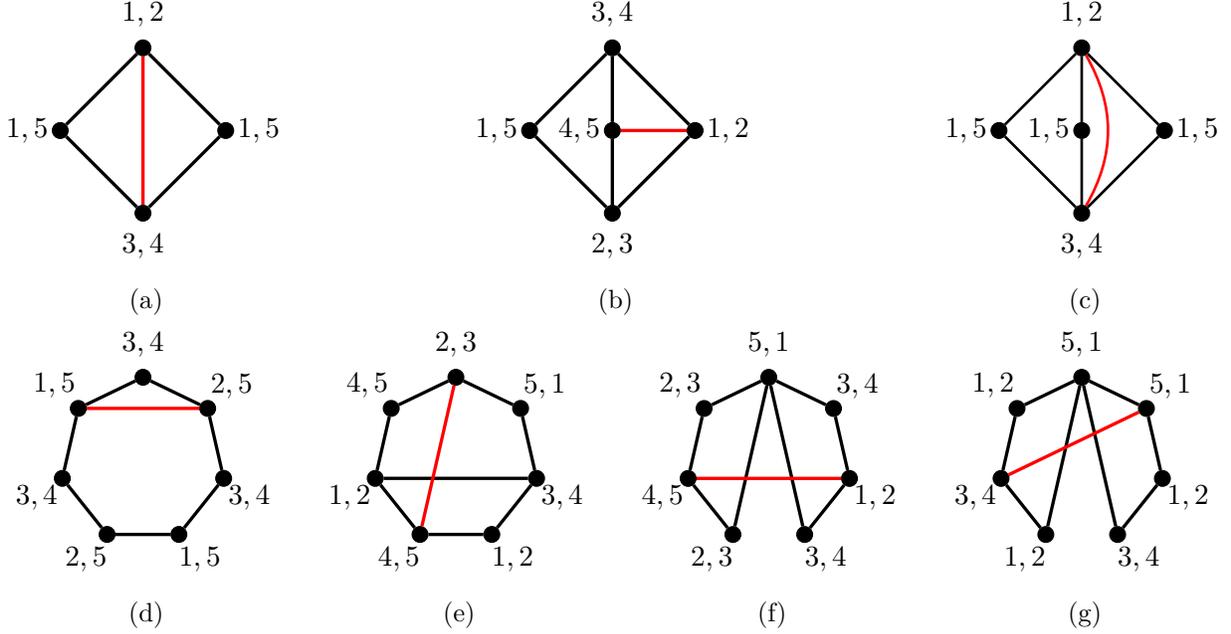

    \begin{claim}
        \label{claim:5/$2$-C4}
        $\Gamma$ has no $C_4$ as a subgraph.
    \end{claim}

    \begin{proofclaim}
        Suppose for the sake of contradiction that $\Gamma$ has a $C_4$ as a subgraph. By \Cref{claim:5/$2$-cutvertex} and \Cref{claim:5/$2$-adjacent}, it must consist of two twin suspended $2$-paths.
        Let $v,v'$ be the middle vertices of these two paths, and let $\Gamma' \coloneqq \Gamma \setminus v$. 
        If $\Gamma' \in \B$, then $\Gamma$ is either isomorphic to $K_{2,4}$, or has a copy of the graph $\Theta_{2,2,5}$ depicted in \Cref{fig:C7+C4} as a spanning subgraph. Both graphs have $\fdom\ge 5/2$ as illustrated in \Cref{fig:K24,fig:C7+C4}, a contradiction.
        We conclude that $\Gamma' \notin \B$, and so $\fdom(\Gamma')\ge 5/2$ (since $\Gamma$ and therefore also $\Gamma'$ have minimum degree at least $2$). Let $\bD$ be a dominating $2/5$-colouring of $\Gamma'$. We add $v$ to $\bD$ whenever $v'\in \bD$, so that we have $\pr{v\in \bD}=\pr{v' \in \bD}=2/5$ and $\pr{v \mbox{ is dominated by } \bD} = \pr{v' \mbox{ is dominated by } \bD} = 1$. So $\bD$ is a dominating $2/5$-colouring of $\Gamma$, a contradiction.
    \end{proofclaim}

    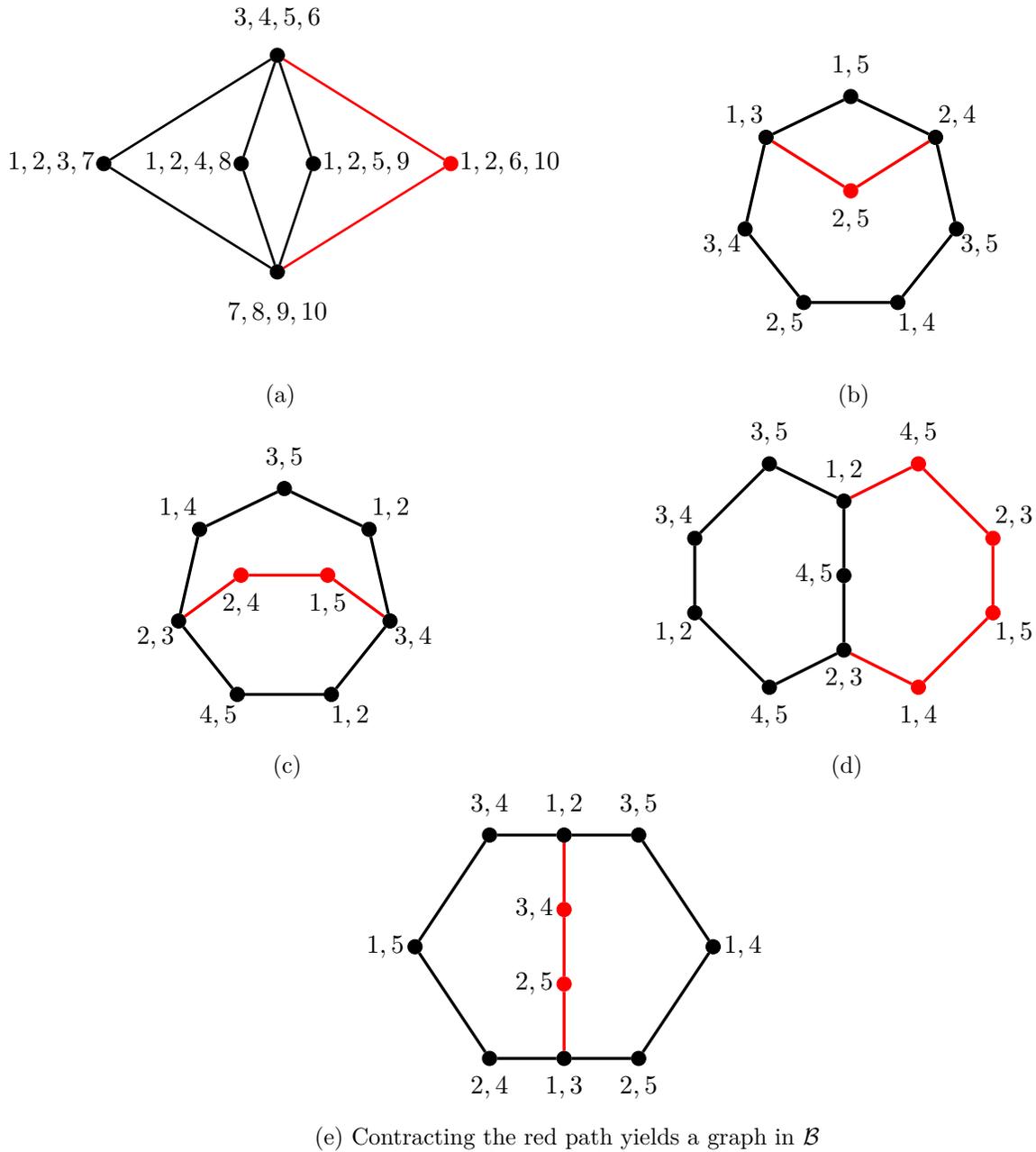
\begin{figure}[!htbp]
        \centering
        \begin{subfigure}[b]{0.49\textwidth}
            \centering
            \begin{tikzpicture}[baseline=0pt, every node/.style={shape=circle,draw=black,fill,inner sep=0pt,minimum size=6pt}, every edge/.style={line width=1pt,draw,black}]
            \def\R{1.6}
                \coordinate[label={[label distance=-0.25cm]above:$3,4,5,6$}, circle] (x) at (0,\R);
                \coordinate[label=left:{$1,2,3,7$}, circle] (A) at (-1.6*\R, 0);
                \coordinate[label=left:{$1,2,4,8$}, circle] (B) at (-\R/3, 0);
                \coordinate[label=right:{$1,2,5,9$}, circle] (C) at (\R/3, 0);\coordinate[color=red, label=right:{$1,2,6,10$}, circle] (D) at (1.6*\R,0);
                \coordinate[label={[label distance=-0.25cm]below:$7,8,9,10$}, circle] (y) at (0,-\R);
                \draw (x) edge (A) edge (B) edge (C) edge[color=red] (D);
                \draw (y) edge (A) edge (B) edge (C) edge[color=red] (D);

                \clip (-1.5*\R, -1.5*\R) rectangle (1.5*\R, 1.5*\R); 
\end{tikzpicture}
            \caption{}
            \label{fig:K24}
        \end{subfigure} \hfill
        \begin{subfigure}[b]{0.49\textwidth}
            \centering
            
            \begin{tikzpicture}[baseline=0pt, every node/.style={shape=circle,draw=black,fill,inner sep=0pt,minimum size=6pt}, every edge/.style={line width=1.25pt,draw,black}]

    \def\n{7}
    \def\R{1.6} 
    \deflabel{0}{$2,5$}
    \deflabel{1}{$1,5$}
    \deflabel{2}{$1,3$}
    \deflabel{3}{$3,4$}
    \deflabel{4}{$2,5$}
    \deflabel{5}{$1,4$}
    \deflabel{6}{$3,5$}
    \deflabel{7}{$2,4$}

    \foreach \i in {1,...,\n} {
        \pgfmathsetmacro{\angle}{90+360/\n*(\i-1)}
        
        \coordinate[label = \angle : \lab{\i}, circle] (x\i) at ({\angle}:\R);
    }
     \foreach \i in {1,...,\n} {
        \pgfmathtruncatemacro{\next}{mod(\i,\n)+1}
        \draw (x\i) edge (x\next);
    }
    \coordinate[color=red, label=below:\lab{0}, circle] (x0) at (0,0.13*\R);
    \draw (x0) edge[color=red] (x2) edge[color=red] (x7);
    \clip (-1.5*\R, -1.5*\R) rectangle (1.5*\R, 1.5*\R); 
    \end{tikzpicture}
            \caption{}
            \label{fig:C7+C4}
        \end{subfigure}
        \begin{subfigure}[b]{0.49\textwidth}
            \centering
        \begin{tikzpicture}[baseline=0pt, every node/.style={shape=circle,draw=black,fill,inner sep=0pt,minimum size=6pt}, every edge/.style={line width=1.25pt,draw,black}]

    \def\n{7}
    \def\R{1.6} 
    \deflabel{1}{$3,5$}
    \deflabel{2}{$1,4$}
    \deflabel{3}{$2,3$}
    \deflabel{4}{$4,5$}
    \deflabel{5}{$1,2$}
    \deflabel{6}{$3,4$}
    \deflabel{7}{$1,2$}
    \deflabel{8}{$2,4$}
    \deflabel{9}{$1,5$}

    \foreach \i in {1,...,\n} {
        \pgfmathsetmacro{\angle}{90+360/\n*(\i-1)}
        
        \coordinate[label = \angle : \lab{\i}, circle] (x\i) at ({\angle}:\R);
    }
     \foreach \i in {1,...,\n} {
        \pgfmathtruncatemacro{\next}{mod(\i,\n)+1}
        \draw (x\i) edge (x\next);
    }
    \coordinate[color=red, label=below:\lab{8}, circle] (y0) at (-0.4*\R,0.2*\R);
    \coordinate[color=red, label=below:\lab{9}, circle] (y1) at (0.4*\R,0.2*\R);
    \draw (y0) edge[color=red] (x3) edge[color=red] (y1);
    \draw (y1) edge[color=red] (x6);
    \end{tikzpicture}
        \caption{}
        \label{fig:C7+C4bis}
        \end{subfigure}
        \begin{subfigure}[b]{0.49\textwidth}
            \centering
            \begin{tikzpicture}[every node/.style={shape=circle,draw=black,fill,inner sep=0pt,minimum size=6pt}, every edge/.style={line width=1.25pt,draw,black}]

    \def\n{7}
    \def\R{1.1} 
    \deflabel{1}{$1,2$}
    \deflabel{2}{$3,5$}
    \deflabel{3}{$3,4$}
    \deflabel{4}{$1,2$}
    \deflabel{5}{$4,5$}
    \deflabel{6}{$2,3$}
    \deflabel{7}{$1,4$}
    \deflabel{8}{$1,5$}
    \deflabel{9}{$2,3$}
    \deflabel{10}{$4,5$}
    \deflabel{11}{$4,5$}

    \coordinate[label = \lab{1}, circle] (a) at
        (0,\R);
    \coordinate[label = \lab{2}, circle] (b) at
        (-\R, 1.5*\R);
    \coordinate[label = above left:\lab{3}, circle] (c) at
        (-2*\R, 0.5*\R);
    \coordinate[label = below left:\lab{4}, circle] (d) at
        (-2*\R, -0.5*\R);
    \coordinate[label = below:\lab{5}, circle] (e) at
        (-\R, -1.5*\R);
    \coordinate[label = below:\lab{6}, circle] (f) at
        (0, -\R);
    \coordinate[color=red, label = below:\lab{7}, circle] (g) at
        (\R, -1.5*\R);
    \coordinate[color=red, label = below right:\lab{8}, circle] (h) at
        (2*\R, -0.5*\R);
    \coordinate[color=red, label = above right:\lab{9}, circle] (i) at
        (2*\R, 0.5*\R);
    \coordinate[color=red, label = \lab{10}, circle] (j) at
        (\R, 1.5*\R);
    \coordinate[label = left: \lab{11}, circle] (x) at (0, 0);
    
    \draw (a) edge (b) edge (x);
    \draw (b) edge (c);
    \draw (c) edge (d);
    \draw (d) edge (e);
    \draw (e) edge (f);
    \draw (f) edge[color=red] (g) edge (x);
    \draw (g) edge[color=red] (h);
    \draw (h) edge[color=red] (i);
    \draw (i) edge[color=red] (j);
    \draw (j) edge[color=red] (a);
    \end{tikzpicture}
            \caption{}
            \label{fig:C7+C7(2)}
        \end{subfigure} \hfill
        \begin{subfigure}[b]{0.49\textwidth}
            \centering
            \begin{tikzpicture}[every node/.style={shape=circle,draw=black,fill,inner sep=0pt,minimum size=6pt}, every edge/.style={line width=1.25pt,draw,black}]

    \def\n{7}
    \def\R{1.1} 
    \deflabel{1}{$1,2$}
    \deflabel{2}{$3,4$}
    \deflabel{3}{$1,5$}
    \deflabel{4}{$2,4$}
    \deflabel{5}{$1,3$}
    \deflabel{6}{$2,5$}
    \deflabel{7}{$1,4$}
    \deflabel{8}{$3,5$}
    \deflabel{9}{$3,4$}
    \deflabel{10}{$2,5$}

    \coordinate[label = \lab{1}, circle] (a) at
        (0,1.5*\R);
    \coordinate[label = \lab{2}, circle] (b) at
        (-\R, 1.5*\R);
    \coordinate[label = left:\lab{3}, circle] (c) at
        (-2*\R, 0);
    \coordinate[label = below:\lab{4}, circle] (d) at
        (-\R, -1.5*\R);
    \coordinate[label = below:\lab{5}, circle] (e) at
        (0, -1.5*\R);
    \coordinate[label = below:\lab{6}, circle] (f) at
        (\R, -1.5*\R);
    \coordinate[label = right:\lab{7}, circle] (g) at
        (2*\R, 0);
    \coordinate[label = \lab{8}, circle] (h) at
        (\R, 1.5*\R);
    \coordinate[color=red, label = left: \lab{9}, circle] (x) at (0, \R/2);
    \coordinate[color=red, label = left: \lab{10}, circle] (y) at (0, -\R/2);
    
    \draw (a) edge (b) edge[color=red] (x);
    \draw (b) edge (c);
    \draw (c) edge (d);
    \draw (d) edge (e);
    \draw (e) edge (f) edge[color=red] (y);
    \draw (f) edge (g);
    \draw (g) edge (h);
    \draw (h) edge (a);
    \draw (x) edge[color=red] (y);
    \end{tikzpicture}
            \caption{Contracting the red path yields a graph in $\B$}
            \label{fig:C7+C7(3)}
        \end{subfigure}
        \caption{In \ref{fig:K24} -- \ref{fig:C7+C7(2)}, removing the red suspended path yields a graph in $\B$. In \ref{fig:C7+C7(3)}, it must be contracted. In every case, there is a dominating $2/5$-colouring.}
    \end{figure}

    Every graph in $\mathcal{B}$ except $C_7$ contains a $C_4$ as a subgraph so from now on, when we consider a proper subgraph $\Gamma'$ of $\Gamma$, we only need to check that $\Gamma'\neq C_7$ for the statement of \Cref{thm:5/$2$-technical} to hold.

    \begin{claim}
        \label{claim:5/$2$-twin-P4}
        $\Gamma$ does not have two twin suspended $3$-paths.
    \end{claim}

    \begin{proofclaim}
        Suppose for the sake of contradiction that $\Gamma$ has two twin suspended $3$-paths $P =  x\link u\link v\link y$ and $P'=x\link u'\link v'\link y$. Let $\Gamma' \coloneqq \Gamma \setminus P$. 
        If $\Gamma' \in \B$ then $\Gamma'$ is isomorphic to $C_7$ since $\Gamma$ is $C_4$-free by \Cref{claim:5/$2$-C4}, so $\Gamma$ is isomorphic to the graph $\Theta_{3,3,4}$ depicted in \Cref{fig:C7+C4bis}, which is given with a dominating $(5:2)$-colouring, a contradiction.
        We conclude that $\Gamma' \notin \B$, and so $\fdom(\Gamma')\ge 5/2$ (since $\Gamma$ and therefore also $\Gamma'$ have minimum degree at least $2$). Let $\bD$ be a dominating $2/5$-colouring of $\Gamma'$.
        We add $u$ (resp. $v$) to $\bD$ whenever $u'\in \bD$ (resp. $v'\in \bD$), so that we have $\pr{u\in \bD}=\pr{u' \in \bD}=2/5$ and $\pr{u \mbox{ is dominated by } \bD} = \pr{u' \mbox{ is dominated by } \bD} = 1$. The same holds by symmetry for $v$. 
        So $\bD$ is a dominating $2/5$-colouring of $\Gamma$, a contradiction.
    \end{proofclaim}

    \begin{claim}
        \label{claim:5/$2$-geodesic-P_4}
        For every suspended $3$-path in $\Gamma$, there exists a suspended $2$-path that shares the same endpoints. In other words, every suspended $3$-path of $\Gamma$ is part of a hammock.
    \end{claim}

    \begin{proofclaim}
        Suppose for the sake of contradiction that $\Gamma$ has a suspended $3$-path $P=u \link x\link y\link v$ such that $u$ and $v$ have no common neighbour.
        We construct the graph $\Gamma'$ from $\Gamma$ by contracting $P$ into a single vertex $w$, which must be of degree at least $4$.
        If we assume that $\Gamma' \in \B$, then $\Gamma'$ is therefore one of the three bad graphs containing a vertex of degree $4$, all of which contain \Cref{fig:badgraphs:vertexgluedC4} as a spanning subgraph. We infer that $\Gamma$ contains the graph depicted in \Cref{fig:C7+C7(3)} as a spanning subgraph, which is given with a dominating $(5:2)$-colouring, a contradiction. 

        We may now assume that $\Gamma' \notin \B$, and so there is a dominating $2/5$-colouring $\bD'$ of $\Gamma'$ since $\Gamma'$ has minimum degree at least 2. 
        We may construct a dominating $2/5$-colouring $\bD$ of $\Gamma$ as follows. Let 
        \[\bD_0 \coloneqq \begin{cases}
        \bD'& \mbox{if $w\notin \bD'$}\\
        \bD' \cup \{u,v\} \setminus \{w\} & \mbox{otherwise}.
        \end{cases}\]
        Since $w$ is always dominated by $\bD'$, we always have that at least one of $u$ and $v$ is dominated by $\bD_0$. 
        Observe that $u$ has a neighbour $u' \neq x$ of degree $2$, and 
        \[1 = \pr{N_\Gamma[u']\cap\bD_0\neq \emptyset} \le \frac{4}{5} + \pr{\{u',u\}\cap \bD_0 = \{u\}}.\]
        We infer that $\pr{\{u',u\}\subseteq \bD_0} \le \frac{1}{5} $ and thus $\pr{u \in N(\bD_0) \setminus \bD_0} \ge \frac{1}{5}$.
        Moreover, 
        \[\pr{u \notin N[\bD_0]} + \pr{u \in N(\bD_0) \setminus \bD_0}=\pr{u\notin \bD_0} \le \frac{3}{5},\]
        and so $\pr{u\notin N[\bD_0]}\le \frac{2}{5}.$
        The same holds for $v$, by symmetry.

        Let us write $p\coloneqq \pr{\{u,v\} \subseteq N[\bD_0] \setminus \bD_0}$. We have
        \[ 1 = \pr{w\in \bD'}+ \pr{u\notin N[\bD_0]} + \pr{v\notin N[\bD_0]} + p,\]
        and so  $\pr{u\notin N[\bD_0]} + \pr{v\notin N[\bD_0]} = \frac{3}{5} - p$.
        
        We initially set $\bD \coloneqq \bD_0$.
        If $u$ (resp. $v$) is not dominated by $\bD_0$, we add $x$ (resp. $y$) to $\bD$. 
        If $\{u,v\} \subseteq N[\bD_0]\setminus \bD_0$, we set
        \[ \bD \gets \begin{cases}
            \bD \cup \{x\} &\mbox{if $\pr{v\notin N[\bD_0]}\ge \frac{1}{5}$;}\\
            \bD \cup \{y\} &\mbox{otherwise if $\pr{u\notin N[\bD_0]}\ge \frac{1}{5}$;}\\
            \bD \cup \{\mathbf{z}\} & \mbox{otherwise, where $\mathbf{z}\sim \cU(\{x,y\})$.}
        \end{cases} \]

        We claim that $\bD$ is a dominating $2/5$-colouring of $\Gamma$. It is clear by construction that $\bD$ is a random dominating set of $G$, and that $\pr{z\in \bD} \le 2/5$ for all $z\notin \{x,y\}$. 
        First, assume that $\pr{v\notin N[\bD_0]}\ge \frac{1}{5}$ and so $\pr{u\notin N[\bD_0]}\le \frac{2}{5} - p$. In that case, we have
        \begin{align*}
            \pr{x \in \bD} &= \pr{u \notin N[\bD_0]} + \pr{\{u,v\} \subseteq N[\bD_0] \setminus \bD_0} \\
            &\le \pth{\frac{2}{5}-p} + p = \frac{2}{5}.
        \end{align*}
        Otherwise, assume that $\pr{u\notin N[\bD_0]}\ge \frac{1}{5}$. In that case, we have
        \begin{align*}
            \pr{x \in \bD} &= \pr{u \notin N[\bD_0]} \le \frac{2}{5}.
        \end{align*}
        Finally, assume that $\pr{u\notin N[\bD_0]}\le \frac{1}{5}$. In that case, we have
        \begin{align*}
            \pr{x \in \bD} &= \pr{u \notin N[\bD_0]} + \frac{1}{2}\cdot \pr{\{u,v\} \subseteq N[\bD_0] \setminus \bD_0} \\
            &\le \max_{x\in [0,1/5]} \pth{x + \frac{1}{2}\pth{\frac{3}{5}-x}} = \frac{2}{5}.
        \end{align*}
        We conclude that, in every case, $\pr{x\in \bD}\le \frac{2}{5}$.
        A similar computation can be made to show that the same holds for $\pr{y\in \bD}$.
    \end{proofclaim}

    \begin{claim}
        For every $k\ge 4$, $\Gamma$ has no suspended $k$-path.
    \end{claim}

    \begin{proofclaim}
        Assume for the sake of contradiction that $\Gamma$ has a suspended path $P$ of length at least $4$. 
        
        First assume that $P$ can be chosen in such a way that the graph $\Gamma' \coloneqq \Gamma \setminus P$ is not isomorphic to $C_7$.
        Since $\Gamma$ is $C_4$-free, it follows that $\Gamma' \notin \B$ in that case,  and so $\fdom(\Gamma')\ge 5/2$. Then, by \Cref{lem:strong-pathlemma}, $\fdom(\Gamma) \ge 5/2$, a contradiction. 

        We conclude that regardless of the choice of $P$, $\Gamma'$ is isomorphic to $C_7$. 
        The extremities of $P$ cannot be adjacent by \Cref{claim:5/$2$-adjacent} nor can then be at distance $3$ in $\Gamma'$ by \Cref{claim:5/$2$-geodesic-P_4}. Therefore they must be at distance $2$ in $\Gamma'$. The other part of the $7$-cycle forms a suspended $5$-path in $\Gamma$, and removing its internal vertices still yields a graph isomorphic to $C_7$. We conclude that $\Gamma$ is the graph depicted in \Cref{fig:C7+C7(2)}, which is given with a dominating $(5:2)$-colouring, a contradiction.
    \end{proofclaim}

    	Observe now that $\Gamma$ cannot be a cycle, since $\{C_4,C_7\}\subseteq \B$, and for every integer $\ell \ge 3$, $\fdom(C_\ell)=\ell/\ceil{\ell/3} \ge 5/2$ unless $\ell \in \{4,7\}$ --- by \Cref{prop:n/gamma} and the fact that $C_\ell$ is vertex-transitive and has $\gamma(C_\ell)=\ceil{\ell/3}$.
	Together with the above claims, this implies that $\Gamma$ satisfies the hypotheses of \Cref{lem:5/2}, and so $\fdom(\Gamma) \ge 5/2$, a contradiction.
\end{proof}

\section{Graphs of bounded expansion}
\label{sec:planar}

Successive applications of \Cref{lem:strong-pathlemma} imply that if one can repetitively extract long suspended paths from a graph $G$ of minimum degree $2$, until reaching a long cycle, then $\fdom(G)$ is close to $3$. This strategy can be applied in particular when $G$ is a planar graph of large girth, and more generally when it has sub-exponential expansion, as we detail in this section.

\begin{defi}
    Let $k\ge 1$ be an integer. A $2$-connected graph $G$ is \emph{$k$-path degenerate} if $G=K_2$ or $G=C_n$ for some integer $n\ge k$, or if there is a suspended $k$-path $P$ in $G$ such that $G\setminus P$ is $k$-path degenerate. A non-$2$-connected graph is $k$-path degenerate if all its blocks are.  
\end{defi}

\begin{thm}
\label{thm:path-degeneracy-fdom}
    Let $G$ be a $(3k-2)$-path degenerate graph of minimum degree at least $2$ and girth at least $6k-4$. Then
    \[ \fdom(G) \ge 3 - \frac{1}{k}.\]
\end{thm}

\begin{proof}
    Up to colouring the blocks of $G$ independently, we may assume that $G$ is $2$-connected.
    We prove the result by induction on $G$, the base case being when $G=C_n$ for some integer $n\ge 6k-4$. First assume that $n\ge 6k-2$; it is well known (see \cite[Proposition 2]{GoHe20}) that 
    \[\fdom(C_n) = \frac{n}{\lceil n/3 \rceil} \ge \frac{3n}{n+2} \ge \frac{18k-6}{6k} = 3 - 1/k.\]
    For the two remaining values of $n$, we have
    \[ \fdom(C_{6k-3}) \ge \fdom(C_{6k-4}) = \frac{6k-4}{2k-1} = 3 - \frac{2}{2k-1} > 3 - \frac{1}{k}.\]
    We conclude that the base case of the induction is verified.
    If $G$ is not a cycle, then it contains a suspended path $P$ of length at least $3k-2$, by the fact that $G$ is $(3k-2)$-path degenerate and $2$-connected. Let $G_0\coloneqq G\setminus P$; by the induction hypothesis, $\fdom(G_0) \ge 3-1/k$, and by \Cref{lem:strong-pathlemma}, $\fdom(G) \ge 3-1/k$, as desired.
\end{proof}

Let $G$ be a graph. An $r$-shallow minor of $G$ is a graph that can be obtained by contracting pairwise vertex-disjoint connected subgraphs of $G$ of (inner) radius at most $r$, and possibly removing some edges and vertices. 
Given a graph $G$ and an integer $r$, we define 
\[ \nabla_r(G) \coloneqq \max\sst{\frac{|E(H)|}{|V(H)|}}{H \mbox{ is an $r$-shallow minor of }G}. \]
Given a function $f\colon \mathbb{N} \to \mathbb{R}$, we say that $G$ has an \emph{expansion bounded by $f$} if $\nabla_r(G) \le f(r)$ for every integer $r\ge 0$.
We say that a class of graphs $\sG$ has \emph{bounded expansion} if there is a function $f$ that bounds the expansion of every graph $G\in \sG$. If $f(r)=2^{o(r)}$, we say that $\sG$ has \emph{sub-exponential expansion}; if $f(r) = O(r^b)$, we say that $\sG$ has \emph{polynomial expansion of degree $b$}. 

In \cite{LiOs25+}, the authors derive lower bounds on the girth of a graph drawn from a class with sub-exponential expansion that ensure that it is $k$-path degenerate. We may sum up a selection of these results as follows; c.f. \cite[Theorem 1.1 (looking at the explicit girth requirement in its proof), Theorem 1.2, Theorem 1.6, Theorem 3.10]{LiOs25+}

\begin{thm}
\label{thm:bounded-expansion-path-degeneracy}
    Let $\eps > 0$, let $k\ge 1$ be an integer, let $\sG$ be a class of graphs, and let $G\in \sG$ be a graph of girth $g$. If
    \begin{enumerate}[label={\rm (\roman*)}]
        \item  $\sG$ has an expansion bounded by $r\mapsto 2^{r/w(r)}$ for some increasing function $w \to \infty$, and $g \ge \pth{2w^{-1}(3k)+3}(k-1)$; or
        \item $\sG$ has polynomial expansion of degree $b$, $g \ge (4+\eps)bk\log_2 k$, and $k$ is large enough in terms of $\eps$; or
        \item $\sG$ is proper minor-closed, $\mad(G')\le d$ for all $G'\in \sG$, and $g\ge (6 \log_2 d + 3)(k-1)$; or
        \item $\sG$ is the class of planar graphs and $g \ge 5k-4$,
    \end{enumerate}
    then $G$ is $k$-path degenerate.
\end{thm}

Together with \Cref{thm:path-degeneracy-fdom}, \Cref{thm:bounded-expansion-path-degeneracy} implies the following as a direct corollary when taking $\sG = \{G\}$ for (i) and (ii) and letting $\sG$ consist of all minors of $G$ for (iii). 

\begin{cor}
    \label{cor:bounded-expansion}
    Let $\eps > 0$, let $k\ge 1$ be an integer, and let $\sG$ be a class of graphs with sub-exponential expansion. 
    Let $G\in \sG$ be a graph of minimum degree at least $2$ and of girth $g$. If
    \begin{enumerate}[label={\rm (\roman*)}]
        \item $G$ has an expansion bounded by $r\mapsto 2^{r/w(r)}$ for some increasing function $w \to \infty$, and $g \ge \pth{2w^{-1}(3k)+3}(k-1)$; or
        \item $G$ has a polynomial expansion of degree $b$, $g \ge (12+\eps)bk\log_2 k$, and $k$ is large enough in terms of $\eps$; or
        \item every minor of $G$ has average degree at most $d$ and $g\ge (6 \log_2 d + 3)(3k+1)$; or
        \item \label{it:planar} $G$ is planar and $g \ge 15k-14$, then
    \end{enumerate}
    \[\fdom(G) \ge 3-\frac{1}{k}.\]
\end{cor}

If we focus on planar graphs of large girth, we note that the girth requirement in \Cref{cor:bounded-expansion}\ref{it:planar} is certainly not optimal: when $k=2$, it is required that the girth of $G$ is at least $16$ to ensure that $\fdom(G) \ge 5/2$, while we can infer from \Cref{thm:5/2} that having girth at least $8$ suffices to ensure that conclusion, even without the planar hypothesis.

\section{Discussion}
\label{sec:conclusion}


While the proof of \Cref{thm:5/2} is formulated using the probabilistic language, it actually describes a deterministic process in order to construct a probability distribution over the family of dominating sets of a given graph $G$. It would be possible to return a dominating $(5q:2q)$-colouring of $G$ instead, for some value of $q$ with an exponential dependency on $|V(G)|$. To illustrate this, we describe how to obtain a precolouring that mimics the set-up of the proof of \Cref{lem:5/2}.

\begin{lemma}
\label{lem:intersecting-family}
    For every disjoint sets of elements $A,B$, there is a mapping $\phi\colon A\cup B \to 2^{\mathbb{Z}}$ such that, writing $t\coloneqq \left |\bigcup_{x\in A\cup B} \phi(x) \right|$, one has
    \begin{enumerate}[label=(\roman*)]
    \item $|\phi(x)|=\frac{2t}{5}$ for all $x\in A\cup B$; \item $|\phi(a) \cap \phi(a')|=\frac{t}{5}$ for every pair $\{a,a'\} \in \binom{A}{2}$;
    \item  $|\phi(b) \cap \phi(x)|=\frac{4t}{25} $ for every $b \in B$ and $x\in A\cup B \setminus\{b\}$.
    \end{enumerate}
\end{lemma}

\begin{proof}
    Let $\beta$ denote an element not in $A\cup B$.
    Let $\Omega \coloneqq [2]^A \times [5]^{B \cup \{\beta\}}$; we have $t\coloneqq |\Omega| = 2^{|A|}5^{|B|+1}$.
    We let $\phi\colon A\cup B \to \Omega$ be defined by
    \begin{align*}
        \phi(a) &\coloneqq \sst{w \in \Omega}{w(a)=1 \mbox{ and } w(\beta)\in [4]}& \mbox{for every $a\in A$, and}\\
        \phi(b) &\coloneqq \sst{w\in \Omega}{w(b)\in [2]} & \mbox{for every $b\in B$.}
    \end{align*}
    Then by construction we have
    \begin{align*}
        |\phi(a)| &= 2^{|A|-1}5^{|B|}\times 4 = \frac{2t}{5} & \mbox{for every } a\in A,\\
        |\phi(b)| &= 2^{|A|+1}5^{|B|-1} = \frac{2t}{5} & \mbox{for every }  b\in B,\\
        |\phi(a) \cap \phi(a')| &= 2^{|A|-2}5^{|B|}\times 4 = \frac{t}{5} & \mbox{for every pair } (a,a') \in\binom{A}{2},\\
        |\phi(b) \cap \phi(x)| &= \frac{2}{5}|\phi(x)|=\frac{4t}{25} & \mbox{for every }  b\in B \mbox{ and } x\in A\cup B\setminus\{b\},
    \end{align*}
    as desired.
\end{proof}

\begin{rk}
    When $B=\emptyset$, it is possible to obtain such a colouring $\phi$ with $t = O(|A|)$.
    To do this, we let $m>1$ be the smallest integer such that $4m-1 \ge |A|$ is a prime power. It has been established in \cite{Bre32} that, for every integer $n\ge 2$, the interval $[n,2n]$ contains at least one prime number $p$ of the form $4m-1$, so we infer that $4m-1 < 2|A|$.
    Then by \cite[Corollary 4.7]{Sti08} together with \cite[Theorem 4.5]{Sti08} there exists a $(4t-1, 2t-1, t-1)$-BIBD $\cS$. That is, $\cS$ consists of a ground set $S$ of $4t-1$ points and a set $B \subseteq \binom{S}{2t}$ of blocks, each of which contains $2t-1$ points, and such that each pair of points is contained in $t-1$ blocks. 
    By \cite[Theorem 1.8]{Sti08}, each point belongs to $2t-1$ blocks; and by \cite[Theorem 1.9]{Sti08}, $|B|=4m-1$.
    Let us add to $\cS$ a block that contains all points and denote $B'\coloneqq B \cup \{S\}$ the new set of blocks; now there are $4t$ blocks, each point belongs to $2t$ blocks, and each pair of points is contained in $t$ blocks.
    Let $f\colon A \to S$ be an injection from $A$ to the set of points of $\cS$. 
    We let $\Omega \coloneqq B' \times [5]$ and $\phi\colon A \to \Omega$ be defined by $\phi(a) \coloneqq \sst{(b,x) \in B'\times [5]}{f(a) \in b \et x\le 4}$. Then $\phi$ satisfies the conclusion of \Cref{lem:intersecting-family} with $t=20m < 10|A|$.
\end{rk}

\section{Open problems}
We conclude with a list of open problems that arise naturally from our work.

\begin{problem}
    Given an integer $n\ge 5$, what is the smallest $q(n)$ such that every $n$-vertex graph $G$ with $\fdom(G) \ge 5/2$ has a dominating $(5q(n):2q(n))$-colouring? 
\end{problem}

We can infer from the proof of \Cref{prop:pqnotenough} that $q(n) = \Omega(\log n)$, and we have claimed in the previous subsection that $q(n) = \exp(O(n))$. One can show that a fractional dominating colouring of the graph obtained from $K_n$ by subdividing each edge into a hammock needs $\Omega(n)$ colours, and so $q(n)=\Omega(\sqrt{n})$. 

\begin{prop}
    Let $n\ge 2$, and let $G$ be the graph obtained from $K_n$ by subdividing each edge into a hammock. Then $\fdom(G)=5/2$, and if $\phi$ is a dominating $(5k:2k)$-colouring of $G$, for some integer $k\ge 1$, then $k \ge n/5$.
\end{prop} 

\begin{proof}
    Since $G$ contains a hammock, by \Cref{prop:5/$2$-hammock} one has $\fdom(G)\le 5/2$. Since $G$ is connected, has minimum degree $2$, and $G\notin \B$, by \Cref{thm:5/2} one has $\fdom(G) \ge 5/2$.
    We conclude that $\fdom(G)=5/2$.
    
    Let $\phi$ be a dominating $(5k:2k)$-colouring of $G$. 
    Let $u,v$ be a pair of vertices from the graph $K_n$ before subdivision. 
    Let $u \link x \link y \link v$ and $u \link z \link v$ be the two suspended paths of the hammock between $u$ and $v$ in $G$.
    We have $\phi(x)\cup \phi(y) \cup \big(\phi(u) \cap \phi(v)\big) = [5k]$, so $|\phi(u)\cap \phi(v)| \ge 5k - |\phi(x)| - |\phi(y)| = k$. 
    We have $\phi(z) \cup \big(\phi(u) \cup \phi(v)\big) = [5k]$, so $|\phi(u)\cap \phi(v)| \le 4k - |\phi(u)\cup \phi(v)| \le 4k - \big( 5k - |\phi(z)|\big) = k$.
    We conclude that $|\phi(u)\cap \phi(v)|=k$.
    
    Let $B_1, \ldots, B_{5k}$ be the colour classes of $\phi$; we claim that we must have $5k \ge n$. To see this, let $M$ be the $n\times 5k$ incidence matrix of the vertices and colours of $\phi$. Then $MM^T = (b_{i,j})_{i,j \in [n]}$ where $b_{i,i}=2k$ and $b_{i,j}=k$ if $i\neq j$. So $n = rank(B) \le rank(M) \le 5k$.
\end{proof}

\Cref{thm:5/2} indirectly provides a complete characterisation of graphs $G$ such that $\fdom(G)<5/2$: those are all graphs with minimum degree $\delta(G)<2$, or which contain some $B\in \B$ as an isolated subgraph. In particular, one can observe that $\fdom(G) \in \{1,2,7/3\}$ in that case.
In \cite[Section 3]{McSh89}, the authors characterise the class $\Gamma$ of all graphs $G$ such that $|V(G)|/\gamma(G)=5/2$. Extending this result to the fractional domatic number would be of interest. 

\begin{problem}
    Characterise the class $\mathscr{F}$ of all graphs $G$ such that $\fdom(G) = 5/2$. 
\end{problem}

We note that these classes are incomparable. For instance, $K_{2,3} \in \Gamma \setminus \mathscr{F}$ and $K_{2,4} \in \mathscr{F}\setminus \Gamma$. 

\begin{problem}
    Provide a full description of the set of possible values for $\fdom(G)$.
\end{problem}

We note that the graphs $G_n$ constructed in the proof of \Cref{prop:bipartite-girth6} have $\fdom(G_n) = \frac{5n-2}{2n-1}$, so the aforementioned set contains infinitely many values in the interval $[5/2, 5/2+\eps]$ for every $\eps>0$. 
On the other hand, there is no graph which has fractional domatic number in the range $(1,2) \cup (2,7/3) \cup (7/3, 5/2)$.
In contrast, the possible range of values for the fractional chromatic number $\chi_f$ is known to be $\big(\{1\}\cup [2,\infty)\big)\cap \mathbb{Q}$. Indeed, for every rational number $a/b \in \mathbb{Q}$ with $a/b \ge 2$, one has $\chi_f(KG(a,b))=a/b$, where $KG(a,b)$ is the Kneser Graph on the $b$-subsets of $[a]$. Moreover $\chi_f(G)=1$ if $G$ has no edge, and $\chi_f(G) \ge \omega(G) \ge 2$ if $G$ has at least one edge.

In what follows, given a class of graphs $\sG$, we denote $\fdom(\sG) \coloneqq \inf \{\fdom(G) : G\in \sG\}$ and $\dom(\sG) \coloneqq \min \{\dom(G) : G\in \sG\}$.

We have proved \Cref{cor:bounded-expansion} by relying on the existence of long suspended paths in graphs of large girth with sub-exponential bounded expansion and minimum degree $2$. This result cannot be extended to the class of all graphs of minimum degree $2$ and girth at least $g$, as for every $\eps > 0$, we can construct a bipartite graph $H$ of minimum degree $2$ and girth $g$ with $\fdom(H) \leq 5/2 + \eps$ (see \Cref{prop:5/2-girth}).

For every integer $g \ge 3$, we denote $\plangirth{g}$ the class of planar graphs of  minimum degree $2$ and girth at least $g$.
\begin{problem}
    For every integer $g\ge 3$, compute the extremal value 
    \[ \fdom(\plangirth{g}).\]
\end{problem}

We have already established that this extremal value is $2$ if $g\le 4$, $7/3$ if $5\le g \le 7$, and $5/2$ if $8 \le g \le 10$. We suspect that this extremal value is attained when $G$ is a cycle, which yields the following conjecture.

\begin{conj}
    For every integer $k \ge 1$, if $3k-1\le g \le 3k+1$, then
    \[ \fdom(\plangirth{g}) = \frac{3k+1}{k+1}.\]
\end{conj}

Given a graph class $\sG$ and an integer $d\ge 1$, we denote $\sG[\delta\ge d]$ the family of graphs in $\sG$ of minimum degree at least $d$. 
When $\sG[\delta\ge d] \neq \emptyset$ for every $d\ge 1$, we say that $\sG$ is \emph{$\dom$-bounded} if \[ \dom(\sG[\delta \ge d]) \underset{d\to\infty}{\to} \infty, \]
and we say that $\sG$ is \emph{linearly $\dom$-bounded} if \[ \dom(\sG[\delta \ge d]) =  \Theta(d).\]
A similar definition can be made for a \emph{(linearly) $\fdom$-bounded class}. We note that, by \Cref{prop:asymptotic-bound}, the class $\sG$ of all graphs is $\fdom$-bounded, since we have
\[ \fdom(\sG[\delta\ge d]) \underset{d\to\infty}{\sim} \frac{d}{\ln d}. \]

\begin{problem}
    Give some explicit non-trivial classes of graphs that are (linearly) $\dom$/$\fdom$-bounded. 
\end{problem}

We observe that if $\sG$ contains the class of bipartite graphs, then $\sG$ is neither $\dom$-bounded (because of \Cref{prop:pqnotenough} with $q_0\coloneqq 1$) nor linearly $\fdom$-bounded (because of \Cref{prop:tight-fdom}). The same holds if $\sG$ contains the class of split graphs (because of Remarks~\ref{rk} and~\ref{rk2}).

\bibliography{references}
\bibliographystyle{plainurl}

\clearpage
\appendix
\section{NP-hardness of fractional dominating colouring}\label{np-hardness}
\label{appendix:np-hardness}

It has been established in \cite{Kap94} that computing the domatic number of a graph is NP-hard. This follows from a reduction from \textsc{$3$-Colour} through the following construction.

\begin{defi}
    Given a graph $G$, we let $\cS(G)$ be the graph consisting of the edge-union of the subdivision $G^{1/2}$ of $G$, and a complete graph on $V(G)$. So $\cS(G)$ has $V(G)\cup E(G)$ as a vertex set, where for every $uv\in E(G)$, $N_{\cS(G)}(uv)=\{u,v\}$, and $uv\in E(\cS(G))$ for every $u,v\in V(G)$.
\end{defi}

Observe that $(V(G),E(G))$ is a partition of $V(\cS(G))$ into a clique and an independent set. Therefore $\cS(G)$ is a split graph, and in particular also a chordal graph, for any graph $G$. 

\begin{prop}[Kaplan, Ron Shamir, 1994]
\label{prop:reduction-dom}
    For every graph $G$ of minimum degree at least $1$,
    \[ \chi(G) \le 3 \iff \dom(\cS(G)) \ge 3. \]
\end{prop}

Observe that, by construction, $\delta(\cS(G))=2$. So Proposition~\ref{prop:reduction-dom} translates to: $G$ is $3$-colourable if and only if $\cS(G)$ is domatically full. 

We show that the fractional relaxation of Proposition~\ref{prop:reduction-dom} holds as well, which proves that, given a graph $G$, deciding whether $\fdom(G)\ge 3$ is NP-hard. Indeed, deciding whether $\chi_f(G)\le 3$ is NP-hard (see e.g. \cite[Theorem 3.9.2]{ScUl13}).

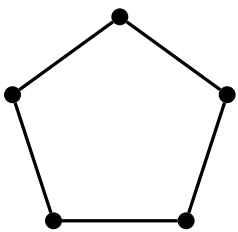
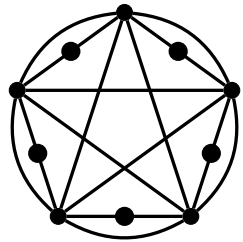
\begin{figure}[!htbp]
    \centering
     \begin{subfigure}[t]{0.45\textwidth}
    \centering
        \begin{tikzpicture}[every node/.style={shape=circle,draw=black,fill,inner sep=0pt,minimum size=6pt}, every edge/.style={line width=1.25pt,draw,black}]
        \graph[clockwise, radius=1.5cm] {subgraph C_n [V={a, b, c, d, e},empty nodes,name=A]  };

        \end{tikzpicture}
        \caption{$C_5$}
        \label{fig:ourcontribution:fractionalcolouringdomatic:C5}
    \end{subfigure}
    \hfill
    \begin{subfigure}[t]{0.45\textwidth}
    \centering
        \begin{tikzpicture}[every node/.style={shape=circle,draw=black,fill,inner sep=0pt,minimum size=6pt}, every edge/.style={line width=1.25pt,draw,black}]
        \graph[clockwise, radius=1.5cm] {
        subgraph C_n [edge node={node [ midway] {}},V={a, b, c, d, e},empty nodes,name=A];
        };
        \draw[] (A a) edge [bend left] (A b);
        \draw[] (A b) edge [bend left] (A c);
        \draw[] (A c) edge [bend left] (A d);
        \draw[] (A d) edge [bend left] (A e);
        \draw[] (A e) edge [bend left] (A a);

        \draw[] (A a) edge  (A d);
        \draw[] (A a) edge  (A c);

        \draw[] (A b) edge  (A d);
        \draw[] (A b) edge  (A e);

        \draw[] (A e) edge  (A c);
        
        \end{tikzpicture}
        \caption{$\mathcal{S}(C_5)$.}
        \label{fig:ourcontribution:fractionalcolouringdomatic:1subdivision}
    \end{subfigure}
    \hfill
  
    \caption{$C_5$ and $\mathcal{S}(C_5)$.}
    \label{fig:ourcontribution:S(C5)}
\end{figure}

\begin{prop}
\label{prop:reduction-fdom}
     For every graph $G$ of minimum degree at least $1$,
    \[ \chi_f(G) \le 3 \iff \fdom(\cS(G)) \ge 3. \]
\end{prop}

\begin{proof}
    If $\fdom(\cS(G)) \ge 3$, let $\phi$ be a dominating $(3q:q)$-colouring of $\cS(G)$. Then for every $uv \in E(G)$, we must have that $\phi(uv)$, $\phi(u)$, and $\phi(v)$ are disjoint. In particular, $\restrict{\phi}{V(G)}$ is a proper $(3q:q)$-colouring of $G$, hence $\chi_f(G) \le 3$.

    Conversely, if $\chi_f(G) \le 3$, let $\phi$ be a proper $(3q:q)$-colouring of $G$. We extend $\phi$ to $\cS(G)$ by letting $\phi(uv) \coloneqq [3q]\setminus (\phi(u)\cup \phi(v))$, for every $uv \in E(G)$, and we claim that $\phi$ is now a dominating $(3q:q)$-colouring of $\cS(G)$.
    Let $uv \in E(G)$. Since $\phi(u)$ and $\phi(v)$ are disjoint, we infer that $|\phi(uv)|=q$. Furthermore, $\phi(uv)$ is disjoint from $\phi(u)$ and $\phi(v)$, by construction, therefore all $3q$ colours appear in $N_{\cS(G)}[uv]$, and also in $N_{\cS(G)}[v]$. Since $G$ has minimum degree $1$, all vertices are incident to some edge, therefore $N_{\cS(G)}[v]$ contains all $3q$ colours, for every $v \in V(G)$.
    Thus, we obtain a dominating $(3q:q)$-colouring of $\cS(G)$.
\end{proof}

The \emph{join} of two graphs $G$ and $H$, denoted $G+H$,
is the graph obtained from the disjoint union of $G$ and $H$ by adding all edges between $V(G)$ and $V(H)$. 

\begin{lemma}\label{lem:join-Kt}
    Let $G$ be a graph, and $t\ge 1$ an integer. Then $\fdom(G + K_t) = \fdom(G) + t$
\end{lemma}

\begin{proof}
    Let $\phi$ be any dominating $(a:b)$-colouring of $G$. We easily extend $\phi$ to $G+K_t$ by using $bt$ additional colours, resulting in a dominating $(a+bt:b)$-colouring of $G$. Thus, $\fdom(G+K_t) \ge \fdom(G) + t$. Conversely, let $\phi$ be a dominating $(a:b)$-colouring of $G+K_t$ such that $a/b = \fdom(G+K_t)$; by the previous point, we have $a/b \ge \fdom(G) + t$, thus $a \ge bt$. Let us ignore the colours of $\phi(V(K_t))$, of which there are at most $bt$: the remaining $a-bt$ colours are dominating sets of $G$, so we obtain a dominating $(a-bt:b)$-colouring of $G$. Thus, $\fdom(G) \ge \fdom(G+K_t) - t$. With both inequalities proven, we have shown that $\fdom(G+K_t) = \fdom(G) + t$.
\end{proof}

Observe that if $G$ is a split graph, then $G + K_t$ is also a split graph. Using \Cref{prop:reduction-fdom,lem:join-Kt} together with the NP-hardness of deciding $\chi_f \ge 3$, we deduce the following.

\begin{cor}
    \label{cor:NP-hard}
    Let $k \ge 3$ be a fixed integer. Given a split graph $G$, it is NP-hard to decide whether $\fdom(G) \ge k$.
\end{cor}

\begin{problem}
    Does there exist a sharp complexity threshold for the \textsc{Fractional Dominating Colouring Problem}, i.e. is there some $x\in \mathbb{R}$ such that, for every $\eps>0$, given any graph $G$, one can decide in polynomial time whether $\fdom(G)\ge x-\eps$, while it is NP-hard to decide whether $\fdom(G) \ge x+\eps$?
\end{problem}

One can infer from \Cref{thm:5/2} and from \Cref{cor:NP-hard} that $5/2 \le x \le 3$, if such a threshold $x$ exists.
We note that the \textsc{Fractional Proper Colouring Problem} has such a threshold, whose value is $2$: for all $\eps>0$, given a graph $G$, it is NP-hard to decide whether $\chi_f(G)\le 2+\eps$ \cite[Theorem 3.9.2]{ScUl13}, while $\chi_f(G)\le 2$ if and only if $G$ contains no odd cycle.

\section{Domatically full regular graphs}

\label{appendix:counterexample}

We recall that a given graph $G$ is \emph{domatically full} if $\dom(G)=\delta(G)+1$, and we say that it is \emph{fdom-full} if $\fdom(G)=\delta(G)+1$. 
We begin with a simple observation concerning regular graphs.

\begin{prop}
\label{prop:dom-full}
    Let $G$ be a $d$-regular graph. Then $G$ is domatically full if and only if $\chi(G^2)=d+1$. It is fdom-full if and only if $\chi_f(G^2)=d+1$.
\end{prop}

\begin{proof}
    Let $\phi$ be a dominating $(p:q)$-colouring of $G$ with $p = q(d+1)$. If there is a vertex $v\in V(G)$ and two vertices $u,u' \in N[v]$ such that $N(u)\cap N(u') \neq \emptyset$, then $|\phi(N[v])|\le (d+1)q-1 < p$.
    Conversely, if $|\phi(N[v])| < p =q(d+1)$ for some integer $v$, then by the Pigeonhole Principle there are two vertices $u,u'\in N[v]$ with $\phi(u)\cap \phi(v) \neq \emptyset$. 
    We conclude that a $(p:q)$-colouring of $G$ with $p=q(d+1)$ is dominating if and only if it is a proper colouring of $G^2$.
    By fixing $q=1$, this implies that $\dom(G)\ge d+1 \iff \chi(G^2)\le d+1$, and by considering all possible values of $q$, this implies that $\fdom(G)\ge d+1 \iff \chi_f(G^2)\le d+1$.
    The conclusion follows since $\dom(G)\le \fdom(G)\le d+1$ by \Cref{prop:d+1} and $\chi(G^2) \ge \chi_f(G^2)\ge \omega(G^2) \ge d+1$.
\end{proof}

Given two graph $G$ and $H$, an \emph{$H$-cover of $G$} is a mapping $f\colon V(G)\to V(H)$ such that $f$ induces a bijection from $N_G(v)$ to $N_H(f(v))$ for every vertex $v\in V(G)$. 
Observe that, if $H$ is $d$-regular and there exists an $H$-cover of $G$, then $G$ must also be $d$-regular.
Note that a proper $(d+1)$-colouring of $G^2$ when $G$ is a $d$-regular graph corresponds exactly to a $K_{d+1}$-cover of $G$. Deciding whether a $K_{d+1}$-cover of $G$ exists is NP-hard (see \cite[Theorem~8]{KPT98}), so deciding whether a $d$-regular graph $G$ is domatically full is NP-hard.

In \cite[Proposition~3.1]{FKY00}, the authors claim that, for every $d$-regular graph $G$, 
 
  \begin{equation}
      \label{eq:dom-full}
      \dom(G) = d+1 \iff \fdom(G) = d+1.
  \end{equation}
This would directly imply that determining whether a $d$-regular graph is fdom-full is NP-hard for every integer $d\ge 3$.
However, that previous statement (which also appears in \cite[Theorem 4]{GoHe20}) turns out to be false already for the Coxeter graph, depicted in \Cref{coxeter-dominating-sets}, which is cubic, has fractional domatic $4$, and domatic number $3$. The odd-clique blow-ups of the Coxeter graph provides an infinite family of counterexamples.

\begin{prop}
    For infinitely many values of $d$, there exists a $d$-regular graph $G$ such that $\fdom(G) = d + 1$ and $\dom(G) \le d$.
\end{prop}
\begin{proof}
    Let $X$ be the Coxeter graph depicted in \Cref{coxeter-dominating-sets}. $X$ is vertex-transitive, and the following facts concerning its minimum dominating sets can be verified by hand in \Cref{coxeter-dominating-sets}:
    \begin{itemize}
        \item $X$ has exactly $8$ minimum dominating sets $D_1,\dots,D_8$.
        \item Each vertex of $X$ belongs to exactly two minimum dominating sets.
        \item Every pair $(D_i,D_j)$ of minimum dominating sets intersects in exactly one vertex.
    \end{itemize}
    Let $t \ge 1$ be an odd integer, and consider the blow-up $X[K_t]$, where we replace each vertex of $X$ by a copy of $K_t$; this graph has $n=28 t$ vertices, is $d$-regular with $d = 4t-1$, is vertex-transitive, and has domination number equal to that of $X$, i.e. $7 = \frac{n}{4t}$, which immediately implies that $\fdom(X[K_t]) = \frac{n}{7} = 4t = d+1$, by \Cref{prop:n/gamma}. Suppose $X[K_t]$ has a domatic partition into $4t$ parts, each of size $7$ and thus projecting to some minimum dominating set in $X$; for $i \in [8]$, let $a_i$ count the number of parts projecting to $D_i$. For $1\le i<j\le 8$, $D_i \cap D_j$ contains a unique vertex $v_{i,j}$; the vertices in the $K_t$-blowup of $v_{i,j}$ therefore belong to parts that project to either $D_i$ or $D_j$, and thus $a_i + a_j = t$. These equations force $a_1 = a_2 = \dots = a_8 = \frac{t}{2}$, contradicting the oddness of $t$.
\end{proof}

\begin{rk}
    The authors observed that the Coxeter graph provided a single counterexample; the structure of its minimum dominating sets was then observed using ChatGPT, from which the odd-clique blow-up construction follows. No other part of this part of this paper used AI assistance.
\end{rk}

\begin{figure}[t]
	\centering
	\begin{tikzpicture}[every node/.style={shape=circle,draw=black,fill,inner sep=0pt,minimum size=6pt}, every edge/.style={line width=1.25pt,draw,black}]

		\graph [nodes={ circle}, n=24, clockwise,radius=5cm, empty nodes, name=A]
		{ 1; 2; 3; 4; 5;6;7;8;9;10;11;12;13;14;15;16;17;18;19;20;21;22;23;24; subgraph C_n};
		
		\node (25) at ($(A 1)!0.5!(A 13)$) {};
		\node (26) at ($(25)!0.5!(A 23)$) {};
		\node (27) at ($(25)!0.5!(A 7)$) {};
		\node (28) at ($(25)!0.5!(A 15)$) {};
		
		\draw[] (25) edge[] (A 1);
		\draw[] (25) edge[] (A 9);
		\draw[] (25) edge[] (A 17);
		
		\draw[] (26) edge[] (A 23);
		\draw[] (27) edge[] (A 7);
		\draw[] (28) edge[] (A 15);
		
		\draw[] (26) edge[] (A 16);
		\draw[] (26) edge[] (A 5);
		
		\draw[] (27) edge[] (A 24);
		\draw[] (27) edge[] (A 13);
		
		\draw[] (28) edge[] (A 21);
		\draw[] (28) edge[] (A 8);
		
		\draw[] (A 2) edge[] (A 20);
		\draw[] (A 4) edge[] (A 10);
		\draw[] (A 12) edge[] (A 18);
		
		\draw[] (A 3) edge[] (A 14);
		\draw[] (A 11) edge[] (A 22);
		\draw[] (A 6) edge[] (A 19);
		
		\begin{scope}[
			every node/.style={
				text depth=0pt,
				font=\small, text=black
			}
			]
			\foreach \v/\a/\b in {
				1/1/2,  2/3/4,  3/5/6,  4/1/7,
				5/3/8,  6/4/5,  7/2/7,  8/1/8,
				9/3/5, 10/2/4, 11/6/8, 12/5/7,
				13/1/4, 14/2/8, 15/3/7, 16/1/5,
				17/4/8, 18/2/3, 19/1/6, 20/7/8,
				21/2/5, 22/1/3, 23/4/7, 24/5/8}
			\path let \p1 = ($(A \v)-(25)$) in
			node at ($(25)+({1.12*\x1},{1.09*\y1})$)
			{$\a,\b$};
			
			\node[below=5pt]       at (25) {$6,7$};
			\node[above right=2pt]       at (26) {$2,6$};
			\node[above right=2pt] at (27) {$3,6$};
			\node[below right=2pt] at (28) {$4,6$};
		\end{scope}
	\end{tikzpicture}
	
	\caption{The Coxeter graph and its $8$ minimum dominating sets (of size $7$).}\label{coxeter-dominating-sets}
\end{figure}
\end{document}